\documentclass{article}
\usepackage{fullpage}
\usepackage{subfig}
\usepackage{mathrsfs}

\usepackage{amsmath,amssymb,amsthm}
\usepackage{booktabs}
\usepackage{subfig}
\usepackage{graphicx}
\usepackage{comment}
\usepackage{parskip}
\usepackage[colorlinks]{hyperref}

\usepackage[sort,capitalise]{cleveref}
\usepackage{color,xcolor}

\newcommand{\trace}[1]{\mathrm{tr}\left(#1\right)}

\usepackage{bm,bbm}
\usepackage[numbers,compress,sort]{natbib}
\usepackage{url}
\renewcommand{\tilde}{\widetilde}

\newcommand{\E}{\mathbb{E}}

\usepackage[plain,noend]{algorithm2e}

\makeatletter
\renewcommand{\algocf@captiontext}[2]{#1\algocf@typo. \AlCapFnt{}#2} 
\def\@algocf@capt@plain{top}
\renewcommand{\algocf@makecaption}[2]{%
  \addtolength{\hsize}{\algomargin}%
  \sbox\@tempboxa{\algocf@captiontext{#1}{#2}}%
  \ifdim\wd\@tempboxa >\hsize
    \hskip .5\algomargin%
    \parbox[t]{\hsize}{\algocf@captiontext{#1}{#2}}
  \else%
    \global\@minipagefalse%
    \hbox to\hsize{\box\@tempboxa}
  \fi%
  \addtolength{\hsize}{-\algomargin}%
}
\makeatother

\newcommand{\vertiii}[1]{{\vert\kern-0.25ex\vert\kern-0.25ex\vert #1 
    \vert\kern-0.25ex\vert\kern-0.25ex\vert}}

\DeclareMathOperator*{\argmax}{arg\,max}
\DeclareMathOperator*{\argmin}{arg\,min}
\newtheorem{theorem}{Theorem}
\newtheorem{lemma}{Lemma}
\newtheorem{corollary}{Corollary}
\newtheorem{proposition}{Proposition}
\theoremstyle{definition}
\newtheorem{definition}{Definition}

\newtheorem{remark}{Remark}
\AtBeginEnvironment{remark}{%
  \pushQED{\qed}%
}
\AtEndEnvironment{remark}{\popQED\endremark}

\numberwithin{equation}{section}

\usepackage{mathtools}
\newcommand{\pr}{\mathbb{P}}
\renewcommand{\hat}[1]{\widehat{#1}}
\renewcommand{\tilde}[1]{\widetilde{#1}}

\begin{document}

    \title{Eigenvector fluctuations and limit results for random graphs with infinite rank kernels}

    \author{Minh Tang \\
      {Department of Statistics, North Carolina State University} \and
      Joshua Cape \\
      {Department of Statistics, University of Wisconsin-Madison}
    }
    \maketitle

\begin{abstract}
    This paper systematically studies the behavior of the leading eigenvectors for independent edge undirected random graphs generated from a general latent position model whose link function is possibly infinite rank and also possibly indefinite. We first derive uniform error bounds in the two-to-infinity norm as well as row-wise normal approximations for the leading sample eigenvectors. We then build on these results to tackle two graph inference problems, namely (i) entrywise bounds for graphon estimation and (ii) testing for the equality of latent positions, the latter of which is achieved by proposing a rank-adaptive test statistic that converges in distribution to a weighted sum of independent chi-square random variables under the null hypothesis. Our fine-grained theoretical guarantees and applications differ from the existing literature which primarily considers first order upper bounds and more restrictive low rank or positive semidefinite model assumptions. Further, our results collectively quantify the statistical properties of eigenvector-based spectral embeddings with growing dimensionality for large graphs.
\end{abstract}

\section{Introduction}
Statistical network analysis, encompassing the study of statistical methodology, theory, and applications for networks or graph data, has witnessed a surge of attention and progress in recent decades. Driven in large part by developments in data collection capabilities and pressing scientific questions, the overarching objective of this domain is to develop the statistical foundations of network or graph data analysis. A principal challenge here is that such data are complex, inter-related, structured, and hence fall outside the scope of traditional data settings in classical Statistics.

Historically, considerable attention has been devoted to exploratory graph analysis, network modeling, and problems of estimation, areas which continue to witness ongoing developments. More recently, building on these achievements, there has been increasing interest in developing genuine inference procedures that quantify uncertainty in order to address network hypothesis testing problems for popular random graph models.

This paper undertakes a detailed study of spectral methods and their statistical inference capabilities for analyzing large, independent edge, undirected, inhomogeneous random graphs generated from a general latent position network model. Given a symmetric adjacency matrix $A \in \{0,1\}^{n \times n}$ generated from an underlying symmetric edge probability matrix $P \in [0,1]^{n \times n}$ determined by a kernel $\kappa$ (of possibly infinite rank), we study $r$-dimensional spectral embeddings of the form $\hat{U} |\hat{\Lambda}|^{\alpha}$, where $\alpha \in \{0, 1/2, 1\}$, where $\hat{\Lambda}$ is the $r \times r$ diagonal matrix of largest-in-magnitude eigenvalues of $A$, and where $\hat{U}$ is an $n \times r$ matrix containing corresponding orthonormal eigenvectors of $A$. We show that $\hat{U} |\hat{\Lambda}|^{\alpha}$ is entrywise and row-wise close (modulo a necessary orthogonal transformation $W$) to the corresponding population-level spectral embedding $U |\Lambda|^{\alpha}$, determined by $P$, in a wide variety of settings (see \cref{thm:psd,cor:main1,thm:general,thm:generalU_ULambda}). For $\alpha = 1/2$, we establish row-wise multivariate asymptotic normality in \cref{cor:normal}, with similar results holding for $\alpha = 0, 1$. Our results are enabled by new deterministic matrix perturbation analysis (see \cref{thm:deterministic}) and careful treatment of concentration inequalities including novel leave-one-out analysis (see \cref{sec:proofs}).

Numerous estimation and inference problems for both single and multiple networks can be formulated and tackled via the spectral embeddings given by $\alpha \in \{0, 1/2, 1\}$. Examples include community detection via spectral clustering, vertex nomination, two-sample hypothesis testing, and more; see \cref{sec:inference} for additional discussion and references. In this paper,  \cref{sec:entrywise_approximation} demonstrates how our results give entrywise bounds on both edge probability matrix estimation and graphon estimation. \cref{sec:two_sample} considers the problem of testing for the equality of latent positions, for which we propose a flexible plug-in test statistic that adapts to the embedding dimension (rank truncation value) $r$, even when $r \equiv r(n) \to \infty$ as $n \to \infty$ (see \cref{thm:quadratic_test,cor:test}). \cref{sec:numerical} provides simulation examples illustrating the finite-sample properties of our proposed test procedure. 

Our results are exceedingly general and hence broadly applicable, in the following aspects.
\begin{enumerate}
    \item We obtain results both for low (i.e., finite) rank models and for infinite rank models.
    \item We obtain results both for models with positive semidefinite kernels and for more general (i.e., possibly indefinite) kernels.
    \item We obtain results without requiring unnecessary or purely mathematically convenient assumptions such as population-level bounded coherence (i.e., eigenvector delocalization).
    \item We obtain results both for settings with and without repeated population eigenvalues.
    \item Our results hold for eigenvector-based embeddings with fixed or growing dimensionality for large networks.
\end{enumerate}
By addressing both the deviations and fluctuations of spectral embeddings derived from the adjacency-based (scaled) sample eigenvectors, this paper simultaneously recovers and improves upon numerous existing results in the literature in a unifying fashion \citep{abbe2020entrywise,athreya2013limit,cape2019two,chatterjee,du_tang,fan2019simple,lei2019unified,graph_root,xu_spectral}.

This paper is situated in the concrete setting of random graphs and demonstrates a wide range of results that can be obtained therein. At the same time, we emphasize that the considerations in this paper can, if desired, be extended beyond networks to signal-plus-noise type models more generally and other data settings, following for example the blueprint in \cite{ctp_biometrika} and elsewhere \citep{xie2021entrywise}. While conceptually straightforward in principle (see \cref{thm:deterministic}), doing so would require the interested researcher to adapt and customize current technical lemmas and proof arguments (see \cref{sec:technical_lemmas}) to their specified data setting of interest.

\subsection{Notations}
\label{sec:notation}
We summarize some notations frequently used in this paper. For a positive integer $p$, let $[p]$ denote the set $\{1,\dots,p\}$. For two non-negative sequences $\{a_{n}\}_{n \ge 1}$ and $\{b_{n}\}_{n \ge 1}$, we write $a_{n} \lesssim b_{n}$ ($a_{n} \gtrsim b_{n}$, resp.) if there exists some constant $C > 0$ such that $a_{n} \leq C b_{n}$ ($a_{n} \geq C b_{n}$, resp.) for all $n \geq 1$. We write $a_{n} \asymp b_{n}$ if simultaneously $a_{n} \lesssim b_{n}$ and $a_{n} \gtrsim b_{n}$. If $a_{n} / b_{n}$ stays bounded away from $+\infty$, then we write $a_{n} = O(b_{n})$ and $b_{n} = \Omega(a_{n})$. We write $a_{n} = \Theta(b_{n})$ to indicate that $a_{n} = O(b_{n})$ and $a_{n} = \Omega(b_{n})$. If $a_{n} / b_{n} \to 0$, then we write $a_{n} = o(b_{n})$ and $b_{n} = \omega(a_{n})$.

We say that a sequence of events $\{\mathcal{A}_{n}\}_{n \ge 1}$ holds with high probability if for any constant $c > 0$ there exists a finite constant $n_{0}$ depending only on $c$ such that $\mathbb{P}(\mathcal{A}_{n}) \geq 1 - n^{-c}$ for all $n \geq n_{0}$. 

In this paper, all vectors and matrices are real-valued. We let $\mathcal{O}_{d}$ denote the set of $d \times d$ orthogonal matrices. Given a matrix $M$, we denote its spectral norm by $\|M\|$, its Frobenius norm by $\|M\|_{F}$, its maximum absolute row sum norm by $\|M\|_{\infty}$, its nuclear norm by $\|M\|_{\ast}$, and its maximum absolute entrywise norm by $\|M\|_{\max}$.

We denote the two-to-infinity norm ($2 \to \infty$ norm) of the matrix $M$ by
\begin{equation*}
    \|M\|_{2 \to \infty}
    =
    \max_{\|x\| = 1} \|Mx\|_{\infty}
    \equiv
    \max_{i} \|m_{i}\|
    ,
\end{equation*}
where $\|x\|$ denotes the Euclidean norm of the vector $x$, and $m_{i}$ denotes the $i$-th row vector of $M$. In particular, $\|M\|_{2 \to \infty}$ is the maximum row-wise $\ell_{2}$ norm of $M$. We emphasize that the $2 \to \infty$ norm is not sub-multiplicative in general, though, for any matrices $M$ and $N$ of conformable dimensions, it holds that (e.g., see \cite[Proposition~6.5]{cape2019two})
\begin{equation}
    \label{eq:2toinf_submultiplicative}
    \|MN \|_{2 \to \infty}
    \leq
    \min\left\{
    \|M\|_{2 \to \infty} \times \|N\|,
    \|M\|_{\infty} \times \|N\|_{2 \to \infty}
    \right\}
    .
\end{equation}
Perturbation bounds using the $2 \to \infty$ norm for the eigenvectors or singular vectors of a noisily observed matrix have recently attracted widespread interest in the statistics community. For example, see \cite{spectral_chen,cape2019two,agterberg_2inf,fan2018eigenvector,abbe2020entrywise,inference_heteroskedastic,inference_pnas,cai_unbalanced,small_eigengaps} and the numerous references therein.

\section{Setup}
\label{sec:setup}
This paper considers latent position random graphs \cite{Hoff2002,bollobas2007phase} specified as follows.

\begin{definition}[Latent position graph]\label{def:lpg}
Given a positive integer $d \ge 1$, let $\mathcal{X} \subset \mathbb{R}^{d}$ be a nonempty compact set, and let $F$ be a probability distribution taking values in $\mathcal{X}$. Let $\kappa : \mathcal{X} \times \mathcal{X} \rightarrow [0,1]$ be a symmetric measurable function, namely $\kappa(x,x') = \kappa(x',x)$ for all $x, x' \in \mathcal{X}$. We say that $(A, X)$ is a {\em latent position graph} on $n$ vertices with distribution $F$, link function $\kappa$, and sparsity $\rho_{n}$ when $A$ is generated via the following procedure.
    \begin{enumerate}
      \item Sample i.i.d. latent positions $X_{1}, X_{2}, \dots, X_{n}$ according to $F$.
      
      \item Define $P$ as the $n \times n$ matrix with entries $P_{ij} = \rho_{n} \kappa(X_{i}, X_{j})$ for all $1 \le i, j \le n$.
      
      \item Given $P$, sample a symmetric binary matrix $A \in \{0,1\}^{n \times n}$ whose upper-triangular entries $A_{ij}$ are independent Bernoulli random variables with $\pr(A_{ij} = 1) = P_{ij}$ for $1 \le i \le j \le n$.
    \end{enumerate}
  For simplicity, we indicate this setting by writing $A \sim \mathrm{LPG}(\kappa, F; \rho_{n})$.
\end{definition}

\begin{remark}[Latent position graph versus graphon]
    \label{rem:graphon}
    If $\mathcal{X} = [0,1]$ and $F$ is associated with Lebesgue measure, then $\kappa$ in \cref{def:lpg} is widely known as a {\em graphon} \cite{lovasz12:_large}. While any latent position graph distribution is equivalent (with respect to the cut metric) to some graphon \cite[Proposition~10]{borgs} and vice versa, there are benefits to considering these models separately, due to the fact that additional structure is possible in higher dimensions. For example, suppose that $\mathcal{X} = [0,1]^{d}$ for $d \ge 2$ is equipped with Lebesgue measure and that $\kappa$ is $\alpha$-Hölder continuous. Then, any graphon $\tilde{\kappa}$ on $[0,1] \times [0,1]$ equivalent to $\kappa$ may only be $(\alpha/d)$-Hölder continuous \cite{janson_ohlede}.
\end{remark}

For ease of exposition, we shall first assume that $\kappa$ is a continuous, {\em positive semidefinite} kernel. Results for general (indefinite) kernels can be obtained similarly but are deferred to \cref{sec:indefinite} because they require substantially more involved notation and derivations.

Concretely, $\kappa$ being continuous and positive semidefinite implies the following useful facts.

\begin{proposition}[Properties of positive semidefinite kernels]
\label{prop:psd}
    Let $\kappa$ in \cref{def:lpg} be a continuous, positive semidefinite kernel, and let $\mathscr{K}$ denote the integral operator induced by $\kappa$, namely
    \begin{equation}
        \label{eq:integral_op}
        (\mathscr{K} h)(x)
        =
        \int_{\mathcal{X}} \kappa(x, y) h(y) \, \mathrm{d}F(y)
    \end{equation}
    for any square-integrable real-valued function $h \in L^{2}(\mathcal{X})$. Let $\mu_{1} \geq \mu_{2} \geq \cdots \geq 0$
    be the {\em eigenvalues} of $\mathscr{K}$, and let $\{\phi_{r}\}_{r \ge 1}$ be the corresponding {\em orthonormal eigenfunctions} satisfying
    \begin{equation*}
        \int_{\mathcal{X}} \phi_{r}(x) \phi_{s}(x) \mathrm{d} F(x)
        =
        \begin{cases}
            1 & \text{~if~~} r = s, \\
            0 & \text{~if~~} r \neq s,
        \end{cases}
    \end{equation*}
    for all integers $r, s \geq 1$.
\begin{enumerate}
  \item For any integer $r \geq 1$, let $U$ be the $n \times r$ matrix whose columns are orthonormal eigenvectors corresponding to the $r$ largest-in-magnitude eigenvalues of $P$. Let $\Lambda$ be the diagonal matrix containing the corresponding eigenvalues of $P$. Then, with probability one,
    \begin{equation}
      \label{eq:bounded_coherence}
        \|U |\Lambda|^{1/2}\|_{2 \to \infty}
        =
        \|U |\Lambda| U^{\top}\|_{\max}^{1/2}
        \leq
        \|P\|_{\max}^{1/2}
        \leq
        \rho_{n}^{1/2}.
    \end{equation}
    Note that the above bound does not depend on $r$. 
    Similarly, $\|U_{\perp} |\Lambda_{\perp}|^{1/2}\|_{2 \to \infty} \leq \rho_{n}^{1/2}$ holds with probability one. 
    
    \item By Mercer's theorem (e.g., see \cite[Theorem~4.49]{steinwart08:_suppor_vector_machin}) it holds that
    \begin{equation}
      \label{eq:mercer_rep}
        \kappa(x,y)
        =
        \sum_{r=1}^{\infty} \mu_{r} \phi_{r}(x) \phi_{r}(y)
        ,
    \end{equation}
    where the sum converges {\em absolutely} and {\em uniformly} on $\mathrm{supp}(F) \times \mathrm{supp}(F)$.
    
    \item Furthermore, $\mathscr{K}$ is a compact, positive operator and of {\em trace-class}, namely $\mathscr{K} = |\mathscr{K}|$, where $|\mathscr{K}| = \sqrt{\mathscr{K}^{\operatorname{H}}\mathscr{K}}$ denotes the positive semidefinite Hermitian square root, and
    \begin{equation}
        \label{eq:trace_class}
        \trace{|\mathscr{K}|}
        =
        \sum_{j=1}^{\infty} \mu_{j}
        =
        \int_{\mathcal{X}} \kappa(x,x) \mathrm{d}F(x)
        <
        \infty
        .
    \end{equation}
    The matrix $P$, properly normalized, is itself positive and of trace-class. Namely, with probability one,
    \begin{equation*}
        \frac{1}{n \rho_{n}} \,
        \trace{|P|}
        =
        \frac{1}{n \rho_{n}} \sum_{j=1}^{n} \lambda_{j}
        =
        \frac{1}{n \rho_{n}} \sum_{i=1}^{n} \rho_{n} \kappa(X_{i}, X_{i})
        \leq
        1
        .
    \end{equation*}
    
    \item The eigenvalues of $P$, written $\lambda_{1} \ge \dots \ge \lambda_{n}$, are {\em consistent} for those of $\mathscr{K}$. In particular, for any $c > 0$, by \cite[Theorem~7]{rosasco_1} it holds that
    \begin{equation}
        \label{eq:eigenvalue_convergence}
        \pr\left(
        \sup_{j \geq 1}\left|(n \rho_{n})^{-1} \lambda_{j} - \mu_{j}\right|
        \leq
        \frac{2 \sqrt{2c} \log^{1/2}{n}}{n^{1/2}}
        \right) 
        \geq
        1 - 2n^{-c}
        ,
    \end{equation}
    where, by definition, $\lambda_{j} = 0$ whenever $j > n$.
\end{enumerate}
\end{proposition} 
We note that \cref{eq:bounded_coherence} and \cref{eq:eigenvalue_convergence} also hold when $\kappa$ is discontinuous but positive semidefinite. The uniform convergence in \cref{eq:mercer_rep} and the trace formula in \cref{eq:trace_class}, however, do require continuity of $\kappa$. 
\begin{remark}[Population eigenvectors and coherence]
    \label{rem:bdd_coherence1}
    \cref{eq:bounded_coherence,eq:2toinf_submultiplicative} together imply that $\|U\|_{2 \to \infty} \leq \rho_{n}^{1/2} \lambda_{r}^{-1/2}$ holds with probability one. Define $\mathfrak{c}(U) = \tfrac{n}{r} \|U\|_{2 \to \infty}^2$ be the {\em coherence} of $U$ (see \cite[Definition~1.2]{candes2012exact}). Now, if we choose $r$ such that $\mu_{r} \geq 4 \sqrt{2c} n^{-1/2} \log^{1/2}{n}$ then $\lambda_{r} \geq \tfrac{1}{2} n \rho_{n} \mu_{r}$ holds with probability at least $1 - 2n^{-c}$ (see \cref{eq:eigenvalue_convergence}), and hence $\mathfrak{c}(U) \leq 2(r \mu_{r})^{-1}$ holds with probability at least $1 - 2n^{-c}$. If we fix $r$, then the previous bound indicates that $U$ has {\em bounded coherence} with high probability. However, in this paper, we are interested in letting $r$ grow with $n$, in which case $\mathfrak{c}(U)$ can diverge with $n$ since $\mu_{r}$ typically decays to zero at a faster rate than $r^{-1}$. See \cref{rem:polynomial_decay,rem:exponential_decay} for additional discussion and specifics.
\end{remark}

\section{Main results}
\label{sec:main}
We begin by presenting our first main result for (scaled) eigenvector estimation in latent position graphs with positive semidefinite kernels.
\begin{theorem}[Fine-grained eigenvector perturbation analysis for LPGs with positive semidefinite kernels]
  \label{thm:psd}
  Let $A \sim \mathrm{LPG}(\kappa, F; \rho_{n})$ be a graph on $n$ vertices generated according to \cref{def:lpg}, where $\kappa$ is positive semidefinite. For a given $r \geq 1$, let $U$ and $\hat{U}$ be the $n \times r$ matrices whose orthonormal columns are the leading eigenvectors of $A$ and $P$, respectively, and let the diagonal matrices $\hat{\Lambda}$ and $\Lambda$ contain the corresponding eigenvalues of $A$ and $P$. Fix $\nu > 0$ and suppose $r \geq 1$ is chosen such that the conditions
  \begin{gather}
    \label{eq:r_select_psd1}
    \lambda_{r} - \lambda_{r+1}
    \geq
    \max\left\{4 \varsigma(\nu,n), \tfrac{16}{3} (\nu + 2) \log n \right\}, \\
    \label{eq:r_select_psd2}
    \lambda_{r}
    \geq
    \max\left\{16 (\nu + 2) \log n + \frac{64 \varsigma(\nu, n)^{2}}{\lambda_{r} - \lambda_{r+1}},
    \vartheta(\nu+1,r,n)\right\},
   \end{gather}
   are both satisfied, where
   \begin{gather}
    \label{eq:varsigma_def}
    \varsigma(\nu,n)
    =
    2 \sqrt{2 e n \rho_{n}} + (56 \sqrt{e} + \sqrt{2\nu}) \sqrt{\log n}, \\ \label{eq:vartheta_def}
    \vartheta(c,r,n)
    = c \log n + r \log 9,
    \qquad
    \text{for $c > 0$}. 
   \end{gather}
   Let $\delta_{r} = \lambda_{r} - \lambda_{r+1}$. 
   Then, there exists an $r \times r$ orthogonal matrix $W^{(n)}$ such that
   \begin{equation}
     \label{eq:expansion_definite}
      \hat{U} \hat{\Lambda}^{1/2} W^{(n)} - U \Lambda^{1/2}
      =
      E U \Lambda^{-1/2}
      +
      Q,
    \end{equation}
    where $E U \Lambda^{-1/2}$ satisfies
    \begin{gather}
      \label{eq:EUlambda}
      \| E U \Lambda^{-1/2} \|_{2 \to \infty}
      \leq
      \frac{11}{2} \rho_{n}^{1/2} \lambda_{r}^{-1/2} \sqrt{\vartheta(\nu+1,r,n)}
    \end{gather}
    with probability at least $1 - n^{-\nu}$, while $Q$ satisfies
     \begin{gather}
        \label{eq:Q2inf_bd}
        \|Q\|_{2 \to \infty}
        \lesssim
        \frac{n \rho_{n}^{3/2}}{\delta_{r}^{2}}
        +
        \frac{(\rho_{n} \log n)^{1/2}}{\lambda_{r}^{1/2}} \left( \frac{(r n \rho_{n})^{1/2} + \log n}{\delta_{r}}\right)
    \end{gather}
    with probability at least $1 - O(n^{-\nu})$. Above, the notation $\lesssim$ and $O(\cdot)$ hides universal constants that depend only on $\nu$ but not on $r$, $n$, $\rho_{n}$, or $\kappa$.
\end{theorem}
As $n \rho_n \geq \lambda_1 \geq \lambda_r \geq \delta_r$, 
\cref{eq:r_select_psd1} through \cref{eq:vartheta_def} automatically imply a lower bound of $n \rho_n = \Omega(\log n)$ as typically seen in the literature on spectral inference for random graphs. 
Furthermore, it holds that $\|Q\|_{2 \to \infty} = o(\lambda_{r}^{-1/2} \rho_{n}^{1/2}(r^{1/2} + \log^{1/2}n))$ with high probability, so $E U \Lambda^{-1/2}$ is the dominant term in the expansion of \cref{eq:expansion_definite}, whenever the conditions 
\begin{gather}
    \label{eq:condition_rem4} 
    \delta_{r}^{4} (r + \log n)
    =
    \omega\left((n \rho_{n})^{2} \lambda_{r}\right)
    \qquad
    \text{and}
    \qquad
    \delta_{r}
    =
    \omega\left(\max\left\{ \sqrt{n \rho_{n} \log n}, \log n\right\}\right)
\end{gather}
are simultaneously satisfied.

\begin{remark}[Explicit constants and bounding the residual perturbation term]
    \label{rem:cor_presentation}
    Derivations of explicit but not-too-large universal constants in \cref{eq:r_select_psd1} through \cref{eq:vartheta_def} are crucial for finite-sample inference as they allow us to choose $r$, interpretable as the embedding dimension or dimensionality, as growing with $n$ in a data-dependent manner (see \cref{cor:main1} below) while {\em simultaneously} guaranteeing that the bound in \cref{eq:EUlambda} for the main order term holds for any finite $n$.
    Meanwhile, explicit albeit possibly sub-optimal values for the constants appearing in the bound for the residual term, $\|Q\|_{2 \to \infty}$, can be derived through careful but tedious book-keeping per the technical lemmas in \cref{sec:technical_lemmas}. For example, one can show
    \begin{equation*}
    \begin{split}
    \|Q\|_{2 \to \infty}
    &\leq
    \frac{74 \rho_{n}^{1/2} \varsigma(\nu,n)^2}{\delta_{r}^{2}}
    +
    \frac{181 \rho_{n} \sqrt{\vartheta(\nu+2,r,n)}}{\delta_{r}}
    +
    \frac{80 \varsigma(\nu,n) \sqrt{(\nu + 2) r \rho_{n} \log n}}{\delta_{r} \lambda_{r}^{1/2}} \\
    &\qquad+
    \frac{\sqrt{\rho_{n} \vartheta(\nu+2,r,n)}}{\lambda_{r}^{1/2}}\left(\frac{176 (\nu + 2) \log n + 357 \varsigma(\nu,n)}{\delta_{r}}\right)
    \end{split}
    \end{equation*}
    holds with probability at least $1 - 10n^{-\nu}$. See \cref{sec:proof_simple} for more details.
\end{remark}

\begin{remark}[On dimension selection and eigengaps in high rank models]
    \label{rem:dim_select}
    This paper is primarily concerned with the setting $\mathrm{rk}(\kappa) = \infty$, where $\mathrm{rk}(\kappa)$ denotes the number of non-zero eigenvalues associated with the integral operator $\mathscr{K}$ in \cref{eq:integral_op}. Notably, if $\mathrm{rk}(\kappa) = \infty$, then for any $n \ge 1$, the $n \times n$ matrix $P$ in \cref{def:lpg} is of rank $n$ with probability one. In this setting, it is often desirable to choose the embedding dimension, $r$, to change (i.e.,~grow) with $n$, such as in the inference problems considered in \cref{sec:entrywise_approximation,sec:two_sample}. Importantly, doing so is always possible when $A$ is generated according to the model in \cref{def:lpg} with $n \rho_n = \omega(\log n)$. More specifically, for any fixed but arbitrary $s \geq 1$, there always exists $r \geq s$ for which \cref{eq:r_select_psd1,eq:r_select_psd2} are satisfied as $n$ increases. Indeed, as $\mathscr{K}$ is compact, the only accumulation point of $\{\mu_{k}\}_{k \ge 1}$ is at $0$, hence if $\mu_{s} > 0$, then there always exists some $r \geq s$ such that $\mu_{r} - \mu_{r+1} > 0$. For this choice of $r$, by \cref{eq:eigenvalue_convergence}, we have for any $c > 0$ and sufficiently large $n$ that
\begin{equation*}
    \lambda_{r} - \lambda_{r+1}
    \geq
    n \rho_{n}(\mu_{r} - \mu_{r+1} - 4 \sqrt{2c} n^{-1/2} \log^{1/2}{n})
    \geq
    \tfrac{1}{2} n \rho_{n} (\mu_{r} - \mu_{r+1})
    \geq 
    4\varsigma(\nu,n)
\end{equation*}
with probability at least $1 - n^{-c}$, and \cref{eq:r_select_psd1} is satisfied. Given \cref{eq:r_select_psd1}, for any fixed $c$, $\nu$, and $r$, we also have
\begin{equation*}
    \begin{split}
    \lambda_{r}
    &\geq
    n \rho_{n} (\mu_{r} - 2 \sqrt{2c} n^{-1/2} \log^{1/2}{n}) \\
    &\geq
    \max\left\{16 (\nu + 2) \log n + \frac{64 \varsigma(\nu,n)^{2}}{\lambda_{r} - \lambda_{r+1}}, \vartheta(\nu+1,r,n)\right\}
    \end{split}
\end{equation*}
for sufficiently large $n$, and \cref{eq:r_select_psd2} is also satisfied. In summary, if $\mathrm{rk}(\kappa) = \infty$, then we can always select a sequence $\{r(n)\}_{n \ge 1}$, with $r(n) \rightarrow \infty$ as $n \rightarrow \infty$ for which \cref{thm:psd} applies. Furthermore, for any $r$ along this sequence, the bounds for $EU \Lambda^{-1/2}$ and $Q$ can be written in terms of $\{\mu_k\}_{k \ge 1}$ as
\begin{gather}
\label{eq:EULambda_pop1}
    \|E U \Lambda^{-1/2}\|_{2 \to \infty}
    \lesssim
    \frac{(r^{1/2} + \log^{1/2}{n})}{n^{1/2} \mu_{r}^{1/2}}, \\
    \label{eq:Q_pop1}
    \|Q\|_{2 \to \infty}
    \lesssim
    \frac{1}{n \rho_{n}^{1/2} (\mu_{r} - \mu_{r+1})^2}
    +
    \frac{\log^{1/2}{n}}{n \rho_{n}^{1/2} \mu_{r}^{1/2}(\mu_{r} - \mu_{r+1})}\left(r^{1/2} + \frac{\log n}{(n \rho_{n})^{1/2}}\right)
    ,
\end{gather} 
which hold with high probability. 
\end{remark}

For ease of exposition, the conditions for $r$ in \cref{thm:psd} are stated in terms of the eigenvalues $\lambda_{r}$ and $\lambda_{r+1}$ of the edge probability matrix $P$. As $P$ is unknown, the next result, \cref{cor:main1}, replaces these conditions with those based on the eigenvalues $\hat{\lambda}_{r}$ and $\hat{\lambda}_{r+1}$ of $A$. For simplicity, $\rho_{n}$ is assumed to be known, though if $\rho_{n}$ is unknown then we can replace $n \rho_{n}$ with $\Delta \log^{1/2}{n}$ in \cref{eq:r_select_cor2}, where $\Delta$ is the average degree of $A$. In particular, $\Delta \log^{1/2}{n} = \omega(n \rho_{n})$ holds asymptotically almost surely. Furthermore, since $\max_{i \ge 1} |\hat{\lambda}_{i} - \lambda_{i}| \leq \|E\|$ holds by Weyl's inequality, applying \cref{lem:bandeira_vanhandel} yields $\|E\| \leq \varsigma(\nu,n)$ with probability at least $1 - n^{-\nu}$, which yields the conditions for $r$ in \cref{eq:r_select_cor1,eq:r_select_cor2}.

\begin{corollary}[Data-driven dimension selection in \cref{thm:psd}]
  \label{cor:main1}
  Assume the setting and notations in \cref{thm:psd}. 
  Suppose $r \geq 1$ is chosen such that the conditions
  \begin{gather}
    \label{eq:r_select_cor1}
    \hat{\lambda}_{r} - \hat{\lambda}_{r+1}
    \geq 
    \max\left\{4 \varsigma(\nu,n), \tfrac{16}{3}(\nu + 2) \log n\right\}
    +
    2 \varsigma(\nu,n)
    , \\
    \label{eq:r_select_cor2}
    \hat{\lambda}_{r}
    \geq
    \max\left\{16 (\nu + 2) \log n + \frac{64 \varsigma(\nu,n)^{2}}{\hat{\lambda}_{r} - \hat{\lambda}_{r+1}}, \vartheta(\nu+1, r, n)\right\}
    +
    \varsigma(\nu,n)
    ,
  \end{gather}
   are both satisfied, where $\varsigma$ and $\vartheta$ are defined in \cref{eq:varsigma_def,eq:vartheta_def}. 
    Then, \cref{eq:expansion_definite} holds,
   where $E U \Lambda^{-1/2}$ and $Q$ therein satisfy the bounds in \cref{eq:EUlambda,eq:Q2inf_bd} with probability at least $1 - O(n^{-\nu})$.
\end{corollary}
\begin{remark}[Eigenvalues with polynomial decay]
    \label{rem:polynomial_decay}
    Suppose that the eigenvalues of $\mathscr{K}$ exhibit polynomial decay, i.e., $\mu_{r} \asymp r^{-\alpha}$ and $\mu_{r} - \mu_{r+1} \asymp r^{-\beta}$ for some constant $\beta > \alpha > 1$. In this setting,
    \begin{gather*}
        \lambda_{r}
        \asymp
        n \rho_{n} r^{-\alpha},
        \quad
        \text{and}
        \quad
        \delta_{r}
        \asymp
        n \rho_{n} r^{-\beta}, \\
        \|E U \Lambda^{-1/2}\|_{2 \to \infty}
        \lesssim
        \frac{r^{(1 + \alpha)/2} \log^{1/2}{n}}{n^{1/2}}, \\
        \|Q\|_{2 \to \infty}
        \lesssim
        \frac{r^{2 \beta} + r^{\alpha + \beta} \log^{1/2}{n}}{n \rho_{n}^{1/2}} +
        \frac{r^{(1 + \alpha)/2 + \beta} \log^{3/2}{n}}{n^{3/2} \rho_{n}}.
    \end{gather*}
    If $r$ is chosen such that 
    \begin{equation}
        r
        \ll
        \min\left\{
        (n \rho_{n})^{\tfrac{1}{\alpha + 2\beta - 1}},
        (n \rho_{n} \log n)^{\tfrac{1}{4 \beta - \alpha - 1}},
        (n \rho_{n} \log^{-1}{n})^{\tfrac{1}{\beta}}
        \right\}
        ,
    \end{equation}
    then $\|Q\|_{2 \to \infty}$ is negligible compared to $\|E U \Lambda^{-1/2}\|_{2 \to \infty}$. If instead, $r$ is chosen such that
    \begin{equation*}
        r
        \ll
        \min\left\{
        (n \rho_{n} \log^{-1}{n})^{\tfrac{3}{1 + \alpha + 2 \beta}},
        (n \rho_{n} \log^{-1/2}{n})^{\tfrac{1}{\alpha + \beta}},
        (n \rho_{n})^{\tfrac{1}{2\beta}}\right\} ,
    \end{equation*}
    then $\|E U \Lambda^{-1/2} + Q\|_{2 \to \infty} = o(\rho_{n}^{1/2})$ with high probability. Ignoring logarithmic factors in $n$, the above condition corresponds to $r^{2 \beta} = o(n \rho_{n})$ which is slightly more restrictive than the condition $\lambda_{r} = \Omega(\sqrt{n \rho_{n}})$ in the matrix USVT literature \citep{xu_spectral,chatterjee}. Indeed, for the current setting, $\lambda_{r} = \Omega(\sqrt{n \rho_{n}})$ is equivalent to $r^{2 \alpha} = O(n \rho_{n})$. Given that subspace estimation is generally more difficult than low-rank approximation (e.g., see \cite{drineas_ipsen}), we conjecture that the condition $r^{2 \beta} = o(n \rho_{n})$ is not easily improvable.
\end{remark}

\begin{remark}[Eigenvalues with exponential decay]
    \label{rem:exponential_decay}
    Suppose now that the eigenvalues of $\mathscr{K}$ exhibit exponential decay, i.e., $\mu_{r} \asymp \exp(-c_{0} r^{\beta})$ and $\mu_{r} - \mu_{r+1} \asymp r^{\alpha} \exp(-c_{0} r^{\beta})$ for some constants $c_{0} > 0$, $\alpha \geq {0}$, and $\beta \in (0,1]$. This assumption on $\mu_{r}$ arises naturally whenever the kernel $\kappa$ is sufficiently smooth (e.g., infinitely divisible); see, e.g., \cite[Theorem~5]{belkin_approximation}. Consequently, with high probability,
    \begin{gather*}
      \|E U \Lambda^{-1/2}\|_{2 \to \infty}
      \lesssim
      \frac{r^{1/2} \exp(c_{0} r^{\beta})}{n^{1/2}}, \\
      \|Q\|_{2 \to \infty} 
      \lesssim
      \frac{r^{-2 \alpha} \exp(2c_{0} r^{\beta}) \log^{1/2}{n}}{n \rho_{n}^{1/2}} +
      \frac{r^{(1 - \alpha)} \exp(3 c_{0} r^{\beta}/2) \log^{3/2}{n}}{n^{3/2} \rho_{n}}.
    \end{gather*}
    In particular, choosing $r = o((\log n \rho_{n})^{1/\beta})$ yields $\|\hat{U} \hat{\Lambda}^{1/2} W^{(n)} - U \Lambda^{1/2} \|_{2 \to \infty} = o(\rho_{n}^{1/2})$ with high probability, where $E U \Lambda^{-1/2}$ is the dominant term. These findings are consistent with the results in \cite{udell} showing that $n \times n$ matrices whose entries are of the form $h(x_{i}, x_{j})$, where $h$ is piecewise analytic and $\{x_{i}\}_{i=1}^{n}$ are bounded latent variables, can be well-approximated entrywise by a matrix of rank $O(\log n)$.
\end{remark}

\begin{figure}[t!]
  \centering
  \subfloat[Normalized eigenvalues $\lambda_r/n$]{\includegraphics[width=.48\textwidth]{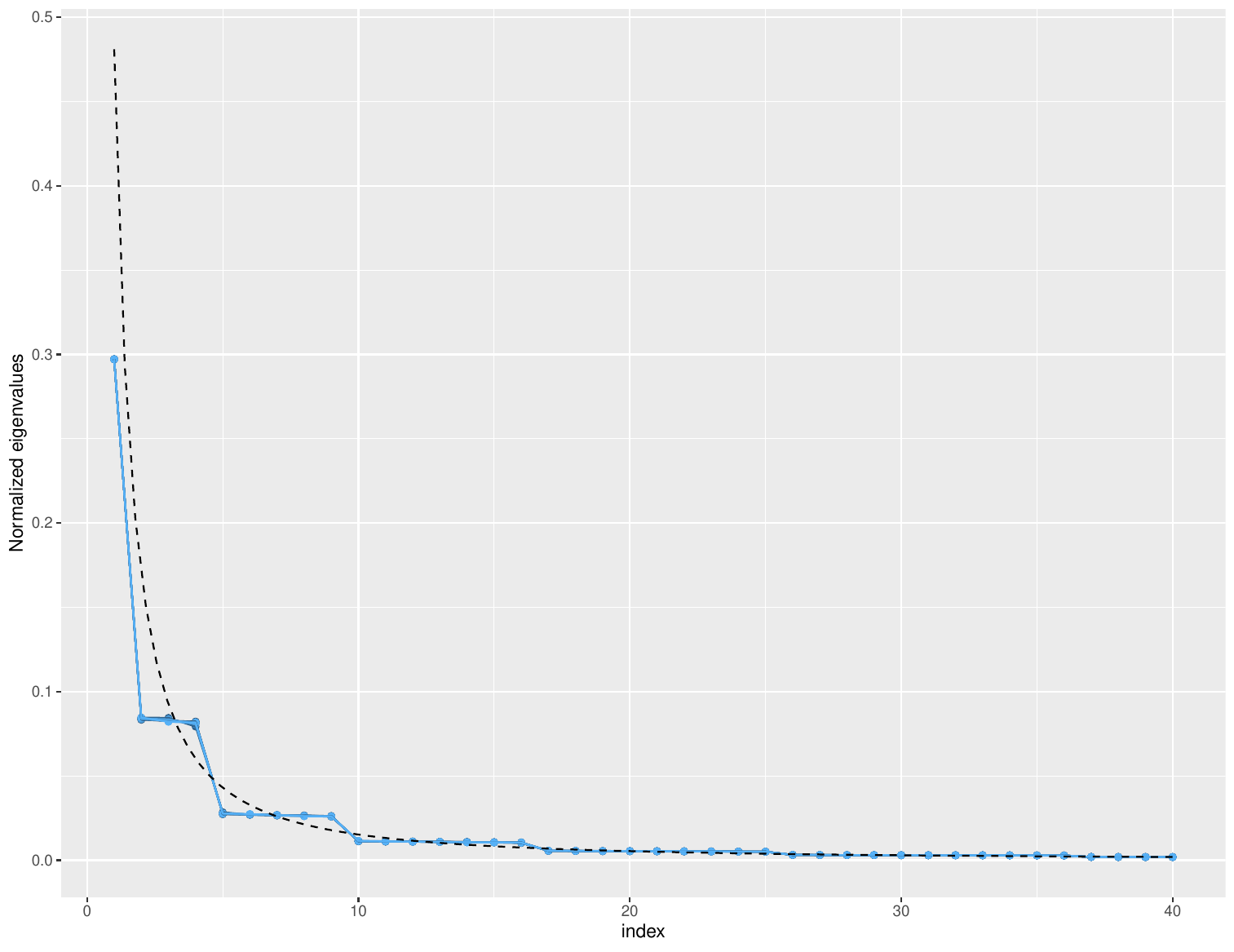}} 
  \subfloat[Normalized eigenvalue gaps $(\lambda_{r} - \lambda_{r+1})/n$]{\includegraphics[width=.48\textwidth]{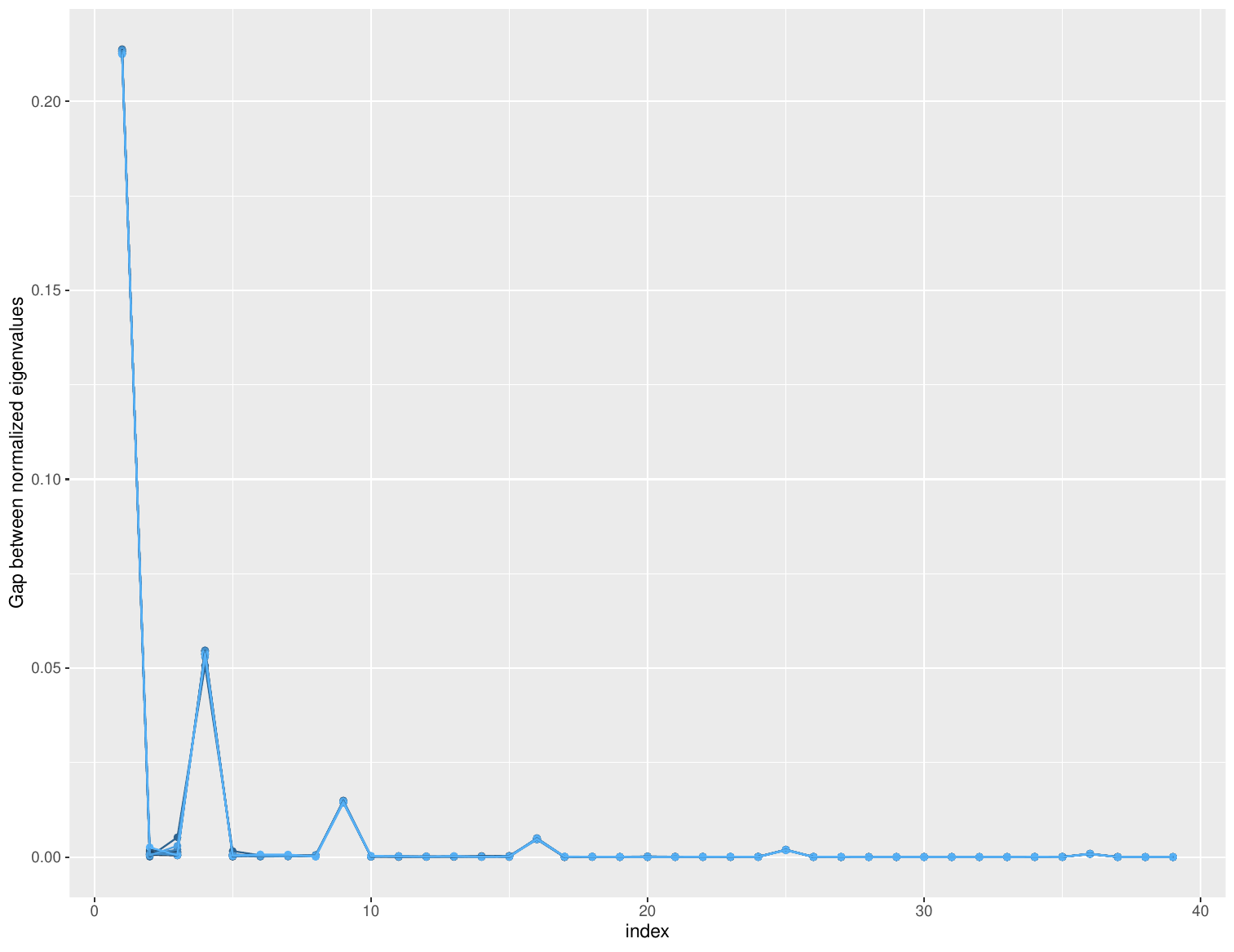}}
  \caption{Plots of the forty largest eigenvalues (left panel) and gap between consecutive eigenvalues (right panel) for $n \times n$ edge probability matrices $P$ when the link function is $\kappa(x,y) = \exp(-\|x - y\|)$, where $X_{1}, \dots, X_{n}$ are sampled i.i.d. from the uniform distribution on the unit sphere in $\mathbb{R}^{3}$, and $n = 8000, \rho_{n} = 1$. There are $10$ blue-colored lines (nearly indistinguishable) in each panel, one for each independent Monte Carlo replicate. The dashed line in the left panel corresponds to the best approximation of the form $\lambda_{r}/n \propto r^{-3/2}$.}
  \label{fig:eigengap_simulation1}
\end{figure}

\begin{remark}[Large and small eigenvalue gaps]
\label{rem:eigenval_gap_plot}
  For ease of exposition, our discussion in \cref{rem:polynomial_decay,rem:exponential_decay} assumes that the gaps between consecutive eigenvalues of $\mathscr{K}$ decay to zero at a given rate. However, in practice, while the eigenvalues themselves can exhibit a certain decay rate, gaps between these eigenvalues can be arbitrarily small. Indeed, \cref{fig:eigengap_simulation1} presents a summary of the $40$ largest eigenvalue of the $n \times n$ matrix $P$ for $n = 8000$ when the link function is $\kappa(x,y) = \exp(-\|x - y\|)$ and the latent positions are sampled i.i.d. from the uniform distribution on the unit sphere in $\mathbb{R}^{3}$. While the eigenvalues of $P$ can be fitted quite well by a curve of the form $\lambda_{r}/n \propto r^{-3/2}$, the gaps between consecutive eigenvalues are generally near-zero except for visible jumps at a few locations, such as $r = 1, 4, 9, 15$.
\end{remark}

\subsection{Extension to indefinite kernels}
\label{sec:indefinite}
We now consider the case where $\kappa$ is indefinite, for which the edge probability matrix $P$ and the integral operator $\mathscr{K}$ induced by $\kappa$ per \cref{eq:integral_op} each have both positive and negative eigenvalues. The first issue we face is that the bound $\|U |\Lambda|^{1/2}\|_{2 \to \infty} \leq \rho_{n}^{1/2}$ for all $r$ in  \cref{eq:bounded_coherence} no longer holds. Indeed, although
\begin{equation*}
    \|U |\Lambda|^{1/2}\|_{2 \to \infty}^2
    =
    \|U |\Lambda| U_{\perp}^{\top}\|_{\max}
    \leq
    \| \, |P| \, \|_{2 \to \infty}
    ,
\end{equation*}
here the entries of $U |\Lambda| U^{\top}$ and $|P| \equiv (P^{2})^{1/2}$ are not directly related to those of $P$. In other words, there is no closed-form expression for $|P|$ in terms of simple element-wise operations on $P$ when $P$ is indefinite.
Nevertheless, from \cref{eq:2toinf_submultiplicative} we at least have
\begin{equation}
    \label{eq:u_lambda_2toinf}
    \|U |\Lambda|^{1/2}\|_{2 \to \infty}
    \leq
    |\lambda_r|^{-1/2} \|U \Lambda\|_{2 \to \infty} 
    =
    |\lambda_r|^{-1/2} \|P U\|_{2 \to \infty}
     \leq
    |\lambda_r|^{-1/2} n^{1/2} \rho_{n}
    ,
\end{equation}
where the last inequality holds because $\|PU\|_{2 \to \infty} = \|P\|_{2 \to \infty}$ whenever $U$ has orthonormal columns, and the entries of $P$ are bounded in magnitude by $\rho_{n}$. Since $|\lambda_{r}| \leq n \rho_{n}$, the bound in \cref{eq:u_lambda_2toinf} is worse than in \cref{eq:bounded_coherence}, yet it is still meaningful
because it yields the following generalization of \cref{thm:psd} presented in \cref{thm:general}.

Note that by convention, unless specified otherwise, the eigenvalues $\{\hat{\lambda}_{j}\}_{j \ge 1}$ of $A$ and eigenvalues $\{\lambda_{j}\}_{j \ge 1}$ of $P$ are always ordered in decreasing modulus, that is, $|\hat{\lambda}_{1}| \geq |\hat{\lambda}_{2}| \geq \dots \geq 0$ and $|\lambda_{1}| \geq |\lambda_{2}| \geq \dots \geq 0$.

\begin{theorem}[Fine-grained eigenvector perturbation analysis for LPGs with general kernels]
    \label{thm:general}
    Let $A \sim \mathrm{LPG}(\kappa, F; \rho_{n})$ be a graph on $n$ vertices generated according to \cref{def:lpg}, where $\kappa$ is possibly indefinite. For a given $r \geq 1$, let the diagonal matrices $\hat{\Lambda}$ and $\Lambda$ contain the $r$ largest in magnitude eigenvalues of $A$ and $P$, respectively, and let $U$ and $\hat{U}$ be the $n \times r$ matrices whose orthonormal columns are the corresponding eigenvectors of $A$ and $P$. Let $\delta_{r} = |\lambda_{r}| - |\lambda_{r+1}|$. Fix $\nu \geq 2$ and suppose $r \geq 1$ is chosen such that the conditions
    \begin{gather}
        \label{eq:r_select_general}
        \delta_{r}
        \geq
        \max\{4 \varsigma(\nu,n), \tfrac{16}{3} (\nu + 2) \log n \}
        ,\\
        \label{eq:r_select_general2}
        |\lambda_{r}|
        \geq
        \max
        \left\{16 (\nu + 2) \log n + \frac{64 \vartheta(\nu,n)^{2}}{\delta_{r}},
        \sqrt{n \rho_{n}\vartheta(\nu+2,r,n)} \right\}
        ,
   \end{gather}
   are both satisfied, where $\varsigma$ and $\vartheta$ are defined in \cref{eq:varsigma_def,eq:vartheta_def}. Then, there exists an $r \times r$ orthogonal matrix $W^{(n)}$ such that
    \begin{equation}
      \label{eq:expansion_indefinite}
        \hat{U} |\hat{\Lambda}|^{1/2} W^{(n)} - U |\Lambda|^{1/2}
        =
        E U |\Lambda|^{-1/2}
        +
        Q,
    \end{equation}
    where $E U |\Lambda|^{-1/2}$ satisfies
    \begin{gather}
        \label{eq:EUlambda_indefinite}
        \| E U |\Lambda|^{-1/2} \|_{2 \to \infty} 
        \leq
        \frac{11}{2} |\lambda_{r}|^{-1/2} \sqrt{\rho_{n} \vartheta(\nu+1,r,n)}
    \end{gather}
    with probability at least $1 - n^{-\nu}$, while $Q$ satisfies
    \begin{equation}
        \begin{split}
        \label{eq:Q2inf_bd_indefinite}
        \|Q\|_{2 \to \infty}
        &\lesssim
        \frac{n^{3/2} \rho_{n}^{2}}{|\lambda_{r}|^{1/2} \delta_{r}^{2}} 
        +
        \frac{(\rho_{n} \log n)^{1/2}}{|\lambda_{r}|^{1/2}}
        \left(\frac{(r n \rho_{n})^{1/2} + \log n}{\delta_{r}} \right)
        \end{split} 
    \end{equation}
    with probability at least $1 - O(n^{-\nu})$. Here, $\lesssim$ and $O(\cdot)$ hide universal constants that depend only on $\nu$ but not on $r$, $n$, $\rho_{n}$, or $\kappa$.
\end{theorem}

\begin{remark}[Residual analysis for positive semidefinite versus possibly indefinite kernels]
    \label{rem:comparison_indefinite}
    We now compare the terms appearing in the upper bound for $Q$ in \cref{thm:psd,thm:general}. To begin, observe that the second terms in each of \cref{eq:Q2inf_bd,eq:Q2inf_bd_indefinite} are identical. The first term in \cref{eq:Q2inf_bd_indefinite} resembles the first term in \cref{eq:Q2inf_bd}, with the main difference being the additional multiplicative factor $(n \rho_{n})^{1/2} |\lambda_{r}|^{-1/2}$ due to the use of \cref{eq:u_lambda_2toinf} as described above. \cref{eq:u_lambda_2toinf} also manifests itself in \cref{eq:r_select_general2}, where the former condition $\lambda_{r} \geq \vartheta(\nu+1,r,n)$ (see \cref{eq:r_select_psd2}) is replaced with the more restrictive condition $|\lambda_{r}| \geq \sqrt{n \rho_{n} \vartheta(\nu+2,r,n)}$. The explanation for this is rooted in the main order term $EU|\Lambda|^{-1/2}$. More specifically, a modified version of Bernstein's inequality (see \cref{eq:ex_2inf}) yields the bound
    \begin{equation*}
        \|EU \Lambda^{-1/2}\|_{2 \to \infty}
        \leq
        2 \sqrt{2} \lambda_{r}^{-1/2} \sqrt{\rho_{n} \vartheta(\nu+1,r,n)}
        +
        \tfrac{8}{3} \|U |\Lambda|^{-1/2}\|_{2 \to \infty} \vartheta(\nu+1,r,n)
    \end{equation*}
    with probability at least $1 - n^{-\nu}$. If we simplify the above display expression to keep only the first term as in \cref{eq:EUlambda_indefinite}, then we need to ensure that
    \begin{equation*}
        \|U |\Lambda|^{-1/2} \|_{2 \to \infty} \vartheta(\nu+1,r,n) 
        \leq
        n^{1/2} \rho_{n} |\lambda_{r}|^{-3/2}
        \leq
        |\lambda_{r}|^{-1/2} \sqrt{\rho_{n} \vartheta(\nu+1,r,n)}
        .
    \end{equation*}
    Ignoring constant factors, this is equivalent to requiring $|\lambda_{r}| \geq \sqrt{n \rho_{n} \vartheta(\nu+2,r,n)}$. Finally, similar to the discussion surrounding \cref{eq:condition_rem4}, with high probability $\|Q\|_{2 \to \infty} = o(|\lambda_{r}|^{-1/2} \sqrt{\rho_{n} \vartheta(\nu+2,r,n)})$, so $EU|\Lambda|^{-1/2}$ is the dominant term in the expansion of \cref{eq:expansion_indefinite} whenever the conditions
    \begin{gather}
        \label{eq:dominant_indefinite1}
        \delta_{r}^{4} (r + \log n)
        =
        \omega\left((n \rho_{n})^{3}\right)
        \qquad
        \text{and}
        \qquad
        \delta_{r}
        =
        \omega\left(\max\left\{\sqrt{n \rho_{n} \log n}, \log n\right\}\right)
    \end{gather}
  are simultaneously satisfied.
\end{remark}

We pause here to discuss choosing $r \equiv r(n)$ to grow with $n$ in \cref{thm:general}. For simplicity of presentation, assume that $\mathrm{rk}(\kappa) = \infty$, otherwise there exists a positive constant $M < \infty$ for which $\mathrm{rk}(P) \leq M$ almost surely and hence we are back in the more widely studied, comparatively easy finite-dimensional low-rank setting where we can simply choose $r = \mathrm{rk}(P)$ for sufficiently large $n$. Recall that when $\kappa$ is positive semidefinite, then we can apply the concentration inequality in \cref{eq:eigenvalue_convergence} to obtain a sequence $\{r(n)\}_{n \ge 1}$ (possibly diverging with $n$) for which \cref{thm:general} applies; see \cref{rem:dim_select,rem:polynomial_decay,rem:exponential_decay} for additional details. In contrast, we are unaware of any prior result that yields a bound as sharp as \cref{eq:eigenvalue_convergence} when $\kappa$ is indefinite.

To elaborate, since $\sup_{x,y \in \mathcal{X}} |\kappa(x,y)| \leq 1$, by \cite[Theorem~3.1]{koltchinskii00:_random} we have
\begin{equation*}
    \delta_{2}\left((n \rho_{n})^{-1} \{\lambda_{k}\}_{k \ge 1}, \{\mu_{k}\}_{k \ge 1}\right) 
    \overset{\mathrm{a.s.}}{\longrightarrow}
    0
\end{equation*}
as $n \rightarrow \infty$, where the distance $\delta_{2}$ between sequences $(v_k)_{k \ge 1}$ and $(w_k)_{k \ge 1}$ can be written as
\begin{equation}
  \label{eq:koltchinskii1}
    \delta_{2}(v,w)
    = \min_{\sigma} \sum_{k} (v_k - w_{\sigma(k)})^2
\end{equation}
and the minimization is over all bijections $\sigma$ of the natural integers.  The upshot is that 
\cref{eq:koltchinskii1} only guarantees that the extended spectrum of $P$, properly normalized, converges to that of $\mathscr{K}$ but does not specify a convergence rate.

If we further assume that
\begin{equation}
    \label{eq:koltchinskii2}
    \sum_{r=1}^{\infty} 
    \sum_{s=1}^{\infty}
    (\mu_{r}^{2} + \mu_{s}^{2})
    \int_{\mathcal{X}}
    \phi_{r}^{2}(x) \phi_{s}^{2}(x) \, \mathrm{d} F(x)
    =
    C_{*}
    <
    \infty,
\end{equation}
then \cite[Corollary~4.3]{koltchinskii00:_random} gives the bound
\begin{equation*}
    \mathbb{E}
    \left[\delta_{2}^{2}\left((n \rho_{n})^{-1}\{\lambda_{k}\}_{k \ge 1}, \{\mu_{k}\}_{k \ge 1})\right)
    \right]
    \leq
    6C_{*} n^{-1}
    .
\end{equation*}
An application of Markov's inequality thus yields a convergence rate of $O_{p}(n^{-1/2})$ for $(n \rho_{n})^{-1}\{\lambda_{k}\}_{k \ge 1}$ (here $X_n = O_p(f(n))$ means that for any $\epsilon > 0$ there exists a $C > 0$ such that $\mathbb{P}(X_n \leq C f(n)) \geq 1 - \epsilon$).  
While this bound is qualitatively similar to \cref{eq:eigenvalue_convergence}, there are two significant differences. Firstly, \cref{eq:koltchinskii2} is difficult to verify in practice, since we only observe $A$ while $\kappa$ is unknown. Secondly, the rate $O(n^{-1/2})$ only holds with probability $1 - o(1)$, so we cannot guarantee that $\delta_{2}((n \rho_{n})^{-1}\{\lambda_{k}\}_{k \ge 1}, \{\mu_{k}\}_{k \ge 1}) = O(n^{-1/2})$ asymptotically almost surely. This makes our goal of letting $r$ grow with $n$ more subtle. For example, suppose $|\mu_k|$ exhibits polynomial decay and let $\mathcal{E}_n$ be the event that \cref{eq:r_select_general,eq:r_select_general2} are satisfied simultaneously. Then, for a sequence $\{r(n)\}_{n \ge 1}$ similar to that in \cref{rem:polynomial_decay}, we only have $\lim_{n \rightarrow \infty} \mathbb{P}(\mathcal{E}_n) = 1$ and not $\mathbb{P}( \cup_{m = 1}^{\infty} \cap_{n = m}^{\infty} \mathcal{E}_n) = 1$. 

On a related note, \cite[Definition~3]{graph_root} considers the assumption 
\begin{equation}
    \label{eq:hypothesis_H} 
    \sup_{x \in \mathcal{X}} h_{\kappa}(x)
    =
    \sup_{x \in \mathcal{X}} \sum_{r \ge 1} |\mu_{r}| \phi_{r}^{2}(x)
    <
    \infty
    .
\end{equation}
We note that \cref{eq:hypothesis_H} is somewhat easier to verify than \cref{eq:koltchinskii2}. In particular, it is satisfied whenever $\mathscr{K}$ has only a {\em finite} number of positive or negative eigenvalues; examples of kernels with this property are conditionally positive definite kernels. By applying \cite[Theorem~8.1.2]{tailen_hsing}, one can show that $\max_{k \ge 1} |(n \rho_{n})^{-1} \lambda_{k} - \mu_{k}| = O_{p}(n^{-1/2})$, with the convention that $\lambda_k = 0$ for $k > n$; see also \cite[Equation~16]{graph_root}. Once again, here the rate of order $n^{-1/2}$ only holds with probability $1 - o(1)$, not with high probability as we desire. 

Finally, under the conditions in \cref{eq:hypothesis_H}, for any fixed but arbitrary $k \ge 1$ \cite[Theorem~1]{relative_1} gives the bound
\begin{equation}
  \label{eq:valdivia1}
  |(n \rho_{n})^{-1} \lambda_{k} - \mu_{k}|
  \lesssim
  |\mu_{k}| \sqrt{\frac{\mathcal{V}_{1}(R(k)) \log (R(k)/\alpha)}{n}}
\end{equation}
with probability at least $1 - \alpha$. Here, $\mathcal{V}_{1}(i) = \sup_{x \in \mathcal{X}} \sum_{k=1}^{i} \phi_{k}^{2}(x)$, and $R(k)$ is the minimum index $R$ for which $|\mu_{k}| \geq \max\left\{\sum_{\ell > R} |\mu_{\ell}|, \sqrt{R \sum_{\ell > R} \mu_{\ell}^{2}}\right\}$. Although \cref{eq:valdivia1} closely resembles \cref{eq:eigenvalue_convergence}, one still needs to bound $R(k)$ and $\mathcal{V}_1(R(k))$. Doing so will require introducing further assumptions, as otherwise both the tail sequence of eigenvalues $\{\mu_{\ell}\}_{\ell > k}$ and the eigenfunctions $\{\phi_{i}\}_{i \leq k}$ are unknown. 

Finally, we end this subsection by presenting a result concerning the limiting distribution for the row-wise fluctuations of $\hat{U} |\hat{\Lambda}|^{1/2} W^{(n)} - U |\Lambda|^{1/2}$ as $n \rightarrow \infty$ with $r$ fixed.

\begin{corollary}[Row-wise limiting distribution for $\hat{U} |\hat{\Lambda}|^{1/2}$]
\label{cor:normal}
Assume the setting and notations in \cref{thm:general}. Fix a $r \ge 1$ such that the following condition is satisfied.
\begin{equation}
\label{eq:cor_deltar_rowwise}
    \delta_r
    =
    \omega\left(\max\left\{|\lambda_r|^{-1/2} \sqrt{n \rho_n} \log^{3/2}{n}, |\lambda_r|^{-1/2} n \rho_n \sqrt{r \log n}, |\lambda_r|^{-1/4} n \rho_n\right\}\right)
\end{equation}
For a specified choice of index $i \in [n]$, define
\begin{equation}
   \label{eq:covariance} 
    \Sigma_{i}
    =
    |\Lambda|^{-1/2} \left(\sum_{k} p_{ik}(1 - p_{ik}) U_{k} U_{k}^{\top} \right) |\Lambda|^{-1/2}
    ,
\end{equation}
and suppose that there exists a constant $c_0 > 0$ such that
\begin{equation}
    \label{eq:cond_cov1}
    c_0 \rho_n I
    \preceq
    \left(\sum_{k} p_{ik}(1 - p_{ik}) U_{k} U_{k}^{\top} \right)
    \preceq
    \rho_n I
    ,
\end{equation} 
where $\preceq$ denote the Loewner ordering for positive semidefinite matrices. 
Then,
\begin{gather}
    \label{eq:clt_rowwise} \Sigma_{i}^{-1/2} \left((W^{(n)})^{\top} |\hat{\Lambda}|^{1/2} \hat{U}_i - |\Lambda|^{1/2} U_i\right) 
    \rightsquigarrow
    N(0, I)
\end{gather}
as $n \rightarrow \infty$. Here, $\hat{U}_i$ and $U_i$ are the $i$-th row of $\hat{U}$ and $U$, respectively. Furthermore, if $\kappa$ is positive semidefinite, then \cref{eq:cor_deltar_rowwise} can be weakened slightly to
\begin{equation}
    \label{eq:cor_deltar_rowwise_psd}
    \delta_r
    =
    \omega\left(
    \max
    \left\{\lambda_r^{-1/2} \sqrt{n \rho_n} \log^{3/2}{n},
    \lambda_r^{-1/2} n \rho_n \sqrt{r \log n},
    (n \rho_n)^{3/4}
    \right\}
    \right)
    .
\end{equation}
\end{corollary}

\cref{cor:normal} extends previous row-wise limiting distribution results for $\hat{U} |\hat{\Lambda}|^{1/2}$ under low-rank assumptions, where $\mathrm{rk}(P) = d$ for some constant $d$ (see e.g., \cite{xie2021entrywise,grdpg1,athreya2013limit}), to the growing rank or full rank setting. \cref{eq:covariance} provides an expression for the covariance matrix of $E_i U |\Lambda|^{-1/2} = |\Lambda|^{-1/2} \sum_{k=1}^{n} (a_{ik} - p_{ik}) U_k$; here $E_i$ is the $i$-th row of $E$. \cref{eq:cond_cov1} is satisfied whenever there exists a constant $c_0 > 0$ such that $\kappa(x,x') \in [c_0, 1 - c_0]$ for all $x,x' \in \mathcal{X}$. \cref{eq:cor_deltar_rowwise} or \cref{eq:cor_deltar_rowwise_psd} ensure that  $\|\Sigma_{i}^{-1/2} Q_i\| \rightarrow 0$ in probability, and we can ignore all terms depending on $Q$ in the limiting distribution.  \cref{eq:clt_rowwise} then follows from applying the Lindeberg--Feller central limit theorem \cite[Proposition~2.27]{van2000asymptotic} to $\Sigma_{i}^{-1/2} E_i U |\Lambda|^{-1/2}$.  Finally, $r$ is fixed in \cref{cor:normal} as otherwise if $r \rightarrow \infty$ then the convergence in distribution of \cref{eq:clt_rowwise} is possibly not well-defined. Nevertheless, to handle growing $r$, we can replace the limiting distribution of $\xi_i = (W^{(n)})^{\top} |\hat{\Lambda}|^{1/2} \hat{U}_i - |\Lambda|^{1/2} U_i$ with that of $\|\xi_i\|$ and then apply a comparison result for quadratic forms in \cite{rotar_1}; see \cref{thm:clt_quadratic} in \cref{sec:quad_test_proof} for more details.

\subsection{Estimation of $U$ versus $U |\Lambda|^{1/2}$ versus $U\Lambda$}
We next present perturbation expansions and two-to-infinity norm bounds for the expressions $\hat{U}W^{(n)} - U$ and $\hat{U} \hat{\Lambda} W^{(n)} - U \Lambda$, where $W^{(n)}$ denotes an orthogonal matrix. Although these quantities are closely related and themselves resemble $\hat{U} |\hat{\Lambda}|^{1/2} W^{(n)} - U |\Lambda|^{1/2}$, certain inference tasks can be simplified if one chooses wisely from among these expressions. For example, if $\kappa$ is positive semidefinite and $P_{r} = U \Lambda U^{\top}$ denotes the best rank-$r$ approximation to $P$, then it is reasonably straightforward to analyze the estimation error of $\hat{P}_{r} = \hat{Z}_{r} \hat{Z}_{r}^{\top}$, where $\hat{Z}_{r} = \hat{U} \hat{\Lambda}^{1/2}$, using the results in \cref{thm:psd}. Similarly, in \cref{sec:two_sample}, we consider a test statistic based on the Euclidean norm row difference $\|(\hat{U} \hat{\Lambda})_{i} - (\hat{U} \hat{\Lambda})_{j}\|$ for which it is more convenient to leverage the row-wise expansion for $\hat{U} \hat{\Lambda}$ given below.

\begin{theorem}[Additional eigenvector perturbation analysis for LPGs with general kernels]
    \label{thm:generalU_ULambda}
    Assume the setting and notations in \cref{thm:general} where $\kappa$ is possibly indefinite. Then, there exists an $r \times r$ orthogonal matrix $W^{(n)}$ such that
  \begin{equation}
  \label{eq:expansion_hatU}
    \hat{U} W^{(n)} - U
    =
    E U \Lambda^{-1} + \mathring{Q}
    ,
  \end{equation}
  where $E U \Lambda^{-1}$ satisfies
  \begin{gather}
    \label{eq:EUlambda_indefiniteU}
     \| E U \Lambda^{-1} \|_{2 \to \infty}
     \leq
     \frac{11}{2} \rho_{n}^{1/2} |\lambda_{r}|^{-1} \sqrt{\vartheta(\nu+1,r,n)}  \end{gather}
   with probability at least $1 - n^{-\nu}$ and $\mathring{Q}$ satisfies
  \begin{equation}
    \begin{split}
    \label{eq:Q2inf_bd_indefiniteU}
    \|\mathring{Q}\|_{2 \to \infty} &\lesssim
    \frac{n^{3/2} \rho_{n}^{2}}{|\lambda_{r}| \delta_{r}^{2}} 
    +  \frac{(\rho_{n} \log n)^{1/2}}{|\lambda_{r}|} \left(\frac{(r n \rho_{n})^{1/2} + \log n}{\delta_{r}}\right)
  \end{split} 
  \end{equation}
  with probability at least $1 - O(n^{-\nu})$. Similarly, it holds that
    \begin{equation*}
        \hat{U} \hat{\Lambda} W^{(n)} - U \Lambda
        =
        E U + \breve{Q},
    \end{equation*}
   where $E U$ and $\breve{Q}$ now satisfy
    \begin{gather}
        \label{eq:EU_indefiniteU}
        \| E U \|_{2 \to \infty} 
        \leq
        \frac{11}{2} \rho_{n}^{1/2} \sqrt{\vartheta(\nu+1,r,n)},\\
        \label{eq:Q2inf_bd_indefiniteU2}
        \|\breve{Q}\|_{2 \to \infty}
        \lesssim
        \frac{n^{3/2} \rho_{n}^{2}}{\delta_{r}^{2}} 
       +
        (\rho_{n} \log n)^{1/2} \left(\frac{(r n \rho_n)^{1/2} + \log n}{\delta_r}\right)
    \end{gather}
    with probability at least $1 - n^{-\nu}$ and $1 - O(n^{-\nu})$ respectively.
    Furthermore, if $\kappa$ is positive semidefinite, then the above bound for $\mathring{Q}$  
    simplifies to
    \begin{gather}
        \label{eq:ringQ_psd}
        \|\mathring{Q}\|_{2 \to \infty} \lesssim
        \frac{n \rho_{n}^{3/2}}{\lambda_{r}^{1/2} \delta_{r}^{2}}
        + \frac{(\rho_{n} \log n)^{1/2}}{\lambda_{r}}
        \left(\frac{(r n \rho_{n})^{1/2} + \log n}{\delta_{r}}\right) 
     \end{gather}
    but the bound for $\breve{Q}$ remains unchanged. Above, the notation $\lesssim$ and $O(\cdot)$ suppress universal constants that depend only on $\nu$ but not on $r$, $n$, $\rho_{n}$, or $\kappa$.  
\end{theorem}

The first term in the bound for $\mathring{Q}$ corresponds to $\|U\|_{2 \to \infty} \times \|U^{\top} \hat{U} - (W^{(n)})^{\top}\|$. If $\kappa$ is positive semidefinite, then $\|U\|_{2 \to \infty} \leq \rho_{n}^{1/2} |\lambda_{r}|^{-1/2}$ which yields the sharper bound in \cref{eq:ringQ_psd}, while if $\kappa$ is indefinite then we can only obtain $\|U\|_{2 \to \infty} \leq n^{1/2} \rho_{n} |\lambda_{r}|^{-1}$ as in \cref{eq:Q2inf_bd_indefiniteU}. See also the discussion in \cref{rem:comparison_indefinite}. Similarly, the first term in the bound for $\breve{Q}$ corresponds to $\|U \Lambda\|_{2 \to \infty} \times \|U^{\top} \hat{U} - (W^{(n)})^{\top}\|$. However, in contrast to the case involving $\|U\|_{2 \to \infty}$, we always have $\|U \Lambda\|_{2 \to \infty} \leq n^{1/2} \rho_{n}$, and thus there are no differences in the upper bound for $\breve{Q}$ when $\kappa$ is positive semidefinite versus when it is indefinite. 

Finally, we note that \cref{thm:generalU_ULambda} can be adapted to yields a deterministic perturbation bound for the $r$ leading eigenvectors $U$ of a symmetric matrix $M$ when it is perturbed by some arbitrary noise matrix $E$; see \cref{thm:deterministic} in \cref{sec:deterministic} for a formal statement. Notably, the bound in \cref{thm:deterministic} only depend on quantities that are {\em linear} in $E$ and thus, for many inference settings, can be analyzed using standard matrix perturbation and concentration inequalities. 
\subsection{Related works}
We now compare our results with existing perturbation analysis and bounds in the literature. Notably, two-to-infinity norm bounds for $\hat{U} - U W^{(n)}$ have appeared in numerous publications to date; see \cite{cape2019two,ctp_biometrika,abbe2020entrywise,lei2019unified,mao_sarkar,pensky_2toinfinity,xie2021entrywise,damle_sun} for an incomplete list of references. However, these works overwhelmingly either focus only on first order upper bounds for $\hat{U} - U W^{(n)}$, which may not be sufficiently refined for inference purpose, or are restricted to the case where $\mathrm{rk}(P) \ll n$. In contrast, this paper delicately decomposes $\hat{U} - U W^{(n)}$ into the main order term $E U \Lambda^{-1}$ and a residual order term $Q$ while minimizing factors that depend on inverse powers of $\delta_{r}$. To illustrate these challenges, suppose that $\kappa$ is positive semidefinite and consider the statement of \cite[Theorem~2.1]{abbe2020entrywise}. Assumption~A.3 in \cite{abbe2020entrywise} posits that $32 \|P\|/\delta_{r} \max\{\gamma, \varphi(\gamma)\} \leq 1$ and that $\|E\| \leq \gamma \delta_{r}$, where $\gamma$ is any value for which $\|P\|_{2 \to \infty} \leq \gamma \delta_{r}$ and $\varphi(x)$ is a function proportional to $1/(\max\{1, \log(x^{-1})\})$. For the setting in \cref{thm:generalU_ULambda}, we have $\|P\| \asymp n \rho_{n}$ and $\|P\|_{2 \to \infty} \asymp n^{1/2} \rho_{n}$, hence we can take $\gamma = \delta_{r}^{-1} (n \rho_{n})^{1/2}$. Theorem~2.1 in \cite{abbe2020entrywise} then yields the high probability bound
\begin{equation*}
    \|\hat{U} W^{(n)} - U\|_{2 \to \infty} 
    \leq
    \|E U \Lambda^{-1}\|_{2 \to \infty}
    +
    \|\tilde{Q}\|_{2 \to \infty}
    ,
\end{equation*}
where $\tilde{Q}$ satisfies
\begin{equation}
\label{eq:Q_abbe_comparison}
\begin{split}
   \|\tilde{Q}\|_{2 \to \infty}
   &\lesssim
   \kappa_{r} (\kappa_{r} + \phi(1)) (\gamma + \phi(\gamma)) \|U\|_{2 \to \infty} + \gamma \|P\|_{2 \to \infty} \times \frac{1}{\delta_{r}} \\
   &\lesssim
   \|U\|_{2 \to \infty} \times \frac{\kappa_{r}^{2} (n \rho_{n})^{1/2}}{\delta_{r}} + \frac{n \rho_{n}^{3/2}}{\delta_{r}^{2}} \\
   &\lesssim 
   \frac{\kappa_{r}^{2} n^{1/2} \rho_{n}}{\lambda_{r}^{1/2} \delta_{r}} + \frac{n\rho_{n}^{3/2}}{\delta_{r}^{2}}
   .
\end{split}
\end{equation}
Here (slight abuse of notation), $\kappa_{r} = \lambda_{1}/\delta_{r}$, with $\kappa_{r} \rightarrow \infty$ as $r \rightarrow \infty$. 

Comparing \cref{eq:Q_abbe_comparison,eq:ringQ_psd}, we see that the first term in \cref{eq:ringQ_psd} is $\lambda_{r}^{-1/2}$ times smaller than the last term in \cref{eq:Q_abbe_comparison}. Meanwhile, the ratio of the second term in \cref{eq:ringQ_psd} to that of the second to last term in \cref{eq:Q_abbe_comparison} is
\begin{equation*}
    \left( \frac{\kappa_{r}^{2} n^{1/2} \rho_{n}}{\lambda_{r}^{1/2} \delta_{r}}\right)^{-1} \frac{(\rho_{n} \log n)^{1/2}}{\lambda_{r}}
    \left(\frac{\log n + (r n \rho_{n})^{1/2}}{\delta_{r}}\right)
    =
    \frac{ (r \log n)^{1/2}}{\kappa_{r}^{2}} \left(\frac{\log n}{\lambda_{r}^{1/2} (n \rho_{n})^{1/2}} + \frac{1}{\lambda_{r}^{1/2}}\right)
    ,
\end{equation*}
which also converges to zero at the rate $O(\kappa_{r}^{-2} \lambda_{r}^{-1/2})$ (ignoring logarithmic factors). As such, the bound in \cref{eq:ringQ_psd} is smaller than in
\cref{eq:Q_abbe_comparison} by a multiplicative factor of at least $\lambda_{r}^{1/2}$. 
Furthermore, for \cref{eq:Q_abbe_comparison}, to guarantee that $\tilde{Q}$ is negligible compared to $E U \Lambda^{-1}$, we would need to require that simultaneously
\begin{gather}
\label{eq:discussion0}
    \lambda_{r}^{-1} (r \rho_{n} \log n)^{1/2}
    =
    \omega(n \rho_{n}^{3/2} \delta_{r}^{-2})
    \quad
    \Longleftrightarrow
    \quad
    \delta_{r}^{2} (r \log n)^{1/2}
    =
    \omega(n \rho_{n} \lambda_{r}), \\ \label{eq:discussion1}
    \lambda_{r}^{-1} (r \rho_{n} \log n)^{1/2}
    =
    \omega(\kappa_{r}^{2} n^{1/2} \rho_{n}/(\lambda_{r}^{1/2} \delta_{r})) 
    \quad
    \Longleftrightarrow
    \quad
    \delta_{r}^{2} (r \log n)
    =
    \omega(\kappa_{r}^{4} n \rho_{n} \lambda_{r})
    .
\end{gather}
However, the conditions in \cref{eq:discussion0,eq:discussion1} might both be infeasible when $n$ increases as $\kappa_{r} > 1$ and $n \rho_{n} \geq \lambda_{1} \geq \lambda_{r} \geq \delta_{r}$ always. In contrast, our bounds in \cref{thm:generalU_ULambda} imply that $\|\mathring{Q}\|_{2 \to \infty}$ is negligible compared to $\|E U \Lambda^{-1}\|_{2 \to \infty}$ whenever
\begin{equation}
\label{eq:discussion4}
    r \delta_{r}^{4} \log n
    =
    \omega((n \rho_{n})^{2} \lambda_{r}),
\end{equation}
which is a much milder condition and can always be satisfied for any given $r$ as $n$ increases. 

Next, \cite[Theorem~3.4]{lei2019unified} improves upon \cite[Theorem~2.1]{abbe2020entrywise} and yields an upper bound for $\tilde{Q}$ of the form
\begin{equation} 
\label{eq:lihua_comparison}
\|\tilde{Q}\|_{2 \to \infty} \lesssim \frac{n^{1/2} \rho_n}{\lambda_r^{1/2} \delta_r} + \frac{\rho_n^{1/2} (r^{1/2} + \log^{1/2}{n})}{\lambda_r} \times \frac{ (n \rho_n)^{1/2} (r + \log n)^{3/4}}{\delta_r}.
\end{equation}
Comparing \cref{eq:ringQ_psd,eq:lihua_comparison} we see that our
bound in \cref{eq:ringQ_psd} is still of smaller order. More
importantly, however, is that to guarantee $\tilde{Q}$ negligible
compared to $EU \Lambda^{-1}$ in \cref{eq:lihua_comparison} one needs
\begin{equation}
\label{eq:discussion2} 
\delta_r (r^{1/2} + \log^{1/2}{n}) = \omega((n \rho_n)^{1/2} \lambda_r^{1/2})
\end{equation} which, similar to the conditions in
\cref{eq:discussion0,eq:discussion1}, may be infeasible as $n$
increases.  Given that our inference tasks generally require choosing
$r$ to grow with $n$, the ability to leverage less stringent
conditions (such as \cref{eq:discussion4}) leads to a faster growth
rate for $r$ which then yields more accurate estimation of $P$ in
\cref{sec:entrywise_approximation} and larger asymptotic power for
testing equality of latent positions in \cref{sec:two_sample}.

Next, we compare our results with those in \cite{graph_root} wherein
the author considered the notion of graph root distribution. More
specifically, let $\kappa$ be a graphon (equivalently, a symmetric
measurable function from $[0,1]^{2}$ to $[0,1]$), and suppose the
integral operator associated with $\kappa$ (see \cref{eq:integral_op})
satisfies \cref{eq:hypothesis_H}. Let $\xi_{1} \geq \xi_{2} \geq \dots
> 0$ be the enumeration of the positive eigenvalues of $\mathscr{K}$,
and, with a slight abuse of notations, let $\phi_{1}, \phi_{2}, \dots$
denote the corresponding eigenfunctions. Similarly, let $\zeta_{1}
\geq \zeta_{2} \geq \dots > 0$ be the enumeration of the {\em moduli}
of the negative eigenvalues of $\mathscr{K}$, and let $\psi_{1},
\psi_{2}, \dots$ denote the corresponding eigenfunctions. Suppose also
that there exists positive numbers $c_{1} \leq c_{2}$, $1 < \alpha
\leq \beta$ such that for all $j \geq 1$
\begin{equation*}
    c_1 j^{-\alpha}
    \leq
    \min\{\xi_j, \zeta_j\}
    \leq
    \max\{\xi_j, \zeta_j\}
    \leq
    c_2 j^{-\alpha},
    \quad
    \text{and}
    \quad\min\{\xi_j - \xi_{j+1}, \zeta_j - \zeta_{j+1} \}
    \geq
    c_1 j^{-\beta};
\end{equation*}
see \cite[Assumption~A2]{graph_root}. If
$n \rho_n \rightarrow \infty$ and $r = o(\min\{n^{1/(2\beta + \alpha)}, (n \rho_n)^{1/(2 \beta)}\})$, then by \cite[Theorem~4.5]{graph_root} we have
\begin{gather}
    \label{eq:graph_root1}
    \|\hat{U} |\hat{\Lambda}|^{1/2} - U |\Lambda|^{1/2} \|_{F}
    =
    O_{p}\left(r^{(\beta - \alpha + 1)/2}\right), \\
    \label{eq:graph_root2}
    n^{-1} \|\rho_n^{-1/2} \hat{U} |\hat{\Lambda}|^{1/2} - Z_r \|_{F}^2
    =
    O_{p}\left(r^{2 \beta - \alpha + 1} (n \rho_n)^{-1} + n^{-(\alpha - 1)/(2\beta)} + r^{2 \beta+1} n^{-1}\right)
    .
\end{gather}
Here, $ \hat{U} |\hat{\Lambda}|^{1/2}$ and $U |\Lambda|^{1/2}$ are now $n \times 2r$ matrices whose columns are the scaled eigenvectors of $A$ and $P$, corresponding to the $r$ largest positive eigenvalues and $r$ largest (in modulus) negative eigenvalues, while $Z_r$ is the $n \times 2r$ matrix with rows of the form
\begin{equation*}
    \left(\sqrt{\xi_1} \phi_1(X_i), \sqrt{\xi_2} \phi_2(X_i),
    \dots,
    \sqrt{\xi_{r}} \phi_{r}(X_i), \sqrt{\zeta_1} \psi_1(X_i),
    \dots,
    \sqrt{\zeta_r} \psi_r(X_i)\right)
    .
\end{equation*}
Above, the notation $O_{p}(\cdot)$ denotes a big-Oh bound that holds with probability converging to one as $n \rightarrow \infty$ and where, for ease of exposition, we have ignored all orthogonal transformations $W$ in the bounds. In contrast, applying our \cref{thm:general} to the current setting yields 
\begin{equation*}
    \begin{split} \|\hat{U} |\hat{\Lambda}|^{1/2} - U |\Lambda|^{1/2} \|_{2 \to \infty} 
    &\lesssim
    \frac{r^{\alpha/2}(r^{1/2}
    +
    \log^{1/2}{n})}{n^{1/2}}
    +
    \frac{r^{2\beta + \alpha/2} + r^{\beta + (\alpha+1)/2} \log^{1/2}{n}}{n \rho_n^{1/2}} \\
    &\qquad+
    \frac{r^{\beta + \alpha/2} \log^{3/2}{n}}{n^{3/2} \rho_n}
\end{split} 
\end{equation*}
with high probability. The first term in our $2 \to \infty$ norm bound is roughly of order $n^{-1/2}$ times smaller than the corresponding Frobenius norm error in \cref{eq:graph_root1}; the extra $\log^{1/2}{n}$ term can be attributed to the bound holding with non-asymptotic high probability as opposed to with probability converging to one. Meanwhile the remaining terms in our $2 \to \infty$ norm bound have denominators that are of order larger than $n^{1/2}$ and can be made negligible for appropriate choices of $r$, e.g., $r = (n \rho_n)^{1/(4\beta - 2)}$.

Our results do not include bounds of the form in \cref{eq:graph_root2}, as our paper is mainly concerned with the row-wise difference of the leading eigenvectors of $A$ compared to that of $P$, while \cref{eq:graph_root2} is, intrinsically, about the difference between the leading eigenvectors of $P$ compared to the {\em eigenfunctions} of $\mathscr{K}$. Nevertheless, if $\kappa$ is positive semidefinite, then we can apply the ideas in \cref{rem:uniform_1} below to also obtain $2 \to \infty$ norm bounds for $\hat{U} \hat{\Lambda}^{1/2} - Z_r$, where now both matrices are of size $n \times r$ as $\mathscr{K}$ cannot have negative eigenvalues. The case when $\kappa$ is indefinite is more complicated and we leave it for future investigation.

Finally, we note that our results explicitly allow for repeated eigenvalues in both $P$ and $\mathscr{K}$. In contrast, \cite{graph_root} requires a gap of order at least $j^{-\beta}$ between the $j$-th and $j+1$-th eigenvalue for all $j \geq 1$. While \cite{graph_root} mentioned that repeated eigenvalues can be handled using the arguments therein, they did not explicitly do so, and to do so would require significant additional technicalities. Nevertheless, we believe that precise statements of results under the repeated eigenvalues setting are necessary in order to make them more widely applicable, especially in the infinite rank setting. Indeed, as \cref{rem:eigenval_gap_plot} clearly shows, it is only meaningful to make assumptions about the rate of decay of eigenvalues but not their gaps.

\section{Implications for inference}
\label{sec:inference}
\subsection{Entrywise bound for $P$}
\label{sec:entrywise_approximation}
We now apply the results in \cref{sec:main} to obtain entrywise error bounds for estimating the edge probability matrix $P$. 
Note that, for ease of exposition, the conditions on $r$ are stated in terms of the eigenvalues $\{\lambda_j\}_{j \ge 1}$ of $P$. In practice, we can reformulated these conditions in terms of the eigenvalues $\{\hat{\lambda}_{j}\}_{j \ge 1}$ of $A$. The details are straightforward (see e.g., \cref{cor:main1}) and thus omitted. 

\begin{corollary}
\label{cor:entrywise}
    Assume the setting and notations in \cref{thm:general}. Also, suppose $r$ is chosen so that the conditions in 
    \cref{eq:dominant_indefinite1} are satisfied. 
    Define $\hat{P}_{r} = \hat{U} \hat{\Lambda} \hat{U}^{\top}$ and $P_{r} = U \Lambda U^{\top}$. 
    Then
    \begin{gather}
   \label{eq:phat-p_max_part1b_indefinite2}
    \rho_{n}^{-1}\|\hat{P}_{r} - P\|_{\max} \leq \rho_n^{-1} \|P_r - P\|_{\max}
    +
    O\left(|\lambda_r|^{-1} (n \rho_n)^{1/2} (r^{1/2} + \log^{1/2}{n}\right),
\end{gather}
which holds with high probability. Furthermore, if $\kappa$ is positive semidefinite then, under the conditions in \cref{eq:condition_rem4}, the above bound can be improved to
\begin{gather}
       \label{eq:phat-p_max_part2}
      \rho_{n}^{-1}\|\hat{P}_{r} - P\|_{\max}
    \leq
    \rho_{n}^{-1}\|P_{r} - P\|_{\max}
    +
    O\left(\lambda_{r}^{-1/2} (r^{1/2} + \log^{1/2} n)\right)
    ,
\end{gather}
which holds with high probability.
\end{corollary}

We now compare \cref{eq:phat-p_max_part2,eq:phat-p_max_part1b_indefinite2} with the results in \cite{xu_spectral} for $\|\hat{P}_{r} - P\|_{F}$. Using the notations in the present paper, \cite[Theorem~1]{xu_spectral} establishes that with high probability,
\begin{equation}
    \label{eq:xu_comparison1}
    (n \rho_{n})^{-2} \|\hat{P}_{r} - P\|_{F}^{2}
    \leq
    \frac{C_{1} r}{n \rho_{n}}
    +
    \frac{C_{2}}{(n \rho_{n})^{2}} \sum_{k > r} \lambda_{k}^{2}
    ,
\end{equation}
where $C_{1}$ and $C_{2}$ are positive constants depending on the eigenvalue threshold for choosing $r$ in universal singular value thresholding (USVT) \cite{chatterjee,xu_spectral}. The first term on the right hand side of \cref{eq:xu_comparison1} is due to the effect of estimating $\hat{P}_{r}$ by $P_{r}$ and is analogous to the terms $O(\lambda_{r}^{-1} (r + \log n))$ and $O(\lambda_{r}^{-2} n \rho_{n} (r + \log n)$ obtained by squaring the second term in \cref{eq:phat-p_max_part2,eq:phat-p_max_part1b_indefinite2}, respectively, with the main differences being either the use of $\lambda_{r}^{-1}$ or $\lambda_{r}^{-2} n \rho_{n}$ in place of $(n \rho_{n})^{-1}$. These differences are attributable to the fact that $(n \rho_{n})^{-2} \|\hat{P}_{r} - P_{r}\|_{F}^{2}$ quantifies the {\em average} squared entrywise error between $\hat{P}_{r}$ and $P$, whereas the second term in \cref{eq:phat-p_max_part2,eq:phat-p_max_part1b_indefinite2} is for the {\em maximum} entrywise errors. Indeed, the factors $\lambda_{r}^{-1}$ and $\lambda_{r}^{-2} n \rho_{n}$ appear due to the fact that although $\|U\| = 1$, the individual rows of $U$ can be heterogeneous, and $\lambda_{r}/(n \rho_{n})$ ($\lambda_{r}^{2}/(n \rho_{n})^{2}$, resp.) serves as a measure for the degree of heterogeneity between different rows of $U$ when $P$ is positive semidefinite (indefinite, resp.). Finally, the term $C_{2} (n \rho_{n})^{-2} \sum_{k > r} \lambda_{k}^{2}$ in \cref{eq:xu_comparison1} is an upper bound for $(n \rho_{n})^{-2} \|P_{r} - P\|_{F}^{2}$ and serves as the analogue of $\rho_{n}^{-1}\|P_{r} - P\|_{\max}$ in \cref{eq:phat-p_max_part2,eq:phat-p_max_part1b_indefinite2}.

\begin{remark}[On edge probability estimation in network analysis]
    \label{rem:graphon_estimate}
    There is a sizable existing literature on the estimation of $P$ for independent edge random graphs. See \cite{wolfe13:_nonpar,yang14:_nonpar,chatterjee,airoldi13:_stoch,klopp,pensky_graphon1,gao,xu_spectral,zhang_biometrika,qin_nips} for an incomplete list of references. In particular, the two most common approaches are (i) fitting, via either (penalized) least squares or maximum likelihood estimation, a stochastic blockmodel (SBM) to the adjacency matrix $A$ \cite{wolfe13:_nonpar, klopp, gao, airoldi13:_stoch,pensky_graphon1} and (ii) (universal) singular value thresholding (USVT) \cite{xu_spectral, chatterjee,yang14:_nonpar}. It has been noted that, under certain conditions on the sparsity of the graphs, estimators obtained by fitting SBMs have Frobenius norm errors that are minimax optimal over the class of Hölder-continuous link functions $\kappa$ on $[0,1] \times [0,1]$; see for example \cite{gao,klopp} or \cite[Appendix~A]{xu_spectral}. These estimators, however, can be computationally infeasible as their running time can be exponential in $n$. In contrast, estimators based on singular value thresholding, and thus (low rank) matrix factorizations, have running time of order $O(n^{3})$. Furthermore, while the Frobenius norm error of USVT estimates are not minimax optimal for the class of Hölder-continuous $\kappa$, \cite{luo_gao} showed that they are nevertheless optimal for stochastic blockmodels among the class of all low-degree polynomial estimators. Subsequently, \cite{lei_graphon_rate} showed that the Frobenius norm error of USVT is, up to a logarithmic factor, also minimax optimal among all $\kappa$ for which $|\mu_{k}| \lesssim k^{-\alpha}$ for some $\alpha \geq 1$ and any $k \geq 1$. In other words, USVT is minimax optimal whenever the eigenvalues of $\mathscr{K}$ have polynomial or exponential rate of decay.
    
    Our results presented in the current section provide further compelling evidence for the use of USVT, particularly since it leads to entrywise bounds and, potentially, confidence intervals for $P$ (see \cref{rem:uniform_1,rem:normal}). Both of these results are currently not available for other type of estimators due to the fact that Frobenius norm error bounds are not suitable for inferring row-wise or entrywise behavior of estimates. 
\end{remark}

\begin{remark}[Entrywise bounds and rates of convergence]
\label{rem:uniform_1}
If $\kappa$ is continuous and positive semidefinite, then we also have the following bound for $\rho_{n}^{-1}\|P_{r} - P\|_{\max}$. 
\begin{proposition}
\label{prop:mercer_entrywise}
Consider the setting in \cref{thm:psd} where $\kappa$ is a continuous, positive semidefinite kernel and suppose $\mu_r > \mu_{r+1}$. Then    
\begin{equation}
    \label{eq:Pr-P_definite1}
    \begin{split}
    \rho_n^{-1} \|P_r - P\|_{\max}
    &\lesssim 
    \frac{\log^{1/2}{n}}{n^{1/2}(\mu_r - \mu_{r+1})}
    +
    \max_{ij} \left|\sum_{k > r} \mu_k \phi_k(X_i) \phi_k(X_j)\right|
\end{split}
\end{equation}
with high probability. 
\end{proposition}
The first term on the right hand side of \cref{eq:Pr-P_definite1} converges to zero as $r \to \infty$, provided that $\mu_r - \mu_{r+1} = \omega(n^{-1/2} \log^{1/2}{n})$. As the only accumulation point for the $\{\mu_i\}_{i \ge 1}$ is at zero, such a sequence for $r$ always exists. The second term also converges to zero due to the uniform convergence in Mercer's theorem (see, e.g., \cite[Theorem~4.49]{steinwart08:_suppor_vector_machin}) but there is no {\em a priori} explicit convergence rate unless we put additional assumptions on $\kappa$. For example, if $\{X_i\} \subset \mathbb{R}^{d}$ and $\kappa$ is $m$ times continuously differentiable, then
\begin{equation*}
    \left|\max_{ij} \sum_{k > r} \mu_k \phi_k(X_i) \phi_k(X_j)\right|
    \lesssim
    \left(\sum_{k > r} \mu_k^2 \right)^{m/(2m+2d)}
    .
\end{equation*}
See \cite[Theorem~1.2]{mercer_convergence} for more details. As a related example, if $\mu_r = O(r^{-\alpha})$ and $\|\phi_r\|_{\infty} = O(r^{\beta})$ for constants $\alpha \geq 2 \beta + 1 > 0$ then \cite[Theorem~1]{modell}
yields the high probability bound
\[
    \rho_n^{-1}\|P_r - P\|_{\max}
    =
    O\left(n^{-(\alpha-1)/\alpha} \log n\right)
    .
\]
Combining \cref{eq:phat-p_max_part2,eq:Pr-P_definite1} together with the Cauchy--Schwarz inequality gives
\[
    \rho_n^{-1} \|\hat{P}_r - P\|_{\max} \lesssim \frac{r^{1/2} + \log^{1/2}{n}}{\lambda_r^{1/2}} + \frac{\log^{1/2}{n}}{n^{1/2}(\mu_r - \mu_{r+1})} + \sup_{x} \sum_{k > r} \mu_k \phi_k^{2}(x)
    .
\]
Of note, the first term in the above bound depends on both $n$ and graph sparsity $\rho_n$ (through $\lambda_r$), the second term depends on $n$ and the eigenvalue gap $\mu_r - \mu_{r+1}$ but not on $\rho_n$, and the last term is independent of $n$. 
\end{remark}

\begin{remark}[Towards entrywise normality and inference]
\label{rem:normal}
Normal approximations and confidence intervals for the entries of $\hat{P}_r$ as estimates for $P$ are much more difficult to derive in the current setting compared to when $P$ is low-rank; see \cite[Theorems 4.10 and 4.11]{spectral_chen} for low-rank examples. More specifically, assume for simplicity that $\kappa$ is positive semidefinite. Then, using a similar derivations as in \cref{eq:uniform_expansion1}, we have
\begin{equation*}
    \rho_n^{-1}(\hat{P}_r - P)
    =
    \rho_n^{-1} (A - P) UU^{\top}
    +
    \rho_n^{-1} UU^{\top}(A - P)
    +
    \rho_n^{-1}(P_r - P) + R
    ,
\end{equation*}
where $R$ is a residual matrix. Let $M = UU^{\top}$ and denote by $e_{ij}$ and $m_{ij}$ the $ij$-th entry of $E = A - P$ and $M$, respectively. Let $\zeta_{ij}$ be the $ij$-th element of $E UU^{\top} + UU^{\top}E$. Then, for $i \not = j$, we have
\begin{equation*}
    \begin{split}
    \rho_n^{-1} (\zeta_{ij} + \zeta_{ji})
   &=
    \rho_n^{-1} \sum_{k \not = j} e_{ik} m_{kj} + \rho_n^{-1}\sum_{k \not = i} e_{kj} m_{ik} + \rho_n^{-1}e_{ij} (m_{ii} + m_{jj})
    ,
  \end{split}
\end{equation*}
which, conditioning on $P$, is a sum of independent mean zero random variables. Next, define
\begin{equation}
    \label{eq:sigma_ij}
    \sigma_{ij}^2
    =
    \rho_n^{-2}
    \left(\sum_{k} [\, p_{ik}(1 - p_{ik}) m_{kj}^2 +  p_{kj}(1 - p_{kj}) m_{ik}^2] + 2 p_{ij}(1 - p_{ij}) m_{ii} m_{jj}\right)
    .
\end{equation}
Then, for a fixed $(i,j)$ pair, we can apply the Lindeberg--Feller central limit theorem to show
\begin{gather*}
    \sigma_{ij}^{-1} \zeta_{ij}
    \rightsquigarrow
    N(0,1)
    \qquad
    \text{as }
    n
    \to
    \infty.
\end{gather*}
Suppose there exists a constant $c_0 > 0$ such that $\kappa(x,y) \geq c_0$ for all $(x,y) \in \Omega \times \Omega$. Then,
\begin{equation*}
    \sigma_{ij}^2
    \geq
    \rho_n^{-2} \times \rho_n c_0 (1 - \rho_n c_0) (m_{jj} + m_{ii})
    ,
\end{equation*}
and $\sigma_{ij}$ converges to zero at rate $\rho_n^{-1/2} \|U\|_{2 \to \infty} \asymp \lambda_r^{-1/2}$ as $n \rightarrow \infty$. Furthermore, there are also plug-in estimators $\hat{\sigma}_{ij}$ for which $\hat{\sigma}_{ij}/\sigma_{ij} \rightarrow 1$ in probability so that, by Slutsky's theorem, we can replace with $\hat{\sigma}_{ij}$ in the above distributional convergence. 
In contrast, the $ij$-th entry of $P_r - P$ induces a bias term $b_{ij} = p^{(r)}_{ij} - p_{ij}$ which, from the discussion in \cref{rem:uniform_1}, converges to zero but at a possibly arbitrarily slow rate. In other words, while we have a normal 
approximation of the form 
\begin{equation}
\label{eq:clt_entrywise}
    \sigma_{ij}^{-1}(\hat{p}^{(r)}_{ij} + b_{ij} - p_{ij}) 
    \rightsquigarrow
    \mathcal{N}(0,1)
    \qquad
    \text{as }
    n
    \to
    \infty,
\end{equation}
it might not always lead to meaningful inference results. For example, we cannot directly apply \cref{eq:clt_entrywise} to construct confidence interval for $b_{ij}$ due to the unknown bias term $b_{ij}$. Furthermore, while we can try to estimate $b_{ij}$, it is not clear if $\sigma_{ij}^{-1} |\hat{b}_{ij} - b_{ij}| \rightarrow 0$ as $n \rightarrow \infty$. Indeed, as $P$ is unknown, we can only estimate each {\em individual} $b_{ij}$ using the adjacency matrix $A$ but since $b_{ij}$ depends on the eigenvalues and eigenvectors of $P$ {\em not in} $U$ and $\Lambda$, we cannot directly control their estimation errors using the results in the current paper. 
We thus leave the open question of constructing confidence intervals for $p_{ij}$ to future work. 
\end{remark}

\subsection{Testing for equality of latent positions}
\label{sec:two_sample}
We now consider the problem of determining whether or not two given vertices $i$ and $j$ have the same latent positions. In other words, we are interested in testing the hypotheses
\begin{equation}
    \label{eq:two-sample-H0H1}
    \mathbb{H}_{0}
    :
    X_{i} = X_{j}
    \quad
    \text{ versus }
    \quad
    \mathbb{H}_{1}
    :
    X_{i} \neq X_{j}
    .
\end{equation}
This problem had been previously studied in \cite{fan2019simple,du_tang} under the simpler setting wherein $P$ is low-rank, i.e., when the kernel $\kappa$ is a finite-rank kernel of the form $\kappa(X_i, X_j) = X_i^{\top} M X_j$ for some $d \times d$ symmetric matrix $M$. The authors of \cite{fan2019simple,du_tang} showed that, with $r = d$, the test statistic based on the Mahalanobis distance between $\hat{U}_i$ and $\hat{U}_j$ (i.e., the $i$ and $j$-th row of $\hat{U}$) converges to a central $\chi^2$ (non-central $\chi^2$, resp.) with $r$ degrees of freedom under the null hypothesis (local alternative hypothesis, resp.). As the Mahalanobis distance is invariant with respect to invertible transformations, the same limiting distribution holds for the difference between $\hat{\Lambda}^{1/2} \hat{U}_i$ and $\Lambda^{1/2} \hat{U}_j$.

We now extend these results to the setting of our paper by allowing $r$ to grow with $n$. Our test statistic, however, will be based on the {\em Euclidean distance} between $\hat{\Lambda} \hat{U}_i$ and $\hat{\Lambda} \hat{U}_j$, as opposed to either form of the Mahalonobis distance described above. The rationale for using the Euclidean distance will be clarified in \cref{rem:comp_potential}.

Let $\hat{\xi}_i \in \mathbb{R}^{r}$ and $\xi_i \in \mathbb{R}^{r}$ be the $i$-th row of $\hat{U} \hat{\Lambda}$ and $U \Lambda$, respectively. Fix an arbitrary pair $\{i,j\}$ and suppose that $X_i = X_j$. Then, $\xi_i = \xi_j$, and by \cref{thm:generalU_ULambda}, we have 
\begin{equation*}
    \begin{split}
    W^{\top}(\hat{\xi}_i - \hat{\xi}_j) 
    =
    (E_i - E_j) U + \gamma_{ij}
\end{split}
\end{equation*}
for some $r \times r$ orthogonal matrix $W$, where $\gamma_{ij}$ is a residual lower-order term while $E_i$ and $E_j$ are the $i$-th and $j$-th row of $A - P$, respectively.
We then have
\begin{equation}
    \label{eq:quadratic_test}
    \begin{split}
    \|\hat{\xi}_i - \hat{\xi}_j\|^2
    =
    \|W^{\top}(\hat{\xi}_i - \hat{\xi}_j)\|^2 
    =
    (1 + o(1)) \|(E_i - E_j) U\|^2
    =
    (1+o(1)) \zeta^{\top} \mathcal{D}U  U^{\top} \mathcal{D} \zeta
    ,
\end{split}
\end{equation}
where $\zeta$ is now a vector whose components are independent sub-Gaussian random variables with mean zero and variance one, and $\mathcal{D}$ is a $n \times n$ diagonal matrix whose diagonal entries are $d_{k} = (2p_{ik}(1 - p_{ik}))^{1/2}$ for $1 \leq k \leq n$. For notational simplicity, we have dropped the indices $\{i,j\}$ from the matrix $\mathcal{D}$. 
The following theorem establishes that $\|\hat{\xi}_i - \hat{\xi}_j\|$, properly translated and scaled, converges in distribution to a weighted sum of independent $\chi^2_1$ random variables under the null hypothesis.

\begin{theorem}[Asymptotic null distribution for testing equality of latent positions]
    \label{thm:quadratic_test}
    Let $A \sim \mathrm{LPG}(\kappa, F; \rho_{n})$ be a graph on $n$ vertices generated according to \cref{def:lpg} with sparsity $\rho_n$. Denote the associated 
    latent positions by $\{X_1, \dots, X_n\}$. Set $r = r(n)$ to be any positive integer such that $\delta_r = \omega(\sqrt{r n \rho_n \log n})$. Let $\hat{U}$ and $U$ denote the $n \times r$ matrices whose orthonormal columns are the eigenvectors corresponding to the $r$ largest  in magnitude eigenvalues of $A$ and $P$, respectively. Let $M^* = \mathcal{D} U U^{\top} \mathcal{D}$, where $\mathcal{D}$ is a diagonal matrix with diagonal entries
    \begin{equation*}
    d_{k} = (p_{ik}(1 - p_{ik}) + p_{jk}(1 - p_{jk}))^{1/2},
    \quad
    \text{for }
    1 \leq k \leq n
    .
    \end{equation*}
    Then, under the null hypothesis, $\mathbb{H}_{0}: X_i = X_j$ in \cref{eq:two-sample-H0H1}, it holds that
    \begin{equation}
    \label{eq:conv_quadratic}
        \frac{\|(A \hat{U})_i - (A \hat{U})_{j}\|^2 - \mathrm{tr} \, M^*}{\|M^*\|_{F}}
        =
        \frac{1}{\|M^*\|_{F}}\sum_{s=1}^{r} (Z_s^2 - 1) \lambda_s(M^*) 
        +
        \epsilon_n
        ,
    \end{equation}
    where $\epsilon_n \rightarrow 0$ in probability as $n \rightarrow \infty$, and $Z_1, Z_2, \dots, Z_r$ are i.i.d. standard normals.
    
    Furthermore, if $\kappa$ is an infinite rank kernel, then $r(n) \rightarrow \infty$ in probability as $n \rightarrow \infty$ and
    \begin{equation}
    \label{eq:conv_normal}
        \frac{\|(A \hat{U})_i - (A \hat{U})_{j}\|^2 - \mathrm{tr} \, M^*}{\|M^*\|_{F}}
        \rightsquigarrow
        N(0,2)
        \qquad
        \text{as }
        n
        \to
        \infty
        .
    \end{equation}
\end{theorem}
\cref{eq:conv_quadratic,eq:conv_normal} yield large-sample approximations for $\|(A \hat{U})_{i} - (A \hat{U})_{j}\|^2$ as a weighted sum of independent $\chi^2_1$ random variables, provided that we can find estimates for $\mathrm{tr}\, M^*$ and $\|M^*\|_{F}$. The following lemma illustrates one valid approach for doing so.
\begin{lemma}[Consistent estimation of centering and scaling terms]
    \label{lem:estimate}
    Let $\hat{\mathcal{D}}$ a be diagonal matrix with diagonal entries
    $$
    \hat{d}_{k}
    =
    |a_{ik} - a_{jk}|
    \quad
    \text{for }
    1 \leq k \leq n
    .
    $$
    Suppose $r$ is chosen such that
    \begin{equation}
    \label{eq:select_r_lemma_test}
    \delta_r = \omega\left( \max\left\{ \log^{3/2}{n}, \sqrt{r n \rho_n \log n}, (n \rho_n)^{3/4}/(r^{1/4} + \log^{1/4}n) \right\}\right)
    \end{equation}
    is satisfied. 
    Next, define
    $$
    \hat{\theta}
    =
    \mathrm{tr} \, \hat{U}^{\top} \, \hat{\mathcal{D}}^2 \hat{U}
    =
    \|\hat{\mathcal{D}} \hat{U}\|_{F}^2,
    \quad
    \text{and}
    \quad
      \hat{\sigma}
    =
    \|\hat{\mathcal{D}} \hat{U} \hat{U}^{\top} \hat{\mathcal{D}}\|_{F}
    .
    $$
    Then, under the setting of \cref{thm:quadratic_test} together with the condition in \cref{eq:select_r_lemma_test}, we have
    $$
    \frac{\hat{\theta} - \mathrm{tr} \, M^*}{\|M^*\|_{F}}
    \overset{\operatorname{p}}{\longrightarrow}
    0,
    \quad
    \text{and}
    \quad
    \frac{\hat{\sigma}}{\|M^*\|_{F}}
    \overset{\operatorname{p}}{\longrightarrow}
    1,
    \qquad
    \text{as }
    n
    \to
    \infty
    .
    $$
\end{lemma}

\begin{remark}[Selection of the embedding dimension $r$]
We remark that the condition for $r$ in \cref{eq:select_r_lemma_test} is more stringent than that in \cref{thm:quadratic_test}. The additional constraint manifests itself when we apply \cref{thm:general} or \cref{thm:generalU_ULambda} to estimate $U$ as it allows us to guarantee that the residual terms $\mathring{Q}$ can be ignored. Although we can in principle relax \cref{eq:select_r_lemma_test} to only require $\delta_r = \omega(\sqrt{r n \rho_n \log n})$, this weaker condition will require a more delicate analysis of $\|\hat{U}^{\top} (\hat{\mathcal{D}}^2 - \mathcal{D}^2) \hat{U}\|_{F}$. More specifically, if we replace $\hat{U}$ by some quantity $\tilde{U} \approx \hat{U}$ that is independent of $\hat{\mathcal{D}}$, then we can bound $\|\tilde{U}^{\top}  (\hat{\mathcal{D}}^2 - \mathcal{D}^2) \tilde{U}\|_F$ via Bernstein's inequality. As $\hat{\mathcal{D}}$ depends only on the $i$-th and $j$-th row of $P$, a natural candidate for $\tilde{U}$ would be based on a leave-two-out analysis that replaces the entries in the $i$-th and $j$-th rows (as well as those in the $i$-th and $j$-th columns) of $A$ with independent copies thereof. The technical details, while potentially interesting, are also rather tedious. Since our primary objective in \cref{lem:estimate} is to present simple estimators for $\mathrm{tr} \, M_*$ and $\|M_*\|_{F}$, we leave proving \cref{lem:estimate} under weaker conditions to the interested reader. 
\end{remark}

Combining \cref{thm:quadratic_test,lem:estimate} yields the following testing procedure for the hypothesis testing problem in \cref{eq:two-sample-H0H1}
Notably, the proposed test statistic only involves $A$.

\begin{corollary}[Testing equality of latent positions with data-driven rank selection]
\label{cor:test}
Consider the setting in \cref{thm:quadratic_test}. Let $\hat{r} = \hat{r}(n)$ be defined as
\begin{equation}
    \label{eq:selection_r_test}
    \hat{r}
    =
    \argmax\Bigl\{j \colon |\hat{\lambda}_j| - |\hat{\lambda}_{j+1}| \geq \max\Bigl( \log^{7/4}(n), j^{1/2} \sqrt{d_{\mathrm{ave}}(A)} \log^{3/4}{n}, (d_{\mathrm{ave}}(A))^{3/4} \Bigr) \Bigr\}
    ,
\end{equation}
where $d_{\mathrm{ave}}(A) = n^{-1} \sum_{i} \sum_{j} a_{ij}$ is the average degree of $A$. Denote by $\hat{U}$ the $n \times \hat{r}$ matrix whose columns are the eigenvectors corresponding to the $\hat{r}$ largest eigenvalues of $A$. For any $i \neq j$, define the test statistic
\begin{align*}
    T(\hat{X}_i, \hat{X}_j)
    =
    \frac{\| (A_{i} - A_j)^{\top}\hat{U} \|^2 - \|\hat{\mathcal{D}} \hat{U}\|_{F}^2 }{\|\hat{U}^{\top} \hat{\mathcal{D}}^2 \hat{U} \|_{F}} 
    =
    \frac{\sum_{k \not = \ell} (a_{ik} - a_{jk})(a_{i\ell} - a_{j \ell}) \hat{U}_k^{\top} \hat{U}_{\ell}}{\bigl(\sum_{k} \sum_{\ell} (a_{ik} - a_{jk})^2(a_{i \ell} - a_{j \ell})^2 (\hat{U}_{k}^{\top} \hat{U}_{\ell})^2\bigr)^{1/2}}
    ,
\end{align*}
where $A_i$ and $A_j$ are the $i$-th and $j$-th columns of $A$, respectively, and where $\hat{\mathcal{D}}$ is the diagonal matrix with diagonal entries $\hat{d}_k = a_{ik} - a_{jk}$. Then, under the null hypothesis, $\mathbb{H}_0 : X_i = X_j$ in \cref{eq:two-sample-H0H1}, we have
\begin{equation}
    \label{eq:approximation_adaptive}
    T(\hat{X}_i, \hat{X}_j)
    =
    \frac{1}{\|U^{\top} \mathcal{D}^2 U\|_{F}} \sum_{s=1}^{\hat{r}}(Z_s^2 - 1) \lambda_s(U^{\top} \mathcal{D}^2 U)
    +
    \epsilon_n
    ,
\end{equation}
where $\epsilon_n \to 0$ in probability as $n \rightarrow \infty$, $\mathcal{D}^2$ is the diagonal matrix with diagonal entries $d_{k} = 2p_{ik}(1 - p_{ik})$ for $1 \le k \le n$, $U$ is the $n \times \hat{r}$ matrix whose columns are the eigenvectors corresponding to the $\hat{r}$ largest eigenvalues of $P$, and $Z_1, \dots, Z_{\hat{r}}$ are independent $N(0,1)$ random variables.  

Furthermore, the sequence $\hat{r} = \hat{r}(n)$ for $n \ge 1$ is adaptive to the rank of $\kappa$, i.e.,
\begin{enumerate}
    \item If $\mathrm{rk}(\kappa) = r^* < \infty$, then as $n \to \infty$, both $\hat{r} \rightarrow r^*$ in probability
        and 
      \begin{equation}
        \label{eq:test_stat_finite_rank}
        T(\hat{X}_i, \hat{X}_j)
        \rightsquigarrow
        \frac{1}{\| U_{r^*}^{\top} \mathcal{D}^2 U_{r^*}\|_{F}} \sum_{s=1}^{r^*} (Z_s^2 - 1) \lambda_s(U_{r^*}^{\top} \mathcal{D}^2 U_{r^*})
        ,
    \end{equation}
    where $U_{r^*}$ is the $n \times r^*$ matrix whose columns are the eigenvectors corresponding to the non-zero eigenvalues of $P$.
   \item If $\mathrm{rk}(\kappa) = \infty$, then as $n \to \infty$, both $\hat{r} \rightarrow \infty$ in probablity and 
     \begin{equation}
        \label{eq:test_stat_infinite_rank}
        T(\hat{X}_i, \hat{X}_j)
        \rightsquigarrow
        N(0,2)
        .
    \end{equation}
\end{enumerate}
\end{corollary}

\begin{remark}[Rank adaptivity of test statistic]
    The test statistic $T(\hat{X}_i, \hat{X}_j)$ in \cref{cor:test} is, up to scaling and translation, the same as $\|(A_i - A_j)^{\top} \hat{U}\|^2$. Notably, it is shown to be adaptive to the rank of $\kappa$. Per \cref{cor:test}, if $\mathrm{rk}(\kappa) = \infty$, then $T(\hat{X}_i, \hat{X}_j)$ converges in distribution to $\mathcal{N}(0,2)$. In contrast, if $\mathrm{rk}(\kappa) = r^* < \infty$, then the limiting distribution of $T(\hat{X}_i, \hat{X}_j)$ depends on the eigenvalues of the $r^* \times r^*$ matrix $U^{\top} \mathcal{D}^2 U$ and is thus less elegant, seeing as we need to additionally estimate these unknown eigenvalues unlike the setting with $\mathrm{rk}(\kappa) = \infty$. Nevertheless, for any $\alpha \in (0,1)$, the rejection region $\mathcal{R}_{\alpha} = \{ T(\hat{X}_i, \hat{X}_j) > c^{*}_{1 - \alpha} \}$ yields a test procedure with asymptotic significance level $\alpha$, where $c^{*}_{1 - \alpha}$ is the $1 - \alpha)$ quantile of
    \begin{equation}
        \label{eq:twosample_critical_estimate}
        \frac{1}{\|\hat{U}^{\top} \hat{\mathcal{D}}^2 \hat{U}\|_{F}}
        \sum_{s=1}^{r} (Z_s^2 - 1) \lambda_s(\hat{U}^{\top} \hat{\mathcal{D}}^2 \hat{U})
        .
    \end{equation}
    We emphasize that $c^{*}_{1 - \alpha}$ is a function depending only on the {\em estimate} $\hat{U}^{\top}\hat{\mathcal{D}}^2 \hat{U}$ and can be obtained by parametric bootstrap or by using the algorithm in \cite{farebrother}. 
\end{remark}

\begin{remark}[Comparisons to other potential test statistics]
\label{rem:comp_potential}
One might ask whether there are test statistics similar to $T$ in \cref{cor:test} but that are simpler to compute or admit simpler limiting distributions. To address this question, we consider the following four secondary test statistics and show that all of them are inferior to $T$, either practically or theoretically.

Specifically, define
\begin{enumerate}
    \item $T_{1}(\hat{X}_i, \hat{X}_j) = \|A_{i} - A_{j}\|^2$,
    
    \item $T_{2}(\hat{X}_i, \hat{X}_j) = \|(A_i - A_j)^{\top}\hat{D}^{-1} \hat{U}\|^2$,
    
    \item $T_{3}(\hat{X}_i, \hat{X}_j) = \|\hat{X}_i - \hat{X}_j\|^2 = \|(A_{i} - A_{j})^{\top} \hat{U} \hat{\Lambda}^{-1/2}\|^2$,
    
    \item $T_{4}(\hat{X}_i, \hat{X}_j) = (\hat{X}_i - \hat{X}_j)^{\top} (\hat{\Sigma}(\hat{X}_i) + \hat{\Sigma}(\hat{X}_j))^{-1}(\hat{X}_i - \hat{X}_j)$.
\end{enumerate}
Above, $T_1$, $T_{2}$, and $T_{3}$ are variants of the $\ell_2$ norm distance between rows $A_i$ and $A_j$, where $\hat{D}$ is a diagonal matrix whose diagonal entries are some estimates of $d_{k}$ for $1 \le k \le n$, whereas $T_4$ is the Mahalanobis distance between the (truncated) embeddings $\hat{X}_i$ and $\hat{X}_j$.

For $T_1$, we have $T_1(\hat{X}_i, \hat{X}_j) = \sum_{k} Y_k$, where $Y_k$ are independent Bernoulli random variables with success probabilities $p_{ik}(1 - p_{jk}) + p_{jk}(1 - p_{ik})$. Now, consider $p_{ik} = 1/2$, so $p_{ik}(1 - p_{jk}) + p_{jk}(1 - p_{ik}) = 1/2$, which does not depend on the value of $p_{jk}$. Therefore, if $p_{ik} \equiv 1/2$ for all $k$ then the distribution of $T_1(\hat{X}_i, \hat{X}_j)$ does not depend on $\{p_{jk}\}$ and we cannot use $T_1$ to construct a consistent test procedure.

For $T_2$, by using the same arguments as for \cref{thm:quadratic_test}, we have that
\begin{equation*}
    r^{-1/2}(\|(A_i - A_j) \mathcal{D}^{-1} U\|^2 - r) 
    \rightsquigarrow
    r^{-1/2} (\chi^2_r - r)
\end{equation*}
as $n \to \infty$ under the null hypothesis. We might therefore hope that $r^{-1/2}(T_2(\hat{X}_i, \hat{X}_j) - r)$ also converges to $r^{-1/2}(\chi^2_r - r)$, thereby yielding a simpler limiting distribution that that of $T$. However, the main difficulty with using $T_2$ is in defining an estimate $\hat{D}$. Indeed, a natural choice for the entries of $\hat{D}$ is either the plug-in estimates $\hat{d}_{k} = (\hat{p}_{ik}(1 - \hat{p}_{ik}) + \hat{p}_{jk}(1 - \hat{p}_{jk}))^{1/2}$, where $\hat{p}_{k \ell}$ are the entries of $\hat{U} \hat{\Lambda} \hat{U}^{\top}$, or $\hat{d}_{k} = a_{ik} - a_{jk}$. The first choice is problematic in that, as \cref{rem:uniform_1} shows, $\hat{p}_{k \ell}$ might converge arbitrarily slowly to $p_{k \ell}$. The second choice is also problematic as $a_{ik} - a_{jk} = 0$ with non-zero probability and so $\hat{D}^{-1}$ may be undefined. In contrast, $T$ only depends on the Frobenius norm of $\hat{U}^{\top} \hat{\mathcal{D}}^2 \hat{U}$ and is thus robust to the specific values of a large number of entries of $\hat{\mathcal{D}}$.

For $T_3$, write $M_3 = \mathcal{D} U \Lambda^{-1} U^{\top} \mathcal{D}$. Once again, using the same arguments as for \cref{thm:quadratic_test}, we have that
\begin{equation*}
    \frac{T_3(\hat{X}_i, \hat{X}_j) - \mathrm{tr} \, M_3}{\|M_3\|_F}
    \rightsquigarrow
    \frac{1}{\|M_3\|_F} \sum_{s=1}^{r} (Z_s^2 - 1) \lambda_s(M_3)
\end{equation*}
as $n \to \infty$ under the null hypothesis. The limiting distributions of $T$ and $T_{3}$ are therefore quite similar if $r$ is fixed as $n \rightarrow \infty$ but can be qualitatively different when $r \rightarrow \infty$ with $n$. Recall that if $r \rightarrow \infty$, then $\|M^*\|_{F}^{-1} \sum_{s=1}^{r} (Z_s^2 - 1) \lambda_s(M^*) \rightsquigarrow N(0,2)$ as $n \to \infty$ under the null hypothesis, which is a consequence of the Lindeberg--Feller central limit theorem. The non-zero eigenvalues of $M^*$ are all of order $\rho_n$ and thus $\max_{s \leq r} \|M^*\|_{F}^2 \lambda_s(M^*)^2 \asymp r^{-1} \rightarrow 0$. In contrast, if $s \leq r$, then $|\lambda_s(M_3)| \asymp \rho_n |\lambda_s^{-1}|$ and thus
\begin{equation*}
    \max_{s \leq r} \frac{\lambda_s(M_3)^2}{\|M_3\|_{F}^2}
    \asymp
    \frac{\lambda_r^{-2}}{\sum_{s=1}^{r} \lambda_s^{-2}} 
    \asymp
    \frac{\mu_r^{-2}}{\sum_{s=1}^{r} \mu_s^{-2}}
    ,
\end{equation*}
where as before $\{\mu_s\}_{s \ge 1}$ are the eigenvalues of $\mathscr{K}$. Now, if $\{\mu_s\}_{s \ge 1}$ decays to zero sufficiently rapidly, then $\max_{s \le r} \lambda_s(M_3)^2/\|M_3\|_{F}^2$ does not converge to zero. Concretely, for $\mu_s \asymp e^{-s}$, we have $e^{2r}/(\sum_{s=1}^{r} e^{2s}) \succsim 1$. Thus, we cannot apply the Lindeberg--Feller CLT to guarantee that $T_3$, properly scaled and translated, converges to a normal distribution.

The distribution of $T_4$ has been previously studied in \cite{du_tang,fan2019simple} when $\kappa$ is a finite rank kernel. There, it has been shown that $T_4$ converges to a central (non-central, resp.) $\chi^2_r$ under the null (local alternative, resp.) hypothesis. If instead $\kappa$ is an infinite rank kernel and $r$ is fixed then, by combining the normal approximation in \cref{cor:normal} with the analysis in \cite{du_tang}, we can show that $T_4$ still converges to a central (non-central, resp.) $\chi^2_r$ under the null (resp. local alternative) hypothesis. A precise statement of this result is left to the interested reader. We note, however, that a fixed value of $r$ does not lead to a consistent test procedure as it can fail to reject the null hypothesis when $X_i \not = X_j$ if their difference does not manifest itself in the coordinates of the leading $r$ eigenvectors of $P$. The case when $\kappa$ is an infinite rank kernel with $r \rightarrow \infty$ is more complicated as computing $T_4$ requires inverting an $r \times r$ estimated covariance matrix $\hat{\Sigma}(\hat{X}_i) + \hat{\Sigma}(\hat{X}_j)$, which then introduce additional constraints on the rate at which $r$ can grow with $n \rho_n$. We leave the precise analysis of this dependency, and its effects on the limiting distribution for $T_4$, to the interested reader.
\end{remark}

\begin{figure}[t!]
    \centering
    \subfloat[$n=4000$]{\includegraphics[width=.5\textwidth]{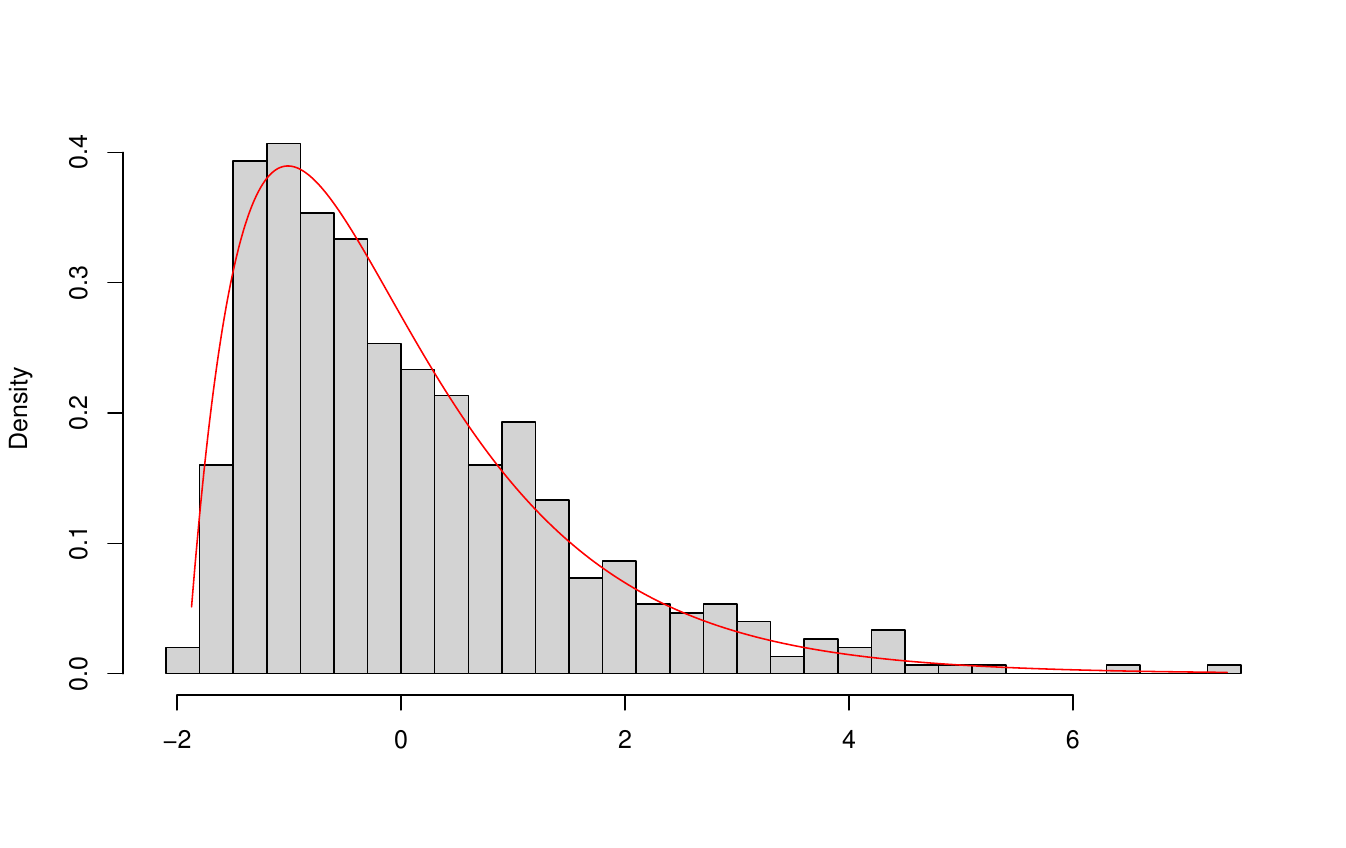}} 
    \subfloat[$n=8000$]{\includegraphics[width=.5\textwidth]{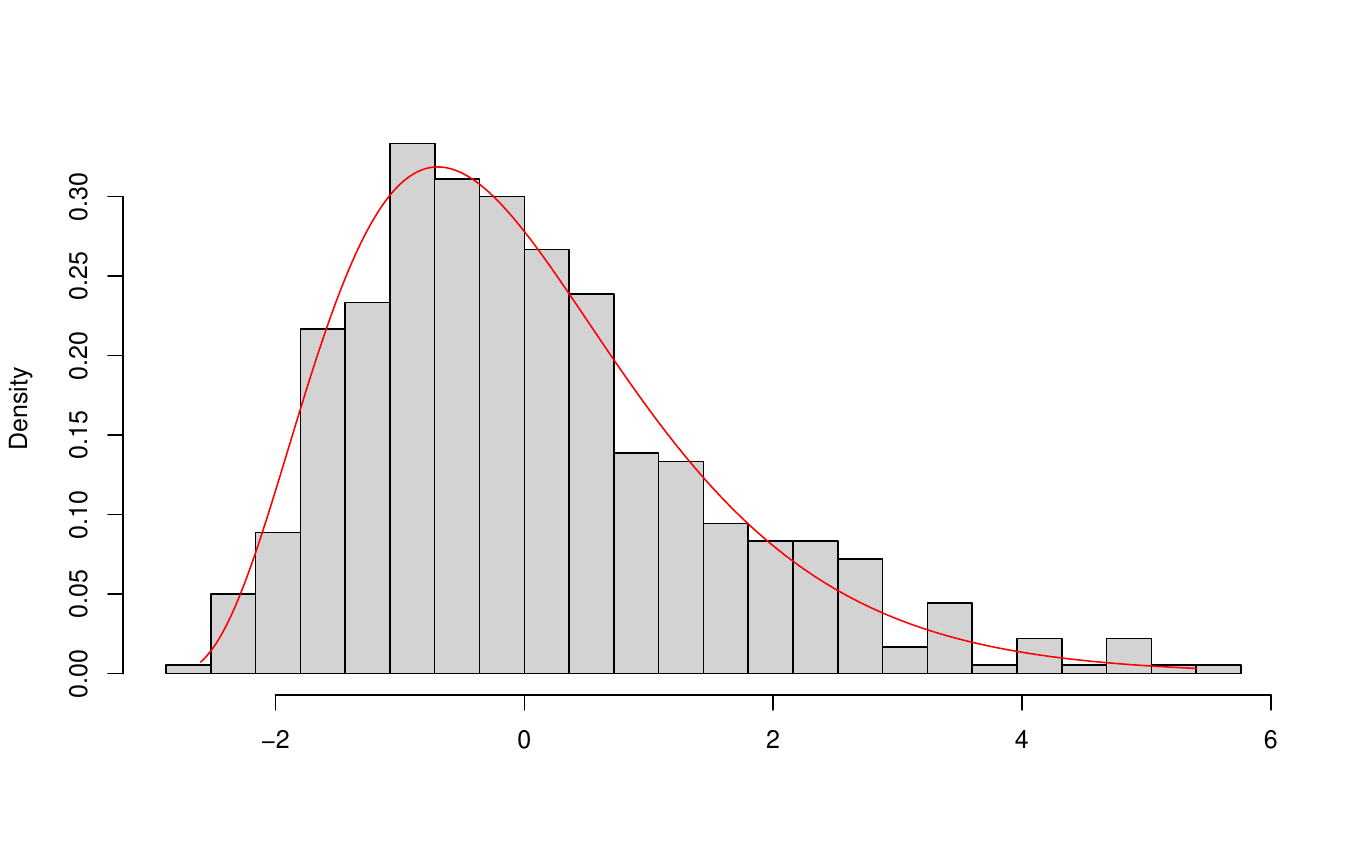}}
    \caption{Empirical histograms, based on $500$ Monte Carlo replicates, for $T(\hat{X}_i, \hat{X}_j)$ under $\mathbb{H}_0 \colon X_i = X_j$ when the link function is $\kappa(x,y) = \exp(-\|x - y\|)$, the latent positions $\{X_i\}$ are sampled i.i.d. from the uniform distribution on the unit sphere in $\mathbb{R}^{3}$, and the sparsity parameter is $\rho_n = 0.4$. The red curve in each plot is the probability density function for a weighted sum of $\hat{r}$ independent $\chi^2_1$ random variables as given in \cref{eq:twosample_critical_estimate}, where $\hat{r}$ is chosen according to \cref{eq:selection_r_test} resulting in $\hat{r} \equiv 4$ when $n = 4000$ and $\hat{r} \equiv 9$ when $n = 8000$.}
    \label{fig:T_simulation1}
\end{figure}

\begin{figure}[t!]
    \centering
    \subfloat[$n = 2000$]{\includegraphics[width=.5\textwidth]{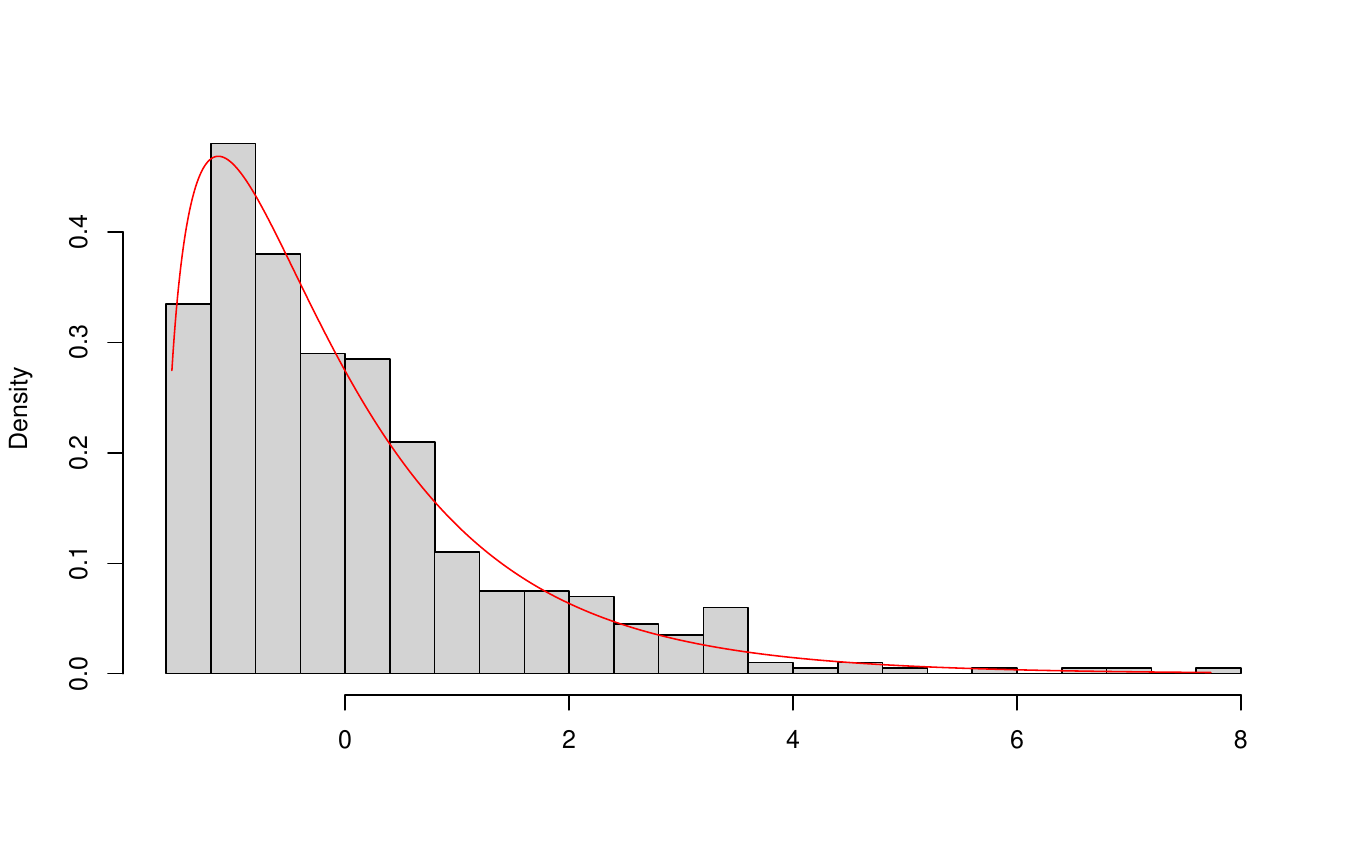}} 
    \subfloat[$n=4000$]{\includegraphics[width=.5\textwidth]{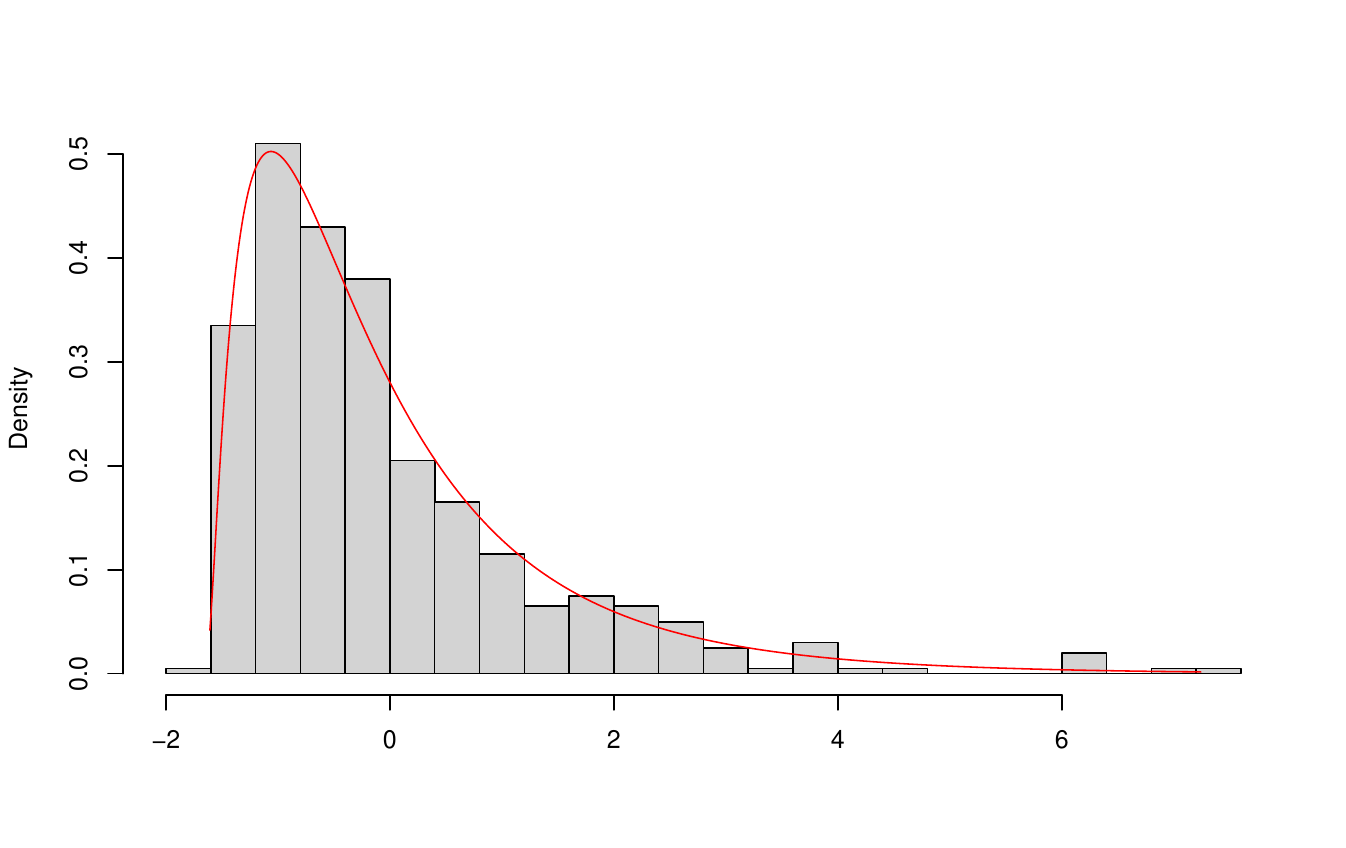}}
    \caption{Empirical histograms, based on $500$ Monte Carlo replicates, for $T(\hat{X}_i, \hat{X}_j)$ under $\mathbb{H}_0 \colon X_i = X_j$ when the link function is $\kappa(x,y) = \exp(-\tfrac{5}{2}\|x - y\|^2)$, the latent positions $\{X_i\}$ are sampled i.i.d. from the bivariate normal distribution with mean zero and identity covariance matrix, and the sparsity parameter is $\rho_n = 0.4$. The red curve in each plot is the probability density function for a weighted sum of $\hat{r}$ independent $\chi^2_1$ random variables as given in \cref{eq:twosample_critical_estimate}, where $\hat{r}$ is chosen via the criteria \cref{eq:selection_r_test} resulting in $\hat{r} \equiv 3$ when $n = 2000$ and $\hat{r} \equiv 6$ when $n = 4000$.}
    \label{fig:T_simulation2}
\end{figure}

\section{Numerical examples}
\label{sec:numerical}
We conduct several simulations to evaluate the finite-sample performance of the test statistic studied in \cref{thm:quadratic_test}.

For our first simulation, we sample a collection of $n - 1$ points $\{X_1, \dots, X_{n-1}\}$ in $\mathbb{R}^3$ according to a uniform distribution on the unit sphere, with $n \in \{4000,8000\}$. We set $X_n = X_1$ and let $p_{ij} = \rho \exp(-\|X_i - X_j\|)$, where $\rho = 0.4$. Given $P = (p_{ij})$, we sample $500$ independent realizations of $A$ from $P$. For each realization $A$ we obtain $\hat{U}$, the matrix formed by the $\hat{r}$ leading eigenvectors of $A$, where $\hat{r}$ is chosen according to \cref{eq:selection_r_test}, and then finally compute $T(\hat{X}_1, \hat{X}_n)$. For the link function specified here, the eigenvalues $\mu_r$ converge to zero at rate $r^{-3/2}$ \cite{scetbon}. Histograms of $T(\hat{X}_1, \hat{X}_n)$ based on these $500$ realizations of $A$ are shown in \cref{fig:T_simulation1}. These histograms indicate that the empirical distributions of $T(\hat{X}_1, \hat{X}_n)$ are well-approximated in finite samples by weighted sums of independent $\chi^2_1$ as specified in \cref{thm:quadratic_test}. We emphasize that, to make this simulation example fully data dependent, the weights are obtained from the eigenvalues of $\hat{U}^{\top} \hat{\mathcal{D}}^2 \hat{U}$. See \cref{eq:twosample_critical_estimate}.

For our second simulation, we sample a collection of $n - 1$ points $\{X_1, \dots, X_{n-1}\}$ in $\mathbb{R}^2$ according to a bivariate normal distribution with mean zero and identity covariance matrix. We set $X_n = X_1$ and let $P$ be the matrix whose entries are $p_{ij} = \rho \exp(-\|X_i - X_j\|^2/\sigma^2)$, where $\rho = 0.4$ and $\sigma^2 = 0.4$. Given $P$, we sample $500$ independent realizations of $A$ from $P$. For each realization $A$ we obtain $\hat{U}$, the matrix formed by the $r$ leading eigenvectors of $A$, where $\hat{r}$ is chosen according to \cref{eq:selection_r_test}, and then finally compute $T(\hat{X}_1, \hat{X}_n)$. Histograms of $T(\hat{X}_1, \hat{X}_n)$ based on these $500$ realizations of $A$ are shown in \cref{fig:T_simulation2}. These histograms again indicate that the empirical distributions of $T(\hat{X}_1, \hat{X}_n)$ are well-approximated in finite samples by weighted sums of independent $\chi^2_1$ as specified in \cref{thm:quadratic_test}. As above, the weights are obtained from the eigenvalues of  $\hat{U}^{\top} \hat{\mathcal{D}}^2 \hat{U}$ as described in \cref{eq:twosample_critical_estimate}. 

\begin{table}[tp]
   \begin{center}
  \begin{tabular}{|c|c|c|c|c|c|c|c|}
    \hline
    & $\epsilon = 0$ & $\epsilon = 0.1$ & $\epsilon = 0.2$ & $\epsilon = 0.3$ & $\epsilon = 0.5$
    & $\epsilon = 1$ & $\hat{r}$ \\
    \hline
    $n = 1000, \rho_n = 0.4$ & $0.046$ & $0.054$ & $0.068$ & $0.048$ & $0.062$ & $0.038$ &  $[1 (1.0)]$ \\ 
    $n = 1000, \rho_n = 0.6$ & $0.068$ & $0.064$ & $0.17$ & $0.332$ & $0.734$ & $0.984$ &  $[4 (1.0)]$ \\ 
    $n = 2000, \rho_n = 0.2$ & $0.056$ & $0.068$ & $0.056$ & $0.076$ & $0.072$ & $0.048$ & $[1 (1.0)]$ \\
    $n = 2000, \rho_n = 0.3$ & $0.046$ & $ 0.062$ & $0.176$ & $0.286$ & $0.666$ & $0.986$ & $[4 (1.0)]$ \\
    $n = 2000, \rho_n = 0.4$ & $0.064$ & $0.126$ & $0.216$ & $0.368$ & $0.796$ & $1$ & $[4 (1.0)]$  \\
    \hline
  \end{tabular}
  \end{center}
  \caption{Empirical estimate of power for testing equality of latent positions with
    exponential kernel. }
     \label{tab:laplace_power1}
\end{table}

\begin{table}[tp]
  \begin{center}
  \begin{tabular}{|c|c|c|c|c|c|c|c|}
    \hline
    & $\epsilon = 0$ & $\epsilon = 0.1$ & $\epsilon = 0.2$ & $\epsilon = 0.3$ & $\epsilon = 0.5$
    & $\epsilon = 1$ & $\hat{r}$ \\
    \hline
    $n = 500, \rho_n = 0.8$ & $0.03$ & $ 0.06$ & $0.172$ & $0.328$ & $0.528$ & $0.784$ & $[1 (0.78), 3 (0.22)]$ \\
    $n = 1000, \rho_n = 0.4$ & $0.042$ & $0.06$ & $0.15$ & $0.3$ & $0.584$ & $0.774$ &  $[1 (0.71), 3 (0.29)]$ \\ 
    $n = 2000, \rho_n = 0.2$ & $0.03$ & $0.06$ & $0.125$ & $0.29$ & $0.56$ & $0.8$ & $[1 (0.68), 3 (0.32)]$ \\
    $n = 2000, \rho_n = 0.4$ & $0.034$ & $0.068$ & $0.35$ & $0.6$ & $0.886$ & $0.972$ & $[3 (1.0)]$  \\
    $n = 2000, \rho_n = 0.8$ & $0.048$ & $ 0.276$ & $0.726$ & $0.882$ & $0.948$ & $0.982$ & $[3 (0.05), 6 (0.95)]$ \\
    \hline
  \end{tabular}
  \end{center}
  \caption{Empirical estimate of power for testing equality of latent positions with Gaussian kernel.}
    \label{tab:gaussian_power1}
\end{table}
Finally, we evaluate the empirical size and power of our proposed test statistic. The settings here are the same as described above except that we now sample $X_n$ so that $\|X_n - X_1\| = \epsilon$ for $\epsilon \in \{0, 0.1, 0.2, 0.3, 0.5, 1\}$. For each realization $A$, we compute $T_{1n} = T(\hat{X}_1, \hat{X}_n)$ and reject the null hypothesis whenever $T_{1n}$ exceeds the $0.95$ quantile of the distribution described in \cref{eq:twosample_critical_estimate}. The results are presented in \cref{tab:gaussian_power1,tab:laplace_power1} for several combinations of $n$ and $\rho_n$. Each entry in these tables is a sample average based on $500$ Monte Carlo replicates. The entries in column $\epsilon = 0$ are the empirical size of the test statistic under the null hypothesis, while those in the columns with $\epsilon > 0$ are the empirical power under the alternative hypothesis. The last column in each table records the embedding dimensions $\hat{r}$ (chosen according to the criteria in \cref{eq:selection_r_test}) and their frequency. For example, in \cref{tab:gaussian_power1} with $n = 500$ and $\rho_n = 0.8$, we observed $\hat{r} = 1$ in $78\%$ of the replicates and $\hat{r} = 3$ in $22\%$ of the replicates, while with $n = 2000$ and $\rho_n = 0.4$ we observed $\hat{r} = 3$ in all replicates. The results in \cref{tab:gaussian_power1,tab:laplace_power1} indicate that our test statistic has well-controlled Type I error under the null and exhibits significant power under the alternative. Finally, recorded values of $\hat{r}$ also suggest that, although both the Laplace and Gaussian kernels are of infinite rank, we only need embedding dimensions of $\hat{r} \leq 6$ to obtain a valid and consistent test procedure. This supports the folklore observation in Statistics and Data Science that, for many inference tasks, high or even full rank models can be well-approximated by their low-rank counterparts.

{\bf Acknowledgments}
    MT was partially supported by the National Science Foundation under grant DMS-2210805. JC was partially supported by the National Science Foundation under grant DMS-2413552. JC gratefully acknowledges support from the University of Wisconsin–Madison, Office of the Vice Chancellor for Research and Graduate Education, with funding from the Wisconsin Alumni Research Foundation.

\bibliographystyle{abbrvnat}
\bibliography{growing_rank}

\newpage
\appendix
\section{Proofs of stated results}
\label{sec:proofs}

\subsection{Proof of \cref{thm:psd} (positive semidefinite kernel $\kappa$)}
\label{sec:proof_simple} 
Recall that by convention, we index the eigenvalues of $A$ and $P$ in decreasing order by magnitude. The eigenvalues of $P$ are all non-negative by the positive semidefiniteness of $\kappa$, so in particular
\begin{equation*}
  |\hat{\lambda}_1|
  \geq
  |\hat{\lambda}_2|
  \geq
  \dots
  \geq
  |\hat{\lambda}_n|,
  \quad
  \text{and}
  \quad
  \lambda_1
  \geq
  \lambda_2
  \geq
  \dots
  \geq
  \lambda_n
  .
\end{equation*}
For any $s \geq 1$, define the eigenvalue gap $\delta_s = \lambda_{s} - \lambda_{s+1} \ge 0$. Let $\mathcal{E}_0$ be the event
\begin{equation}
    \label{eq:event_0}
    \mathcal{E}_0
    =
    \left\{
    \delta_r \geq \max\{4\|E\|, 8 C_2(\nu) \log n\}
    \right\}
    , 
\end{equation}
where $C_2(\nu) = \tfrac{2}{3}(\nu + 2)$. By Weyl's inequality and the stated hypotheses, on $\mathcal{E}_0$ it holds that
\begin{equation}
    \label{eq:eigenvalues_relation}
    \hat{\lambda}_r
    \geq
    \lambda_r - \|E\|
    \geq
    \lambda_r - \frac{1}{4} \delta_r
    \geq
    \frac{1}{2} \lambda_r
    >
    0
    .
\end{equation}
Furthermore, by the discussion in the main text, since $\kappa$ is positive semidefinite, for any $r \geq 1$, 
\begin{gather}
    \label{eq:two_infU_psd}
    \|U \Lambda^{1/2}\|_{2 \to \infty}
    \leq
    \rho_n^{1/2},
    \quad
    \|U_{\perp} \Lambda_{\perp}^{1/2}\|_{2 \to \infty} 
    \leq
    \rho_n^{1/2},
    \quad
    \|U\|_{2 \to \infty}
    \leq
    \lambda_r^{-1/2} \rho_n^{1/2}
    .
\end{gather}
Now, letting $W^{(n)} = \operatorname{arg~min}_{W \in \mathcal{O}(r)} \|\hat{U}-UW\|_{F}$, consider the expansion
\begin{equation}
  \label{eq:decomp}
  \begin{split}
    \hat{U} \hat{\Lambda}^{1/2} - U \Lambda^{1/2} W^{(n)}
    =
    (I - UU^{\top})\hat{U} \hat{\Lambda}^{1/2}
    &+
    U U^{\top} \hat{U} \hat{\Lambda}^{1/2} - U \Lambda^{1/2} U^{\top} \hat{U} \\
    &+
    U \Lambda^{1/2} (U^{\top} \hat{U} - W^{(n)})
    .
  \end{split}
\end{equation}
Next, we shall bound the two-to-infinity norm of each term on the right hand side of \cref{eq:decomp}. Our bounds involve the following quantities that are well-controlled using the technical lemmas in \cref{sec:technical_lemmas}. Specifically,
\begin{gather*}
  \|(I - UU^{\top}) \hat{U}\|
  \leq
  \frac{2\|E\|}{\delta_r}
  =
  \psi_0,
  \qquad
  \psi_1
  =
  \|U^{\top} E U\|, \\
  \psi_2
  =
  \| U^{\top} \hat{U} \hat{\Lambda}^{1/2} - \Lambda^{1/2} U^{\top} \hat{U} \|,
  \qquad
  \psi_3^{(k)}
  =
  \|U_{\perp} |\Lambda_{\perp}|^{k}  U_{\perp}^{\top} E U |\Lambda|^{1/2}\|_{2 \to \infty}.
\end{gather*}

\subsubsection*{Bounding $R_0 \coloneqq U \Lambda^{1/2} (U^{\top} \hat{U} - W^{(n)})$}
Applying \cref{eq:2toinf_submultiplicative} and the aforementioned technical lemma gives
\begin{equation}
    \label{eq:main_term_1}
    \begin{split}
    \|R_0\|_{2 \to \infty}
    &=
    \|U \Lambda^{1/2} (U^{\top} \hat{U} - W^{(n)})\|_{2 \to \infty} \\
    &\leq
    \|U \Lambda^{1/2}\|_{2 \to\infty} \times \|U^{\top} \hat{U} - W^{(n)}\| \\
    &\le
    \|U \Lambda^{1/2}\|_{2 \to \infty} \times \|(I-UU^{\top})\widehat{U}\|^{2}\\
    &\le
    \|U \Lambda^{1/2}\|_{2 \to \infty} \times \psi_0^2
    ,
    \end{split}
\end{equation}
where the second inequality follows from \cite[Lemma~6.7]{cape2019two}.
\subsubsection*{Bounding $R_1 \coloneqq UU^{\top} \hat{U} \hat{\Lambda}^{1/2} - U \Lambda^{1/2} U^{\top}\hat{U}$} 
We have
\begin{equation*}
    \|R_1\|_{2 \to \infty}
    \leq
    \|U\|_{2 \to \infty} \times \| \Lambda^{1/2} U^{\top} \hat{U} - U^{\top} \hat{U} \hat{\Lambda}^{1/2} \|
    \leq
    \|U\|_{2 \to \infty} \times \psi_2
    ,
\end{equation*}
where $\psi_2$ is the upper bound for $\| \Lambda^{1/2} U^{\top} \hat{U} - U^{\top} \hat{U} \hat{\Lambda}^{1/2}\|$ given in \cref{lem:approximate_commute} and mentioned above (with an immaterial sign change inside the operator norm).

\subsubsection*{Bounding $(I - UU^{\top}) \hat{U} \hat{\Lambda}^{1/2}$}
By expanding $\hat{U}\hat{\Lambda} = A\hat{U} = P \hat{U} + E \hat{U}$, we have
\begin{equation*}
  \begin{split}
  (I - U U^{\top}) \hat{U} \hat{\Lambda}^{1/2}
  &=
  (I - UU^{\top}) E \hat{U} \hat{\Lambda}^{-1/2} + (I - UU^{\top}) P \hat{U} \hat{\Lambda}^{-1/2} \\
  &=
  E \hat{U} \hat{\Lambda}^{-1/2} - U U^{\top} E \hat{U} \hat{\Lambda}^{-1/2} + (I - UU^{\top}) P \hat{U} \hat{\Lambda}^{-1/2}
  .
  \end{split}
\end{equation*}
The term $U U^{\top} E \hat{U} \hat{\Lambda}^{-1/2}$ admits the bound
\begin{equation*}
    \begin{split}
    \|U U^{\top} E \hat{U} \hat{\Lambda}^{-1/2} \|_{2 \to \infty} 
    &\leq 
    \|U\|_{2 \to \infty} \times \|U^{\top} E \hat{U} \| \times \|\hat{\Lambda}^{-1/2}\| \\
    &\leq
    \|U\|_{2 \to \infty} \times \frac{1}{|\hat{\lambda}_r|^{1/2}} \times (\|U^{\top} E U\| + \|U^{\top} E (I - UU^{\top}) \hat{U}\|) \\
    &\leq
    \frac{\sqrt{2} \|U\|_{2 \to \infty}(\psi_1 + \|E\| \psi_0)}{\lambda_r^{1/2}}
    .
  \end{split}
\end{equation*}
Above, the final inequality follows from bounding $\|(I - UU^{\top}) \hat{U}\|$ via the Davis--Kahan theorem \cite{davis70} and using \cref{eq:eigenvalues_relation} to replace $\hat{\lambda}_r$ by $\lambda_r$. 

Next, write $\Pi_{U}^{\perp} = I - UU^{\top}$. Repeatedly expanding $\hat{U}$ as $P \hat{U} \hat{\Lambda}^{-1} + E \hat{U} \hat{\Lambda}^{-1}$ yields
\begin{equation*}
    \begin{split}
    (I - UU^{\top}) P \hat{U} \hat{\Lambda}^{-1/2}
    &=
    \Pi_{U}^{\perp} P^2 \hat{U} \hat{\Lambda}^{-3/2} + \Pi_{U}^{\perp} PE \hat{U} \hat{\Lambda}^{-3/2} \\
    &=
    \Pi_{U}^{\perp} P^3 \hat{U} \hat{\Lambda}^{-5/2} + \Pi_{U}^{\perp} P^2 E \hat{U} \hat{\Lambda}^{-5/2} + \Pi_{U}^{\perp} P E \hat{U} \hat{\Lambda}^{-3/2} \\
    &=
    \dots \\
    &=
    \Pi_{U}^{\perp} P^{m} \hat{U} \hat{\Lambda}^{-(m-1/2)}
    +
    \sum_{k=1}^{m-1} \Pi_{U}^{\perp} P^{k} E \hat{U} \hat{\Lambda}^{-(k+1/2)}
    ,
    \end{split}
    \end{equation*}
which holds for any integer $m \geq 2$. For the first term on the right-hand side, we have
\begin{equation*}
    \begin{split}
    \|\Pi_{U}^{\perp} P^m \hat{U} \hat{\Lambda}^{-(m - 1/2)}\|_{2 \to \infty}
    &=
    \|U_{\perp} \Lambda_{\perp}^{m} U_{\perp}^{\top} \hat{U} \hat{\Lambda}^{-(m-1/2)}\|_{2 \to \infty} \\
    &\leq
    \|U_{\perp} \Lambda_{\perp}^{1/2}\|_{2 \to \infty} \times \|\Lambda_{\perp}\|^{m-1/2} \times \|U_{\perp}^{\top} \hat{U}\| 
    \times \|\hat{\Lambda}^{-(m-1/2)}\| \\
    &\leq
    \rho_n^{1/2} \times \psi_0 \times \left|\frac{\lambda_{r+1}}{\hat{\lambda}_r}\right|^{m-1/2}
    ,
    \end{split}
\end{equation*}
where the second inequality follows from bounding $\|U_{\perp}^{\top} \hat{U}\|$ using the Davis--Kahan theorem, i.e., $\|U_{\perp}^{\top} \hat{U}\| = \|(I - UU^{\top}) \hat{U}\| \le \psi_0$.

On the event $\mathcal{E}_0$, it holds that
\begin{equation}
    \label{eq:bound_gap_psd}
    \frac{\lambda_{r+1}}{\hat{\lambda}_r}
    \leq
    \frac{\lambda_{r+1}}{\lambda_r - \|E\|}
    \leq
    \frac{\lambda_{r+1} + \|E\|}{\lambda_r}
    \leq
    1 - \frac{\lambda_r - \lambda_{r+1}}{2\lambda_r}
    =
    1 - \frac{\delta_r}{2 \lambda_r}
    .
\end{equation}
Since $\delta_r \geq 4 \varsigma(\nu,n) \geq 18 (n \rho_n)^{1/2}$ (see \cref{eq:r_select_psd1}) and $\lambda_r \leq n \rho_n$, choosing 
$m = n \rho_n^{1/2} + 1 $ yields 
\begin{equation}
    \label{eq:bound_gap_psd2}
    \left(1 - \frac{\lambda_{r} - \lambda_{r-1}}{2\lambda_r}\right)^{m-1/2}
    \leq
    \left(1 - \frac{9}{(n \rho_n)^{1/2}}\right)^{n \rho_n^{1/2}}
    \leq
    \exp\left(-9n^{1/2}\right)
    ,
    \end{equation}
where the final inequality holds since $1 - x \leq e^{-x}$ for $x \in [0,1]$. For simplicity of notations, we have assumed that $n \rho_n^{1/2}$ is an integer, for otherwise we can replace it with $\lceil n \rho_n^{1/2}\rceil$ without changing the subsequent argument. In summary,
\begin{equation}
\label{eq:bound_gap_psd3}
    \|\Pi_{U}^{\perp} P^m \hat{U} \hat{\Lambda}^{-(m - 1/2)}\|_{2 \to \infty}
    \leq
    \rho_n^{1/2} \psi_0 \exp(-9n^{1/2})
    .
\end{equation}
We now evaluate $\zeta_k = 
\|\Pi_{U}^{\perp} P^k E  \hat{U} \hat{\Lambda}^{-(k+1/2)}\|_{2 \to \infty}$ for $k \geq 1$.
Since $I = UU^{\top} + U_{\perp}U_{\perp}^{\top}$, we first write
\begin{equation*}
    \begin{split}
    \zeta_k
    &\leq
    \underbrace{\|\Pi_{U}^{\perp} P^k E U U^{\top} \hat{U} \hat{\Lambda}^{-(k+1/2)}\|_{2 \to \infty}}_{\zeta_k^{(1)}}
    +
    \underbrace{\|\Pi_{U}^{\perp} P^k E U_{\perp} U_{\perp}^{\top} \hat{U}\hat{\Lambda}^{-(k+1/2)}\|_{2 \to \infty}}_{\zeta_k^{(2)}}
    .
\end{split}
\end{equation*}
For $\zeta_k^{(2)}$, observe that
\begin{equation*}
    \begin{split}
    \zeta_k^{(2)}
    &=
    \|U_{\perp} \Lambda_{\perp}^{k} U_{\perp}^{\top} E U_{\perp} U_{\perp}^{\top} \hat{U} \hat{\Lambda}^{-(k+1/2)}\|_{2 \to \infty} \\
    &\leq
    \|U_{\perp} \Lambda_{\perp}^{1/2}\|_{2 \to \infty}
    \times
    \|\Lambda_{\perp}\|^{k-1/2}
    \times
    \|E\|
    \times
    \|U_{\perp}^{\top} \hat{U}\|
    \times
    \|\hat{\Lambda}^{-1}\|^{(k+1/2)} \\
    &\leq
    \frac{\rho_n^{1/2} \|E\| \psi_0}{|\hat{\lambda}_r|}
    \times
    \left|\frac{\lambda_{r+1}}{\hat{\lambda}_r}\right|^{k-1/2}  \\
    &\leq
    \frac{2 \rho_n^{1/2} \|E\| \psi_0}{\lambda_r} \times \left(1 - \frac{\delta_r}{2\lambda_r}\right)^{k-1/2}
    ,
\end{split}
\end{equation*}
where the final inequality follows from \cref{eq:bound_gap_psd}. This in turn implies
\begin{equation*}
    \sum_{k=1}^{m-1} \zeta_k^{(2)}
    \leq
    \sum_{k=1}^{m-1} \frac{2 \rho_n^{1/2} \|E\| \psi_0}{\lambda_r} \left(1 - \frac{\delta_r}{2\lambda_r}\right)^{k-1/2}
    \leq
    \frac{4 \rho_n^{1/2} \|E\| \psi_0}{\delta_r}
    .
\end{equation*}
Next, for $1 \leq k \leq m$, it holds that
\begin{equation*}
    \begin{split}
    \zeta_k^{(1)}
    &\leq
    \|U_{\perp} \Lambda_{\perp}^{k} U_{\perp} E U\|_{2 \to \infty}
    \times
    \|\hat{\Lambda}^{-(k+1/2)}\|
    \leq
    \psi_3^{(k)} |\hat{\lambda}_r|^{-(k+1/2)}
    ,
\end{split}
\end{equation*}
where $\psi_3^{(k)}$ is an upper bound for $\|U_{\perp} \Lambda_{\perp}^{k} U_{\perp}^{\top} E U\|_{2 \to \infty}$ as given in \cref{lem:technical4_new}, and we have used the fact that $\|U^{\top}\hat{U}\| \le 1$.

In summary, thus far, we have established that
\begin{equation}
    \label{eq:prior_R2}
    (I - UU^{\top})\hat{U} \hat{\Lambda}^{1/2}
    =
    E \hat{U} \hat{\Lambda}^{-1/2}
    +
    R_2
    ,
\end{equation}
where $R_2$ satisfies
\begin{equation}
    \label{eq:residual_r3}
    \begin{split}
    \|R_2\|_{2 \to \infty}
    &\leq
    \frac{\sqrt{2}\|U\|_{2 \to \infty}(\psi_1 + \|E\| \psi_0)}{\lambda_r^{1/2}} + \rho_n^{1/2} \psi_0 \left(\exp(-9n^{1/2})
    + 
    \frac{4 \|E\|}{\delta_r}\right) \\
    &\qquad+
    \sum_{k=1}^{n \rho_n^{1/2}} \psi_3^{(k)} |\hat{\lambda}_{r}|^{-(k+1/2)}
    .
\end{split}
\end{equation}

\subsubsection*{Bounding $E \hat{U} \hat{\Lambda}^{-1/2}$}
We begin with the expansion
\begin{equation*}
    \begin{split}
    E \hat{U} \hat{\Lambda}^{-1/2}
    &=
    E U U^{\top} \hat{U} \hat{\Lambda}^{-1/2}
    +
    E (I - U U^{\top}) \hat{U} \hat{\Lambda}^{-1/2} \\ 
    &=
    E U \Lambda^{-1/2} \left[W^{(n)}
    +
    (U^{\top} \hat{U} - W^{(n)})
    +
    (\Lambda^{1/2} U^{\top} \hat{U} - U^{\top} \hat{U} \hat{\Lambda}^{1/2}) \hat{\Lambda}^{-1/2}\right] \\
    &\qquad+
    E (I - U U^{\top}) \hat{U} \hat{\Lambda}^{-1/2}
    ,
  \end{split}
\end{equation*}
where $W^{(n)}$ denotes the orthogonal matrix mentioned above that solves the Frobenius norm Procrustes problem for $U$ and $\hat{U}$. Considering the above expansion, we have
\begin{align*}
    \|E U \Lambda^{-1/2} (U^{\top} \hat{U} - W^{(n)})\|_{2 \to \infty}
    &\leq
    \|E U \Lambda^{-1/2}\|_{2 \to \infty} \times \psi_0^2, \\
    \| E U \Lambda^{-1/2} (\Lambda^{1/2} U^{\top} \hat{U} - U^{\top} \hat{U} \hat{\Lambda}^{1/2}) \hat{\Lambda}^{-1/2} \|_{2 \to \infty}
    &\leq
    \sqrt{2} |\lambda_r|^{-1/2}\|E U \Lambda^{-1/2}\|_{2 \to \infty} \times \psi_2,
\end{align*}
where $\psi_2$ is an upper bound for $\|\Lambda^{1/2} U^{\top} \hat{U}  - U^{\top} \hat{U} \hat{\Lambda}^{1/2}\|$ given in \cref{lem:approximate_commute}.

Our last remaining technical challenge in this proof is to bound
$\|E (I - UU^{\top}) \hat{U}\|_{2 \to \infty}$, and we do so through a
careful leave-one-out analysis. Leave-one-out style arguments provide a useful
and elegant approach for handling the (often times) complicated dependencies
between the rows of $\hat{U}$. See
\cite{abbe2020entrywise,spectral_chen,lei2019unified,javanmard_montanari,zhong_boumal,xie2021entrywise}
for various examples of leave-one-out analysis in the context of random graph
inference, linear regression using LASSO, and phase synchronization.

Although we are motivated in part by the ideas used in the proofs of \cite[Theorem~2.1]{abbe2020entrywise} and \cite[Theorem~3.1]{lei2019unified}, our arguments require noticeably more delicate analysis in order to decompose the bound for $E(I - UU^{\top}) \hat{U}$ into a main order term $E U \Lambda^{-1/2}$ and a negligible lower order term. Inference results in \cref{sec:entrywise_approximation} and \ref{sec:two_sample} are subsequently made possible by further study of the main order term.

To proceed with the proof, we first introduce some helpful notations. Given $A$, define the collection of auxiliary matrices $A^{[1]}, A^{[2]}, \dots, A^{[n]}$, where for each $1 \leq h \leq n$,
\begin{equation}
    \label{eq:def_Ah}
    A^{[h]}(i,j)
    =
    \begin{cases}
        A(i,j)
        &
        \text{if $i \not = h$ and $j \not = h$}, \\
        P(i,j)
        &
        \text{otherwise}.
    \end{cases}    
\end{equation}
\cref{eq:def_Ah} implies that the entries of $A^{[h]} - A$ are all zeros except for the entries in the $h$-th row and $h$-th column which are the same as those for the $h$-th row and $h$-th column of $E = A - P$, up to a negative sign. For consistency of notation, let $\hat{U}^{[h]}$ be the $n \times r$ matrix whose columns are the leading eigenvectors of $A^{[h]}$.

\begin{lemma}
\label{lem:loo}
Consider the setting in \cref{thm:psd} or \cref{thm:general}. Define $\hat{U}^{[h]}$ as above, and let $V^{[h]} = (I - U U^{\top}) \hat{U}^{[h]}$. Suppose there exists $\nu > 0$ such that the events 
\begin{gather}
    \mathcal{E}_0
    =
    \left\{
    \delta_r
    \geq
    \max\{4\|E\|, 8 C_2(\nu) \log n\}
    \right\}, \\
    \label{eq:e1_event}
    \mathcal{E}_1
    =
    \left\{
    \forall h \in [n],
    \,\,\,
    \|e_h^{\top} E V^{[h]}\|
    \leq
    C_1(\nu) \sqrt{\rho_n \log n} \|V^{[h]}\|_F + C_2(\nu) \|V^{[h]}\|_{2 \to \infty} \log n
    \right\}
\end{gather}
hold simultaneously, with $C_1(\nu) = \sqrt{2 (\nu + 2)}$ and $C_2(\nu) = \tfrac{2}{3} (\nu + 2)$. Then,
\begin{equation}
    \label{eq:loo_final3_lemma}
    \begin{split}
    \|E(I - UU^{\top}) \hat{U}\|_{2 \to \infty}
    &\leq
    \frac{16 C_1(\nu) \sqrt{r \rho_n \log n} \|E\|}{\delta_r} \\
    &\qquad+
    \frac{8 (\|E\| \cdot\|U\|_{2 \to \infty} + \|EU\|_{2 \to \infty}) (C_2(\nu) \log n + 2 \|E\|)}{\delta_r} \\
    &\qquad+
    \|(I - UU^{\top}) \hat{U}\|_{2 \to \infty} \left(6 C_2(\nu) \log n + \frac{16 \|E\|^2 }{\delta_r}\right)
    .
    \end{split}
\end{equation}
\end{lemma}
We emphasize that, conditional on $P$, $V^{[h]}$ is independent of the entries in the $h$-th row and $h$-th column of $E$. Therefore, $e_h^{\top} E$ and $V^{[h]}$ are independent, and we can bound the probability of the event $\mathcal{E}_1$ via \cref{lem:ex_2inf}.

By applying \cref{lem:loo}, we have
\begin{equation*}
    \begin{split}
    &\|E(I - UU^{\top}) \hat{U} \hat{\Lambda}^{-1/2}\|_{2 \to \infty} \\
    &\qquad\leq
   \frac{16 \sqrt{2} C_1(\nu) \sqrt{r \rho_n \log n} \|E\|}{\delta_r \lambda_r^{1/2}} \\
   &\qquad\qquad+
   \frac{8 \sqrt{2}(\|E\| \cdot\|U\|_{2 \to \infty} + \|EU\|_{2 \to \infty}) (C_2(\nu) \log n + 2 \|E\|)}{\delta_r \lambda_r^{1/2}} \\
   &\qquad\qquad+
   \sqrt{2} \lambda_r^{-1/2}\|(I - UU^{\top}) \hat{U}\|_{2 \to \infty} \left(6 C_2(\nu) \log n + 
  \frac{16 \|E\|^2 }{\delta_r}\right)
  .
\end{split} 
\end{equation*}
Next, observe that
\begin{equation*}
    \|(I - UU^{\top}) \hat{U}\|_{2 \to \infty}
    \leq
    \sqrt{2} \lambda_r^{-1/2} \|(I - UU^{\top}) \hat{U} \hat{\Lambda}^{1/2}\|_{2 \to \infty}
    ,
\end{equation*}
where the term $(I - UU^{\top}) \hat{U} \hat{\Lambda}^{1/2}$ previously appeared in \cref{eq:prior_R2}. Hence, we have
\begin{equation*}
    \begin{split}
    &\|E(I - UU^{\top}) \hat{U} \hat{\Lambda}^{-1/2}\|_{2 \to \infty} \\
    &\qquad\le
    \frac{16 \sqrt{2} C_1(\nu) \sqrt{r \rho_n \log n} \|E\|}{\delta_r \lambda_r^{1/2}} \\
    &\qquad\qquad+ 
    \frac{8 \sqrt{2}(\|E\| \cdot\|U\|_{2 \to \infty} + \|EU\|_{2 \to \infty}) (C_2(\nu) \log n + 2 \|E\|)}{\delta_r \lambda_r^{1/2}} \\
    &\qquad\qquad+
    \lambda_r^{-1}\|(I - UU^{\top}) \hat{U} \hat{\Lambda}^{1/2}\|_{2 \to \infty} \left(12 C_2(\nu) \log n + \frac{32 \|E\|^2 }{\delta_r}\right)
    .
  \end{split}
\end{equation*} 
Let $T_* = E U \Lambda^{-1/2} W^{(n)}$. Combining the above expressions, we obtain
\begin{align*}
    (I - UU^{\top}) \hat{U} \hat{\Lambda}^{1/2}
    &=
    T_* + R_2 + Y_0 + Y_1 + Y_2, \\
    \|Y_0\|_{2 \to \infty}
    &\leq
    \|T_*\|_{2 \to \infty} (\psi_0^2 + 2^{1/2} \lambda_r^{-1/2} \psi_2), \\
    \|Y_1\|_{2 \to \infty}
    &\leq
    \frac{16 \sqrt{2} C_1(\nu) \sqrt{r \rho_n \log n} \|E\|}{\delta_r \lambda_r^{1/2}} \\
    &\qquad+
    \frac{8 \sqrt{2} (\|E\| \cdot\|U\|_{2 \to \infty} + \|EU\|_{2 \to \infty}) (C_2(\nu) \log n + 2  \|E\|)}{\delta_r \lambda_r^{1/2}} \\
    \|Y_2\|_{2 \to \infty}
    &\leq
    \frac{12 C_2(\nu) \log n + 32 \delta_r^{-1} \|E\|^2}{\lambda_r} \|(I - UU^{\top}) \hat{U} \hat{\Lambda}^{1/2}\|_{2 \to \infty}
    .
\end{align*}
Hence, provided $\lambda_r \geq 24 C_2(\nu) \log n + 64 \delta_r^{-1} \|E\|^2$, then
\begin{align*}
    &\|(I - UU^{\top})\hat{U} \hat{\Lambda}^{1/2}\|_{2 \to \infty} \\
    &\qquad\leq
    \frac{1}{1 - \frac{12 C_2(\nu) \log n + 32 \delta_r^{-1} \|E\|^2}{\lambda_r}}
    \left(\|T_*\|_{2 \to \infty} + \|R_2\|_{2 \to \infty} + \|Y_0\|_{2 \to \infty} + \|Y_1\|_{2 \to \infty} \right)
\end{align*}
and
\begin{align*}
    \|Y_2\|_{2 \to \infty}
    &\leq
    \frac{24 C_2(\nu) \log n + 64 \delta_r^{-1} \|E\|^2}{\lambda_r}
    \left(\|T_*\|_{2 \to \infty} + \|R_2\|_{2 \to \infty} + \|Y_0\|_{2 \to \infty} +\|Y_1\|_{2 \to \infty} \right)
    .
\end{align*}

\subsubsection{Wrapping up}
In summary, suppose there exists a constant $\nu > 0$ such that all of the following events simultaneously hold.
\begin{gather*}
    \mathcal{E}_0
    =
    \{ \delta_{r} \geq \max\{6 \|E\|, 8 C_2(\nu) \log n\}\}, \\
    \mathcal{E}_1
    =
    \{\forall h \in [n], \,\,\, \|e_h^{\top} E V^{[h]}\| \leq C_1(\nu) \sqrt{\rho_n \log n} \|V^{[h]}\|_{F} + C_2(\nu) \|V^{[h]}\|_{2 \to \infty} \log n\}, \\
    \mathcal{E}_2
    =
    \{ \lambda_r \geq 24 C_2(\nu) \log n + 64 \delta_r^{-1} \|E\|^2)\}
    .
\end{gather*}
Recall that $V^{[h]} = (I - UU^{\top}) \hat{U}^{[h]}$, $C_1(\nu) = \sqrt{2 (\nu + 2)}$, and $C_2(\nu) = \tfrac{2}{3} (\nu + 2)$. 

We therefore have
\begin{equation}
    \label{eq:wrapup_expansion}
    \hat{U} \hat{\Lambda}^{1/2} - U \Lambda^{1/2} W^{(n)}
    =
    E U \Lambda^{-1/2} W^{(n)}
    +
    R_0 + R_1 + R_2 + Y_0 + Y_1,
\end{equation}
where $T_* = E U \Lambda^{-1/2} W^{(n)}$ and the residuals $R_0$ through $Y_2$ satisfy
{\footnotesize
\begin{gather*}
    \|R_0\|_{2 \to \infty} \leq \|U \Lambda^{1/2} \|_{2 \to \infty} \psi_0^2,
    \qquad
    \|R_1\|_{2 \to \infty}
    \leq
    \|U\|_{2 \to \infty} \psi_2 \\
    \|R_2\|_{2 \to \infty}
    \leq
    \frac{\sqrt{2}\|U\|_{2 \to \infty}( \psi_1 + \|E\| \psi_0)}{\lambda_r^{1/2}} + \rho_n^{1/2} \psi_0 \left(\exp(-9n^{1/2}) + \frac{4 \|E\|}{\delta_r}\right) + \sum_{k=1}^{n \rho_n^{1/2}} \psi_3^{(k)} |\hat{\lambda}_r|^{-(k+1/2)}, \\
    \|Y_0\|_{2 \to \infty}
    \leq
    \|T_*\|_{2 \to \infty} (\psi_0^2 + 2^{1/2} \lambda_r^{-1/2} \psi_2), \\
    \|Y_1\|_{2 \to \infty}
    \leq
    \frac{16 \sqrt{2} C_1(\nu) \sqrt{r \rho_n \log n} \|E\|}{\delta_r \lambda_r^{1/2}} + \frac{8 \sqrt{2} (\|E\| \cdot\|U\|_{2 \to \infty} + \|EU\|_{2 \to \infty}) (C_2(\nu) \log n + 2  \|E\|)}{\delta_r \lambda_r^{1/2}}, \\
    \|Y_2\|_{2 \to \infty}
    \leq
    \frac{ 24 C_2(\nu) + 64 \delta_r^{-1} \|E\|^2)}{\lambda_r} \left(\|T_*\|_{2 \to \infty} + \|R_2\|_{2 \to \infty} + \|Y_0\|_{2 \to \infty} + \|Y_1\|_{2 \to \infty} \right).
\end{gather*}
}
We now bound each term in the above display equations. By \cref{lem:bandeira_vanhandel},
\begin{equation}
    \label{eq:E_spectral_norm}
    \|E\|
    \leq
    \varsigma(\nu,n)
\end{equation}
with probability at least $1 - n^{-\nu}$, and hence
\begin{equation*}
    \psi_0
    =
    \frac{2 \|E\|}{\delta_r}
    \leq
    \frac{2\varsigma(\nu,n)}{\delta_r}
\end{equation*}
with probability at least $1 - n^{-\nu}$. Next, for $\psi_1$, we have by \cref{lem:technical2a} that
\begin{equation*}
    \psi_1
    =
    \|U^{\top} E U \|
    \leq
    4 \sqrt{\rho_n \vartheta(\nu,r,n)} + \frac{8}{3} \|U\|_{2 \to \infty}^2 \vartheta(\nu,r,n)
\end{equation*}
with probability at least $1 - 2n^{-\nu}$, where $\vartheta(\nu,r,n) = \nu \log n + r \log 9$. 
In the positive semidefinite setting, we have $\|U\|_{2 \to \infty} \leq \rho_n^{1/2} \lambda_r^{-1/2}$ so that
\begin{equation*}
    \psi_1
    \leq
    4 \sqrt{\rho_n \vartheta(\nu,r,n)} + \frac{8}{3} \lambda_r^{-1} \rho_n \vartheta(\nu,r,n)
\end{equation*}
with probability at least $1 - 2n^{-\nu}$.
For $\psi_2$, by \cref{lem:approximate_commute},
\begin{equation*}
    \psi_2
    \leq
    \lambda_r^{-1/2}\left(4 \sqrt{\rho_n \vartheta(\nu,r,n)} + \frac{8}{3} \lambda_r^{-1} \rho_n \vartheta(\nu,r,n) + \frac{2 \varsigma(\nu,n)^2}{\delta_r}\right)
\end{equation*}
with probability at least $1 - 2n^{-\nu}$.
For $\psi_3^{(k)}$, by \cref{lem:technical4_new},
\begin{equation*}
    \begin{split}
    \psi_3^{(k)}
    &\leq
    \|U_{\perp} \Lambda_{\perp}^{k}\|_{2 \to \infty}
    \left(\frac{8\|U\|_{2 \to \infty} \vartheta(\nu+2,r,n)}{3} + \sqrt{8 \rho_n \vartheta(\nu+2,r,n)}\right) \\
    &\leq
    \rho_n^{1/2} \lambda_{r+1}^{k-1/2} \left(\frac{8 \rho_n^{1/2} \vartheta(\nu+2,r,n)}{3 \lambda_r^{1/2}} + \sqrt{8 \rho_n \vartheta(\nu+2,r,n)}\right)
    \end{split}
\end{equation*} 
with probability at least $1 - n^{-(\nu+1)}$.
The above derivations use the fact that in the positive semidefinite setting, 
\begin{equation*}
    \|U\|_{2 \to \infty}
    \leq
    \rho_n^{1/2} \lambda_r^{-1/2},
    \quad
    \text{and}
    \quad
    \|U_{\perp} \Lambda_{\perp}^{1/2}\|_{2 \to \infty} 
    \leq
    \rho_n^{1/2}.
\end{equation*}
Taking a union over all $k$ satisfying $1 \le k \leq n \rho_n^{1/2}$, we obtain
\begin{equation*}
    \begin{split}
    \sum_{k=1}^{n \rho_n^{1/2}} \psi_3^{(k)} |\hat{\lambda}_{r}|^{-(k+1/2)}
    &\leq
    \rho_n^{1/2} |\hat{\lambda}_r|^{-1} \left(\frac{8 \rho_n^{1/2} \vartheta(\nu+2,r,n)}{3 \lambda_r^{1/2}} + \sqrt{8 \rho_n \vartheta(\nu+2,r,n)}\right) \sum_{k=1}^{n \rho_n^{1/2}} \left|\frac{\lambda_{r+1}}{\hat{\lambda}_r}\right|^{k-1/2} \\
    &\leq \frac{2\rho_n}{\lambda_r} \Bigl( \frac{8 \vartheta(\nu+2,r,n)}{3 \lambda_r^{1/2}} + \sqrt{8 \vartheta(\nu+2,r,n)}\Bigr) \sum_{k=1}^{n \rho_n^{1/2}}\Bigl(1 - \frac{\delta_r}{2 \lambda_r}\Bigr)^{k-1/2} \\
    &\leq \frac{4\rho_n}{\delta_r} \Bigl(\frac{8 \vartheta(\nu+2,r,n)}{3 \lambda_r^{1/2}} + \sqrt{8 \vartheta(\nu+2,r,n)}\Bigr) 
    \end{split}
\end{equation*}
with probability at least $1 - n^{-\nu}$, where the second inequality in the above display follows from \cref{eq:bound_gap_psd}.
Next, by \cref{lem:ex_2inf},
\begin{equation}
  \begin{split}
  \label{eq:EU_psd}
  \|EU\|_{2 \to \infty}
  &\leq
  \sqrt{8 \rho_n \vartheta(\nu + 1,r,n)} + \tfrac{8}{3}\|U\|_{2 \to \infty} \vartheta(\nu + 1,r,n) \\
  &\leq
  \frac{11}{2} \rho_n^{1/2} \sqrt{\vartheta(\nu + 1, r,n)}
  \end{split}
\end{equation}
with probability at least $1 - n^{-\nu}$, where the final inequality follows from the fact that $\|U\|_{2 \to \infty} \leq \rho_n^{1/2} \lambda_r^{-1/2}$ and $\lambda_{r} \geq \vartheta(\nu + 1, r, n)$.

Using the same argument as before, we also have
\begin{equation}
  \begin{split}
  \label{eq:EULambda_half_psd}
  \|EU \Lambda^{-1/2}\|_{2 \to \infty}
  &\leq
  \sqrt{8 \rho_n \vartheta(\nu + 1,r,n)} \lambda_r^{-1/2} + \tfrac{8}{3}\|U \Lambda^{-1/2}\|_{2 \to \infty} \vartheta(\nu + 1,r,n) \\
  &\leq
  \frac{11}{2} \rho_n^{1/2} \lambda_r^{-1/2} \sqrt{\vartheta(\nu + 1, r,n)}
  ,
  \end{split}
\end{equation}
where now $\|U \Lambda^{-1/2}\|_{2 \to \infty} \leq \rho_n^{1/2} \lambda_r^{-1}$.  \cref{eq:EULambda_half_psd} implies the bound in \cref{eq:EUlambda}.

Finally, we bound the probability of the event $\mathcal{E}_1$ in \cref{lem:loo}. Recall that $V^{[h]}_i$ denotes the $i$-th row of $V^{[h]} = (I - UU^{\top})\hat{U}^{[h]}$ and that $V^{[h]}$ is independent of $e_h^{\top} E$ as $\hat{U}^{[h]}$ is a function of the entries in $A^{[h]}$ which are mutually independent from the entries in the $h$-th row of $E$. Therefore, by the matrix Bernstein inequality (e.g., see \cite[Theorem~1.6]{tropp}), we have
\begin{equation*}
    \mathbb{P}\left(\|e_h^{\top} E V^{[h]}\| \geq t\right)
    \leq
    (r+1) \exp\left(-\frac{t^2/2}{\sigma_{h}^2 +
    \|V^{[h]}\|_{2 \to \infty} t/3}\right)
    ,
\end{equation*}
where $\sigma_h^2 = \sum_{j} P(i,j) (1 - P(i,j)) \|V^{[h]}_i\|^2 \leq \rho_n \|V^{[h]}\|_F^2$.
We thus have
\begin{equation}
    \label{eq:loo4} 
    \|e_h^{\top} E V^{[h]}\|_{2 \to \infty}
    \leq
    \sqrt{2 (\nu + 2) \rho_n \log n} \|V^{[h]}\|_{F} + \frac{2(\nu + 2)}{3} \|V^{[h]}\|_{2 \to \infty} \log n
\end{equation}
with probability at least $1 - 2n^{-(\nu + 1)}$. By taking a union over all $h \in [n]$, we conclude that $\mathcal{E}_1$ holds with probability at least $1 - 2n^{-\nu}$, where $C_1(\nu) = \sqrt{2(\nu+2)}$ and $C_2(\nu) = \tfrac{2}{3}(\nu + 2)$. 

For ease of presentation, we shall now drop explicit constants from our derivations. Readers who are interested in keeping track of these constants can continue to do so using the values given above. By combining the above bounds and noting that $\vartheta(\nu,n) \lesssim (n \rho_n)^{1/2}$, after some tedious algebraic manipulations, we have
\begin{gather*}
    \|R_0\|_{2 \to \infty}
    \lesssim
    \|U \Lambda^{1/2}\|_{2 \to \infty} \times \frac{n \rho_n}{\delta_r^2}
    , \\
    \|R_1\|_{2 \to \infty}
    \lesssim
    \|U\|_{2 \to \infty} \left(\frac{\rho_n^{1/2} (r^{1/2} + \log^{1/2}{n})}{\lambda_r^{1/2}} + \frac{n \rho_n}{\delta_r \lambda_r^{1/2}}\right)
    , \\
    \|R_2\|_{2 \to \infty}
    \lesssim
    \|U\|_{2 \to \infty} \left(\frac{\rho_n^{1/2} (r^{1/2} + \log^{1/2}{n})}{\lambda_r^{1/2}} + \frac{n \rho_n}{\delta_r \lambda_r^{1/2}}\right) + \frac{n \rho_n^{3/2}}{\delta_r^2} + \frac{\rho_n (r^{1/2} + \log^{1/2}{n})}{\delta_r}
    , \\
    \|Y_0\|_{2 \to \infty}
    \lesssim
    \frac{\rho_n^{1/2}(r^{1/2} + \log^{1/2}{n})}{\lambda_r^{1/2}} \left(\frac{n \rho_n}{\delta_r^2} + \frac{\rho_n^{1/2}(r^{1/2} + \log^{1/2}{n})}{\lambda_r}\right)
    , \\
    \|Y_1\|_{2 \to \infty}
    \leq
    \frac{(r \rho_n \log n)^{1/2} \times (n \rho_n)^{1/2}}{\delta_r \lambda_r^{1/2}} + \frac{\rho_n^{1/2} \log^{3/2}{n}}{\lambda_r^{1/2} \delta_r} + \|U\|_{2 \to \infty} \times \frac{n \rho_n}{\delta_r \lambda_r^{1/2}}
\end{gather*}
with probability at least $1 - O(n^{-\nu})$. Substituting these expressions into \cref{eq:wrapup_expansion} and using the fact that $\|U\|_{2 \to \infty} \leq \rho_n^{1/2} \lambda_r^{-1/2}$ and $\|U \Lambda^{1/2}\|_{2 \to \infty} \leq \rho_n^{1/2}$, we obtain
\begin{equation*}
    \hat{U} \hat{\Lambda}^{1/2} - U \Lambda^{1/2} W^{(n)}
    =
    E U \Lambda^{-1/2} W^{(n)} + \tilde{R}
    ,
\end{equation*}
where $\tilde{R}$ satisfies
\begin{equation*}
    \begin{split}
    \|\tilde{R}\|_{2 \to \infty}
    &\lesssim
    \frac{n \rho_n^{3/2}}{\delta_r^2} + \frac{ (\rho_n \log n)^{1/2}}{\lambda_r^{1/2}}\left(\frac{(r n \rho_n)^{1/2} + \log n}{\delta_r}\right) 
    \end{split}
\end{equation*}
with probability at least $1 - O(n^{-\nu})$. The stated result follows from multiplying both side of the above display by $(W^{(n)})^{\top}$ and then redefining (rewriting) this matrix without the transpose operation. Note that the two-to-infinity norm is right-orthogonal invariant, i.e., $\|M W\|_{2 \to \infty} = \|M\|_{2 \to \infty}$ for any $n \times r$ matrix $M$ and $r \times r$ orthogonal matrix $W$. This concludes the proof of \cref{thm:psd}.

\subsection{Proof of \cref{lem:loo}}
\label{sec:loo_proof}
Let $e_h$ denote the $h$-th elementary basis vector in $\mathbb{R}^{n}$. For each $h \in [n]$, consider the expansion
\begin{equation}
    \label{eq:loo0aa_psd}
    \begin{split} e_h^{\top} E (I - UU^{\top}) \hat{U} 
    &=
    e_h^{\top} E (I - UU^{\top}) \hat{U}^{[h]} \hat{U}^{[h]\top} \hat{U}
    +
    e_h^{\top} E (I - \hat{U}^{[h]} \hat{U}^{[h]\top}) \hat{U} \\
    &\qquad-
    e_h^{\top} E U U^{\top} (I - \hat{U}^{[h]} \hat{U}^{[h]\top}) \hat{U}
    .
\end{split}
\end{equation}
By submultiplicativity of the operator norm, together with the fact that $\|UU^{\top}\| = 1$ and $\|e_h\| = 1$ for each $h$, we have
\begin{gather*}
    \|e_h^{\top} E U U^{\top} (I - \hat{U}^{[h]} \hat{U}^{[h]\top}) \hat{U} \|
    \leq
    \|E\|
    \times
    \|(I - \hat{U}^{[h]} \hat{U}^{[h]\top}) \hat{U} \|, \\
    \|e_h^{\top} E (I - \hat{U}^{[h]} \hat{U}^{[h]\top}) \hat{U}\|
    \leq
    \|E\|
    \times
    \|(I - \hat{U}^{[h]} \hat{U}^{[h]\top}) \hat{U}\|
    .
\end{gather*}

Let $\hat{\lambda}_r^{[h]}$ denote the $r$-th largest eigenvalue of $A^{[h]}$. Under the event $\mathcal{E}_0$ specified in \cref{eq:event_0}, we have by Weyl's inequality that
\begin{equation*}
    \hat{\lambda}_r^{[h]} - \hat{\lambda}_{r+1}
    \geq
\lambda_{r} - \|E\| - \hat{\lambda}_{r+1}
    \geq
    \lambda_{r} - 2\|E\| - \lambda_{r+1}
    \geq
    \frac{1}{2} \delta_r
    .
\end{equation*}
So, by applying the general form of the Davis--Kahan theorem (e.g., see \cite[Theorem~VII.3.1]{bhatia}), we have
\begin{equation}
    \label{eq:loo00}
    \begin{split}
    \|(I -  \hat{U}^{[h]} \hat{U}^{[h]\top}) \hat{U} \|
    &\leq
    \frac{\|(A^{[h]} - A) \hat{U}^{[h]}\|}{\hat{\lambda}_r^{[h]} - \hat{\lambda}_{r+1}}
    \leq
    \frac{2\|(A^{[h]} - A) \hat{U}^{[h]}\|}{\delta_r}
    .
    \end{split}
\end{equation}
For the numerator term, by the construction of $A^{[h]}$, we have 
\begin{equation}
    \label{eq:loo01}
    \begin{split}
    \|(A^{[h]} - A) \hat{U}^{[h]}\|
    &\leq
    \|\operatorname{diag}(E_{1,h}, \dots, E_{n,h}) 1_{n \times 1} \hat{U}^{[h]}_{h}\|
    +
    \|e_h^{\top} E \hat{U}^{[h]}\| \\
    &=
    \|(E_{1,h}, \dots, E_{n,h})^{\top}\| \times \|\hat{U}^{[h]}_{h}\|
    +
    \|e_h^{\top} E \hat{U}^{[h]}\| \\
    &\leq
    \|E\|_{2 \to \infty} \times \|\hat{U}^{[h]}\|_{2 \to \infty} + \|e_h^{\top} E \hat{U}^{[h]}\| \\
    &\leq
    \|E\| \times \|\hat{U}^{[h]}\|_{2 \to \infty} + \|e_h^{\top} E \hat{U}^{[h]}\|
    .
  \end{split}
\end{equation}
Next, since $\|U^{\top} \hat{U}^{[h]}\| \leq 1$ and $\|\hat{U}^{\top} \hat{U}^{[h]}\| \leq 1$, it follows that
\begin{equation}
    \label{eq:loo02}
    \begin{split}
    \|\hat{U}^{[h]}\|_{2 \to \infty} &\leq \|UU^{\top} \hat{U}^{[h]}\|_{2 \to \infty} + \|(I - UU)^{\top} \hat{U}^{[h]}\|_{2 \to \infty} \\
    &\leq
    \|U\|_{2 \to \infty} + \|(I - UU^{\top})[\hat{U} \hat{U}^{\top} \hat{U}^{[h]} + (I - \hat{U} \hat{U}^{\top}) \hat{U}^{[h]}]\|_{2 \to \infty} \\ 
    &\leq
    \|U\|_{2 \to \infty} + \|(I - UU^{\top}) \hat{U}\|_{2 \to \infty} + \|(I - \hat{U} \hat{U}^{\top}) \hat{U}^{[h]}\|
    .
\end{split}
\end{equation}
Here, properties of the $\sin$ $\Theta$ distance (e.g., see \cite[Lemma~1]{cai2018rate}) yield
\begin{equation}
    \label{eq:sintheta_equivalence}
    \|(I - \hat{U} \hat{U}^{\top}) \hat{U}^{[h]}\|
    =
    \|\sin \Theta(\hat{U}, \hat{U}^{[h]})\|
    =
    \|(I - \hat{U}^{[h]}
  \hat{U}^{[h]\top}) \hat{U}\|
  .
\end{equation}
Hence, by substituting \cref{eq:loo02} into \cref{eq:loo01}, we obtain
\begin{equation}
    \begin{split}
    \label{eq:loo03}
    \|(A^{[h]} - A) \hat{U}^{[h]}\|
    &\leq
    \|E\|(\|U\|_{2 \to \infty} + \|(I - UU^{\top}) \hat{U}\|_{2 \to \infty}) \\
    &\qquad+
    \|E\| \times \|(I - \hat{U}^{[h]} \hat{U}^{[h]\top}) \hat{U}\| + \|e_h^{\top} E \hat{U}^{[h]}\|
    .
  \end{split}
\end{equation}
Substituting \cref{eq:loo03} into \cref{eq:loo00} and then rearranging terms yields
\begin{equation}
    \label{eq:loo04}
    \|(I - \hat{U}^{[h]}\hat{U}^{[h]\top}) \hat{U}\| 
    \leq
    \frac{4\|E\|(\|U\|_{2 \to \infty} + \|(I - UU^{\top})\hat{U}\|_{2 \to \infty}) + 4 \|e_h^{\top} E \hat{U}^{[h]}\|}{\delta_r}
\end{equation}
under the condition $\delta_r \geq 4\|E\|$ implied by the event $\mathcal{E}_{0}$.

Writing
$E \hat{U}^{[h]} = E U U^{\top} \hat{U}^{[h]} + E V^{[h]}$ leads to
\begin{equation}
    \begin{split}
    \label{eq:sintheta1_psd}
    \|e_h^{\top} E \hat{U}^{[h]}\|
    &\leq
    \|e_h^{\top} E U U^{\top} \hat{U}^{[h]}\| + \|e_h^{\top} E V^{[h]}\|
    \leq
    \|E U\|_{2 \to \infty} + \|e_h^{\top} E V^{[h]}\|
    ,
  \end{split}
\end{equation}
and hence 
\begin{equation}
    \label{eq:sintheta2_psd} \begin{split}
    \|(I - \hat{U}^{[h]}\hat{U}^{[h]\top}) \hat{U}\| 
    &\leq
    \frac{4\|E\|(\|U\|_{2 \to \infty} + \|(I - UU^{\top})\hat{U}\|_{2 \to \infty})}{\delta_r} \\
    &\qquad+
    \frac{4 \|E U\|_{2 \to \infty} + 4 \|e_h^{\top} E V^{[h]}\|}{\delta_r}
    .
  \end{split}
\end{equation}
Now, recall the definition of the event $\mathcal{E}_1$ in \cref{eq:e1_event}, and note that
\begin{equation*}
    \|V^{[h]}\|_{F}
    =
    \|\sin \Theta(\hat{U}^{[h]}, U)\|_{F}
    \leq
    \|\sin \Theta(\hat{U}^{[h]}, \hat{U})\|_{F} + \|\sin \Theta(\hat{U}, U)\|_F
    \leq
    \frac{4 r^{1/2} \|E\|}{\delta_r}
    .
\end{equation*}
Furthermore,
\begin{equation*}
    \begin{split}
    \|V^{[h]}\|_{2 \to \infty}
    &\leq
    \|(I - UU^{\top}) \hat{U} \hat{U}^{\top} \hat{U}^{[h]}\|_{2 \to \infty}
    +
    \|(I - UU^{\top}) (I - \hat{U} \hat{U}^{\top}) \hat{U}^{[h]}\|_{2 \to \infty} \\
    &\leq
    \|(I - UU^{\top}) \hat{U}\|_{2 \to \infty} + \|(I - \hat{U} \hat{U}^{\top}) \hat{U}^{[h]}\| \\
    &\le
    \|(I - UU^{\top}) \hat{U}\|_{2 \to \infty} + \|(I - \hat{U}^{[h]} \hat{U}^{[h]\top}) \hat{U}\|
    ,
    \end{split}
\end{equation*}
where the final inequality follows from \cref{eq:sintheta_equivalence}.

Substituting the above expressions for $\|V^{[h]}\|_{F}$ and $\|V^{[h]}\|_{2 \to \infty}$ into \cref{eq:e1_event} yields
\begin{equation}
    \label{eq:loo6_psd}
    \begin{split}
    \| e_h^{\top} E V^{[h]}\|
    &\leq
    \frac{4 C_1(\nu) \sqrt{r \rho_n \log n}\|E\|}{\delta_r} \\
    &\qquad+
    C_{2}(\nu) (\|(I - UU^{\top}) \hat{U}\|_{2 \to \infty} + \|(I - \hat{U}^{[h]} \hat{U}^{[h]\top}) \hat{U}\|) \log n
    .
 \end{split}
\end{equation}
Substituting \cref{eq:loo6_psd} into \cref{eq:sintheta2_psd} and then rearranging terms yields
\begin{equation}
    \label{eq:loo_not_yet_final_psd}
    \begin{split}
    \|(I - \hat{U}^{[h]} \hat{U}^{[h]\top}) \hat{U} \| 
    &\leq
    \frac{8 \|E\|(\|U\|_{2 \to \infty} + \|(I - UU^{\top})\hat{U}\|_{2\ \to \infty})}{\delta_r} \\
    &\qquad+
    \frac{8\|EU\|_{2 \to \infty}}{\delta_r} + \frac{32 C_1(\nu) \sqrt{r \rho_n \log n} \|E\|}{\delta_r^2}  \\
    &\qquad+
    \frac{8 C_2(\nu) \|(I - UU^{\top}) \hat{U}\|_{2 \to \infty} \log n}{\delta_r}
    ,
    \end{split}
\end{equation}
provided that $\delta_r \geq 8 C_{2}(\nu) \log n$. Another substitution of the above into \cref{eq:loo6_psd} yields
\begin{equation}
    \label{eq:loo_final_psd}
    \begin{split}
    \| e_h^{\top} E V^{[h]}\|
    &\leq
    \frac{4 C_1(\nu) \sqrt{r \rho_n \log n} \|E\|}{\delta_r} + C_2(\nu) \|(I - UU^{\top}) \hat{U}\|_{2 \to \infty} \log n \\
    &\qquad+
    \frac{8 C_2(\nu) \|E\|(\|U\|_{2 \to \infty} + \|(I - UU^{\top})\hat{U}\|_{2\ \to \infty}) \log n}{\delta_r} \\
    &\qquad+
    \frac{8 C_2(\nu) \|EU\|_{2 \to \infty} \log n}{\delta_r} + \frac{32 C_1(\nu) C_2(\nu) \sqrt{r \rho_n \log^{3}{n}} \|E\|}{\delta_r^2} \\
    &\qquad+
    \frac{8 C_2(\nu)^2 \|(I - UU^{\top}) \hat{U}\|_{2 \to \infty} \log^{2}{n}}{\delta_r}
    .
    \end{split}
\end{equation}
Finally, substituting \cref{eq:loo_final_psd} and \cref{eq:loo_not_yet_final_psd} into \cref{eq:loo0aa_psd} yields 
\begin{equation*}
    \begin{split}
    \|e_h^{\top} E(I - UU^{\top}) \hat{U}\|
    &\leq
    \frac{4 C_1(\nu) \sqrt{r \rho_n \log n} \|E\|}{\delta_r} + C_2(\nu) \|(I - UU^{\top}) \hat{U}\|_{2 \to \infty} \log n \\
    &\qquad+
    \frac{8 C_2(\nu) \|E\|(\|U\|_{2 \to \infty} + \|(I - UU^{\top})\hat{U}\|_{2\ \to \infty}) \log n}{\delta_r} \\
    &\qquad+
    \frac{8 C_2(\nu) \|EU\|_{2 \to \infty} \log n}{\delta_r} + \frac{32 C_1(\nu) C_2(\nu) \sqrt{r \rho_n \log^{3}{n}} \|E\|}{\delta_r^2} \\
    &\qquad+
    \frac{8 C_2(\nu)^2 \|(I - UU^{\top}) \hat{U}\|_{2 \to \infty} \log^{2}{n}}{\delta_r} \\
    &\qquad+
    \frac{16 \|E\|^2(\|U\|_{2 \to \infty} + \|(I - UU^{\top})\hat{U}\|_{2\ \to \infty})}{\delta_r} \\
    &\qquad+
    \frac{16 \|E\| \cdot \|EU\|_{2 \to \infty}}{\delta_r} + \frac{32 C_1(\nu) \sqrt{r \rho_n \log n} \|E\|^2}{\delta_r^2} \\
    &\qquad+
    \frac{8 C_2(\nu) \|E\| \cdot \|(I - UU^{\top}) \hat{U}\|_{2 \to \infty} \log n}{\delta_r}
    .
\end{split} 
\end{equation*}
Under the event $\mathcal{E}_0$, we can further obtain the simplification
\begin{equation}
    \label{eq:loo_final3_psd}
    \begin{split}
    \|e_h^{\top} E(I - UU^{\top}) \hat{U}\|
    &\leq
    \frac{16 C_1(\nu) \sqrt{r \rho_n \log n} \|E\|}{\delta_r} \\
    &\qquad+
    \frac{8 (\|E\| \cdot\|U\|_{2 \to \infty} + \|EU\|_{2 \to \infty}) (C_2(\nu) \log n  +  2 \|E\|)}{\delta_r} \\
    &\qquad+
    \|(I - UU^{\top}) \hat{U}\|_{2 \to \infty} \left(6 C_2(\nu) \log n + \frac{16 \|E\|^2 }{\delta_r}\right)
    .
    \end{split} 
\end{equation}
Note that the right hand side of \cref{eq:loo_final3_psd} does not depend on $h$. In summary, we have shown
\begin{equation}
    \label{eq:loo_final4_psd}
    \begin{split}
    \|E (I - UU^{\top}) \hat{U}\|_{2 \to \infty}
    &\leq
    \frac{16 C_1(\nu) \sqrt{r \rho_n \log n} \|E\|}{\delta_r} \\
    &\qquad+
    \frac{8 (\|E\| \cdot\|U\|_{2 \to \infty} + \|EU\|_{2 \to \infty}) (C_2(\nu) \log n + 2 \|E\|)}{\delta_r} \\
    &\qquad+
    \|(I - UU^{\top}) \hat{U}\|_{2 \to \infty} \left(6 C_2(\nu) \log n +  \frac{16 \|E\|^2 }{\delta_r}\right)
    ,
  \end{split} 
\end{equation}
as claimed. This completes the proof of \cref{lem:loo}.

\subsection{Proof of \cref{thm:general}}
\label{sec:proof_general}
In this proof we follow the same steps as in \cref{sec:proof_simple}, with the main difference being that we have to avoid terms depending on
$\|U_{\perp} |\Lambda_{\perp}|^{1/2}\|_{2 \to \infty}$
as we can no longer guarantee that $\|U_{\perp} |\Lambda_{\perp}|^{1/2}\|_{2 \to \infty} \leq \rho_n^{1/2}$. Indeed, in general
\begin{equation*}
    \|U_{\perp} |\Lambda_{\perp}|^{1/2}\|_{2 \to \infty}^2
    =
    \|U_{\perp} |\Lambda_{\perp}|U_{\perp}^{\top}\|_{\max}
    ,
\end{equation*}
and the entries of both $U_{\perp} |\Lambda_{\perp}|U_{\perp}^{\top}$ and $|P| = (P^2)^{1/2}$ are not directly related to those of $P$. In other words, there is no closed form expression for $|P|$ in terms of simple element-wise operations on $P$ when $P$ is indefinite. Similarly, we cannot guarantee that $\|U |\Lambda|^{1/2}\|_{2 \to \infty} \leq \rho_n^{1/2}$. 
Nevertheless, our approach involves the quantities
\begin{gather}
    \label{eq:U2inf_nonpsd0}
    \epsilon_0
    \coloneqq
    \|U |\Lambda|^{1/2}\|_{2 \to \infty}
    \leq
    \|U \Lambda\|_{2 \to \infty} \times \||\Lambda|^{-1/2}\|
    \leq
    n^{1/2} \rho_n \times |\lambda_r|^{-1/2}, \\
    \label{eq:U2inf_nonpsd2}
    \epsilon_1
    \coloneqq
    \|U_{\perp} \Lambda_{\perp}\|_{2 \to \infty}
    =
    \|U_{\perp} \Lambda_{\perp}^2 U_{\perp}^{\top}\|_{\max}^{1/2}
    \leq
    n^{1/2} \rho_n
    ,
\end{gather}
where the final inequalities above are due to the fact that $P^2$ is positive semidefinite, so 
\begin{equation}
    \label{eq:2inf_U_nonpsd1}
    \begin{split}
    \max\{\|U \Lambda\|_{2 \to \infty}^2, \|U_{\perp} \Lambda_{\perp}\|_{2 \to \infty}^2\}
    &=
    \max\{\|U \Lambda^2 U^{\top}\|_{\max}, \|U_{\perp} \Lambda_{\perp}^2 U_{\perp}^{\top}\|_{\max}\} \\
    &\leq
    \|P^2\|_{\max}
    \leq
    n \|P\|_{\max}^2
    \leq
    n \rho_n^2
    .
  \end{split}
\end{equation}
If $P$ were positive semidefinite, then \cref{eq:two_infU_psd} implies $\epsilon_1 \leq \rho_n^{1/2} \lambda_{r+1}^{1/2}$. The bounds in \cref{eq:U2inf_nonpsd0,eq:U2inf_nonpsd2} are thus weaker that those in \cref{eq:two_infU_psd} as $|\lambda_r/(n \rho_n)|^{1/2}$ is a decreasing function of $r$ (for $n$ fixed). 

Recall that, by convention, we arrange the eigenvalues of $A$ and $P$ in decreasing order of magnitude, i.e.,
\begin{equation*}
    |\hat{\lambda}_1|
    \geq
    |\hat{\lambda}_2|
    \geq
    \dots
    \geq
    |\hat{\lambda}_n|
    ,
    \qquad
    \text{and}
    \qquad
    |\lambda_1|
    \geq
    |\lambda_2|
    \geq
    \dots
    \geq
    |\lambda_n|
    .
\end{equation*}
Once again, as in previous proofs, consider the event
\begin{equation}
    \label{eq:select_r_general}
    \mathcal{E}_0
    =
    \{\delta_r \geq \max\{4 \|E\|, 8 C_{2}(\nu) \log n\}\}
    ,
\end{equation}
where $C_2(\nu) = \tfrac{2}{3}(\nu + 2)$.   
Now, consider
\begin{equation*}
    W^{(+)}
    =
    \argmin_{W \in \mathcal{O}_{r_{+}}} \|\hat{U}_{+}-U_{+} W\|_{F},
    \quad
    W^{(-)}
    =
    \argmin_{W \in \mathcal{O}_{r_{-}}} \|\hat{U}_{-}-U_{-} W\|_{F}
    ,
\end{equation*} 
where $r_{+}$ and $r_{-}$ are the number of positive and negative eigenvalues among $\{\lambda_1, \dots, \lambda_{r}\}$, respectively, where $U_{+}$ and $U_{-}$ are the eigenvectors of $P$ corresponding to the positive and negative eigenvalues among $\{\lambda_1,\dots, \lambda_r\}$, and where $\hat{U}_{+}$ and $\hat{U}_{-}$ are the eigenvectors of $A$ corresponding to the positive and negative eigenvalues among $\{\hat{\lambda}_1, \dots, \hat{\lambda}_r\}$. 
Using these matrices, define the block-structured orthogonal matrix
\begin{equation*}
    W^{(n)}
    =
    \begin{bmatrix}
        W^{(+)} & 0 \\
        0 & W^{(-)}
    \end{bmatrix}
    .
\end{equation*}
Let $J$ be the $r \times r$ diagonal matrix with diagonal entries equal to $-1$ or $1$ such that $\Lambda = |\Lambda| J$. Note that, under $\mathcal{E}_0$, we also have $|\hat{\Lambda}| = \hat{\Lambda} J$. Furthermore, $J W^{(n)} = W^{(n)} J$. Since
\begin{equation*}
    U^{\top} \hat{U} - W^{(n)}
    =
    \begin{bmatrix}
        U_{+}^{\top} \hat{U}_{+} - W^{(+)} & 0 \\
        0 & U_{-}^{\top} \hat{U}_{-} - W^{(-)}
    \end{bmatrix}
    +
    \begin{bmatrix}
        0 & U_{+}^{\top} \hat{U}_{-} \\
        U_{-}^{\top} \hat{U}_{+} & 0
    \end{bmatrix}
    ,
\end{equation*}
it holds that
\begin{equation*}
    \|U^{\top} \hat{U} - W^{(n)}\|
    \leq
    \max\{ \| U_{+}^{\top} \hat{U}_{+} - W^{(+)} \|, \|U_{-}^{\top} \hat{U}_{-} - W^{(-)} \| \}
    +
    \max\{ \|U_{+}^{\top} \hat{U}_{-}\|, \|U_{-}^{\top} \hat{U}_{+}\| \}
    .
\end{equation*}
By \cite[Lemma~6.7]{cape2019two} and the Davis--Kahan theorem, we have the bound
\begin{equation*}
    \| U_{+}^{\top} \hat{U}_{+} - W^{(+)} \|
    \leq
    \|(I - U_{+} U_{+}^{\top}) \hat{U}_{+})\|^2
    \leq
    \left(\frac{2 \|E\|}{\delta_r}\right)^2
    =
    \psi_0^2
    ,
\end{equation*}
which likewise applies to $\|U_{-}^{\top} \hat{U}_{-} - W^{(-)}\|$. Next, by the general form of the Davis--Kahan theorem \cite[Theorem~VII.3.1]{bhatia}, we obtain
\begin{equation*}
    \begin{split}
    \|U_{+}^{\top} \hat{U}_{-}\|
    &\leq
    \frac{\|U_{+}^{\top} (A - P) \hat{U}_{-}\|}{|\lambda_r|} \\
    &\leq
    \frac{\|U_{+}^{\top} E U_{-}\| + \|U_{+}^{\top} E ( I - U_{-} U_{-}^{\top}) \hat{U}_{-}\|}{|\lambda_r|} \\
    &\leq
    \frac{\|U^{\top} E U\| + \|E\| \times \|(I - U_{-}U_{-}^{\top}) \hat{U}_{-}\|}{|\lambda_r|} \\
    &\leq
    \frac{\psi_1 + \|E\| \psi_0}{|\lambda_r|}
    ,
    \end{split}
\end{equation*}
and similarly for $\|U_{-}^{\top} \hat{U}_{+}\|$. Combining the above bounds yields
\begin{equation}
    \label{eq:sintheta_bound_indefinite}
    \| U^{\top} \hat{U}- W^{(n)} \|
    \leq
    \psi_0^2
    +
    |\lambda_r|^{-1}(\psi_1 + \|E\| \psi_0)
    .
\end{equation}
Now, consider the decomposition
\begin{equation}
    \label{eq:decomp_general}
    \begin{split}
    \hat{U} |\hat{\Lambda}|^{1/2} - U |\Lambda|^{1/2} W^{(n)}
    &=
    (I - UU^{\top})\hat{U} |\hat{\Lambda}|^{1/2} \\
    &\qquad+
    U (U^{\top} \hat{U} |\hat{\Lambda}|^{1/2} J - |\Lambda|^{1/2} J U^{\top} \hat{U}) J \\
    &\qquad+
    U |\Lambda|^{1/2} J (U^{\top} \hat{U} - W^{(n)}) J
    ,
    \end{split}
\end{equation}
where we have used the fact that $J W^{(n)} J = W^{(n)}$. 

For $R_0 \coloneqq U |\Lambda|^{1/2} J (U^{\top} \hat{U}) - W^{(n)}) J$, applying \cref{eq:sintheta_bound_indefinite} yields
\begin{equation*}
    \begin{split}
    \|U |\Lambda|^{1/2} J (U^{\top} \hat{U} - W^{(n)}) J\|_{2 \to \infty}
    &\leq
    \|U |\Lambda|^{1/2}\|_{2 \to\infty} \times \|J (U^{\top} \hat{U} - W^{(n)}) J\| \\
    &\leq
    \|U |\Lambda|^{1/2}\|_{2 \to\infty} \times \|U^{\top} \hat{U} - W^{(n)} \| \\
    &\leq
    \epsilon_0 (\psi_0^2 + |\lambda_r|^{-1} \|E\| \psi_0 + |\lambda_r|^{-1} \psi_1)
    .
    \end{split}
\end{equation*}
Next, for $R_1 \coloneqq U ( U^{\top} \hat{U} |\hat{\Lambda}|^{1/2} J - |\Lambda|^{1/2} J U^{\top} \hat{U}) J$, we have
\begin{equation*}
    \|R_1\|_{2 \to \infty}
    \leq
    \|U\|_{2 \to \infty} \times \| |\Lambda|^{1/2} J U^{\top} \hat{U} - U^{\top} \hat{U} |\hat{\Lambda}|^{1/2} J\|
    \leq 
    \epsilon_0 |\lambda_r|^{-1/2} \psi_2
    ,
\end{equation*}
where $\psi_2$ is the upper bound for $\||\Lambda|^{1/2} J U^{\top} \hat{U} - U^{\top} \hat{U} |\hat{\Lambda}|^{1/2} J\|$ given in \cref{lem:approximate_commute}.

\subsubsection*{Bounding $(I - UU^{\top}) \hat{U} |\hat{\Lambda}|^{1/2}$}
Due to $\hat{U} |\hat{\Lambda}| = \hat{U} \hat{\Lambda} J = A\hat{U} J = (P \hat{U} + E \hat{U}) J$, we have
\begin{equation*}
    \begin{split}
    (I - U U^{\top}) \hat{U} |\hat{\Lambda}|^{1/2}
    &=
    (I - UU^{\top}) E \hat{U} |\hat{\Lambda}|^{-1/2} J + (I - UU^{\top}) P \hat{U} |\hat{\Lambda}|^{-1/2} J \\
    &=
    E \hat{U} |\hat{\Lambda}|^{-1/2} J - U U^{\top} E \hat{U} |\hat{\Lambda}|^{-1/2} J + (I - UU^{\top}) P \hat{U} |\hat{\Lambda}|^{-1/2} J
    .
    \end{split}
\end{equation*}
The term $U U^{\top} E \hat{U} |\hat{\Lambda}|^{-1/2} J$ is bounded using the same argument as in the positive semidefinite case, namely
\begin{equation*}
    \begin{split}
    \|U U^{\top} E \hat{U} |\hat{\Lambda}|^{-1/2} J\|_{2 \to \infty}
    &=
    \|U U^{\top} E \hat{U} |\hat{\Lambda}|^{-1/2} \|_{2 \to \infty}
    \leq
    \frac{\sqrt{2} \epsilon_0 (\psi_1 + \|E\| \psi_0)}{|\lambda_r|}
    .
    \end{split}
\end{equation*}
Let $\Pi_{U}^{\perp} = (I - UU^{\top})$. Observe that the term $(I - UU^{\top})P \hat{U} |\hat{\Lambda}|^{-1/2} J$ admits the expansion
\begin{equation*}
    \begin{split}
    \Pi_{U}^{\perp} P \hat{U} |\hat{\Lambda}|^{-1/2} J
    &=
    (\Pi_{U}^{\perp} P^2 \hat{U} \hat{\Lambda}^{-1} + \Pi_{U}^{\perp} PE \hat{U} \hat{\Lambda}^{-1}) |\hat{\Lambda}|^{-1/2} J\\
    &=
    (\Pi_{U}^{\perp} P^3 \hat{U} \hat{\Lambda}^{-2} + \Pi_{U}^{\perp} P^2 E \hat{U} \hat{\Lambda}^{-2} + \Pi_{U}^{\perp} P E \hat{U} \hat{\Lambda}^{-1}) |\hat{\Lambda}|^{-1/2} J\\
    &\dots \\
    &=
    \left(\Pi_{U}^{\perp} P^{m} \hat{U} \hat{\Lambda}^{-(m-1)} + \sum_{k=1}^{m-1} \Pi_{U}^{\perp} P^{k} E \hat{U} \hat{\Lambda}^{-k}\right)|\hat{\Lambda}|^{-1/2} J
    .
    \end{split}
\end{equation*}
We then have
\begin{equation*}
    \begin{split}
    &\|\Pi_{U}^{\perp} P^m \hat{U} \hat{\Lambda}^{-(m - 1)}|\hat{\Lambda}|^{-1/2} J\|_{2 \to \infty} \\
    &\qquad\qquad=
    \|U_{\perp} \Lambda_{\perp}^{m} U_{\perp}^{\top} \hat{U} |\hat{\Lambda}|^{-(m-1/2)}\|_{2 \to \infty} \\
    &\qquad\qquad\leq
    \|U_{\perp} \Lambda_{\perp}\|_{2 \to \infty} \times \|\Lambda_{\perp}\|^{m-1} \times \|U_{\perp}^{\top} \hat{U}\| \times \||\hat{\Lambda}|^{-(m-1/2)}\| \\
    &\qquad\qquad\leq
    \epsilon_1 \times \psi_0 \times |\hat{\lambda}_{r}|^{-1/2} \left|\frac{\lambda_{r+1}}{\hat{\lambda}_r}\right|^{(m-1)}
    .
    \end{split}
  \end{equation*}
Under $\mathcal{E}_0$, we have
\begin{equation}
    \label{eq:bound_gap}
    \frac{|\lambda_{r+1}|}{|\hat{\lambda}_r|}
    \leq
    1 - \frac{|\lambda_r| - |\lambda_{r+1}|}{2|\lambda_r|}
    .
\end{equation}
Setting $m = n \rho_n^{1/2} + 1$, the same argument as that for \cref{eq:bound_gap_psd3} yields
\begin{equation}
\label{eq:bound_gap_indefinite3}
    \|\Pi_{U}^{\perp} P^m \hat{U} \hat{\Lambda}^{-(m - 1)}|\hat{\Lambda}|^{-1/2} J\|_{2 \to \infty}
    \leq
    \sqrt{2} \epsilon_1 \psi_0 |\lambda_r|^{-1/2} \exp(-9 n^{1/2})
    .
\end{equation}
We now evaluate $\zeta_k \coloneqq \|\Pi_{U}^{\perp} P^k E  \hat{U} \hat{\Lambda}^{-k} |\hat{\Lambda}|^{-1/2} J\|_{2 \to \infty}$ for $k \geq 1$. We first write
\begin{equation*}
    \begin{split}
    \zeta_k
    &\leq
    \underbrace{\|\Pi_{U}^{\perp} P^k E U U^{\top} \hat{U} |\hat{\Lambda}|^{-(k+1/2)}\|_{2 \to \infty}}_{\zeta_k^{(1)}}
    +
    \underbrace{\|\Pi_{U}^{\perp} P^k E U_{\perp} U_{\perp}^{\top} \hat{U} |\hat{\Lambda}|^{-(k+1/2)}\|_{2 \to \infty}}_{\zeta_k^{(2)}}
    .
\end{split}
\end{equation*}
We then have
\begin{equation*}
    \begin{split}
    \zeta_k^{(2)}
    &=
    \|U_{\perp} \Lambda_{\perp}^{k}
    U_{\perp}^{\top} E U_{\perp} U_{\perp}^{\top} \hat{U} |\hat{\Lambda}|^{-(k+1/2)}\|_{2 \to \infty} \\
    &\leq
    \|U_{\perp} \Lambda_{\perp}\|_{2 \to \infty} \times \|\Lambda_{\perp}\|^{k-1} \times \|E\| \times  \|U_{\perp}^{\top} \hat{U}\| \times \|\hat{\Lambda}^{-1}\|^{(k+1/2)} \\
    &\leq
    \frac{\epsilon_1 \|E\| \psi_0}{|\hat{\lambda}_r|^{3/2}} \times \left|\frac{\lambda_{r+1}}{\hat{\lambda}_r}\right|^{k-1} \\
    &\leq
    \frac{2^{3/2} \epsilon_1 \|E\| \psi_0}{|\lambda_r|^{3/2}} \times \left(1 - \frac{|\lambda_{r}| - |\lambda_{r+1}|}{2|\lambda_r|}\right)^{k-1}
    ,
\end{split}
\end{equation*}
and hence
\begin{equation}
    \label{eq:zeta_k2_indefinite}
    \sum_{k=1}^{n^{1/2} \rho_n} \zeta_{k}^{(2)}
    \leq
    \frac{2^{3/2} \epsilon_1 \|E\|\psi_0}{|\lambda_r|^{3/2}}
    \times
    \left(1 - \frac{|\lambda_{r}| - |\lambda_{r+1}|}{2|\lambda_r|}\right)^{k-1}
    \leq
    \frac{2^{5/2} \epsilon_1 \|E\| \psi_0}{|\lambda_r|^{1/2} \delta_r}
    .
\end{equation}
Observe that we avoid directly bounding $\|U_{\perp} |\Lambda_{\perp}\|^{1/2}\|_{2 \to \infty}$. For $\zeta_k^{(1)}$, we have
\begin{equation*}
    \begin{split}
    \zeta_k^{(1)}
    &\leq
    \|U_{\perp} \Lambda_{\perp}^{k} U_{\perp}^{\top} E U \|_{2 \to \infty} \times \|\hat{\Lambda}^{-1}\|^{(k+1/2)} 
    \leq
    \psi_3^{(k)} \times |\hat{\lambda}_r|^{-(k+1/2)}
    ,
    \end{split}
\end{equation*}
where $\psi_3^{(k)}$ is an upper bound for  $\|U_{\perp} \Lambda_{\perp}^{k}  U_{\perp}^{\top} E U\|$ given in \cref{lem:technical4_new}.


In summary, we have
\begin{equation*}
    (I - UU^{\top})\hat{U} |\hat{\Lambda}|^{1/2}
    =
    E \hat{U} |\hat{\Lambda}|^{-1/2} J
    +
    R_2,
\end{equation*}
where $R_2$ satisfies
\begin{equation}
    \label{eq:residual_r2}
    \begin{split}
    \|R_2\|_{2 \to \infty}
    &\leq
    \frac{2^{1/2} \epsilon_0(\psi_1 + \|E\| \psi_0)}{|\lambda_r|} + \frac{\sqrt{2} \epsilon_1 \psi_0}{|\lambda_r|^{1/2}}\left(\exp(-18 \sqrt{n}) + \frac{4\|E\|}{\delta_r}\right) \\
    &\qquad+
    \sum_{k=1}^{n \rho_n^{1/2}} \psi_3^{(k)} |\hat{\lambda}_r|^{-(k+1/2)}
    .
    \end{split}
\end{equation}

\subsubsection*{Bounding $E \hat{U} |\hat{\Lambda}|^{-1/2} J$}
We once again have the expansion
\begin{equation*}
    \begin{split}
    E \hat{U} |\hat{\Lambda}|^{-1/2} J 
    &=
    E U |\Lambda|^{-1/2} J [W^{(n)} + (U^{\top} \hat{U} - W^{(n)})] \\
    &\qquad+
    E U |\Lambda|^{-1/2} J (|\Lambda|^{1/2} J U^{\top} \hat{U} - U^{\top} \hat{U} |\hat{\Lambda}|^{1/2} J) |\hat{\Lambda}|^{-1/2} J \\
    &\qquad+ 
    E (I - U U^{\top}) \hat{U} |\hat{\Lambda}|^{-1/2} J
    .
    \end{split}
\end{equation*}
Therefore,
\begin{gather*}
    \|E U |\Lambda|^{-1/2} J (U^{\top} \hat{U} - W^{(n)})\|_{2 \to \infty} 
    \leq
    \|E U |\Lambda|^{-1/2}\|_{2 \to \infty} \times (\psi_0^2 + |\lambda_r|^{-1} \|E\| \psi_0 + |\lambda_r|^{-1} \psi_1), \\ 
    \| E U |\Lambda|^{-1/2} (|\Lambda|^{1/2} J U^{\top} \hat{U} - U^{\top} \hat{U} |\hat{\Lambda}|^{1/2} J) |\hat{\Lambda}|^{-1/2} J \|
    \leq
    \|E U |\Lambda|^{-1/2}\|_{2 \to \infty} \times 2^{1/2} |\lambda_r|^{-1/2} \psi_2
    ,
\end{gather*}
where $\psi_2$ is an upper bound for $\||\Lambda|^{1/2} J U^{\top} \hat{U} - U^{\top} \hat{U} |\hat{\Lambda}|^{1/2} J\|$ given in \cref{lem:approximate_commute}, and we bounded  $\|U^{\top} \hat{U} - W^{(n)}\|$ using \cref{eq:sintheta_bound_indefinite}. 

The term $\|E(I - UU^{\top})\hat{U}|\hat{\Lambda}|^{-1/2} J\|_{2 \to \infty} = \|E(I - UU^{\top})\hat{U}|\hat{\Lambda}|^{-1/2} \|_{2 \to \infty}$ can be bounded using \cref{lem:loo} within the proof of \cref{thm:psd}. In summary, we have 
\begin{equation*}
    \begin{split}
    &\|E(I - UU^{\top}) \hat{U} |\hat{\Lambda}|^{-1/2} J\|_{2 \to \infty} \\
    &\qquad\qquad\leq
    \frac{16 \sqrt{2} C_1(\nu) \sqrt{r \rho_n \log n} \|E\|}{\delta_r |\lambda_r|^{1/2}} \\
    &\qquad\qquad\qquad+
    \frac{8 \sqrt{2}(\|E\| \cdot\|U\|_{2 \to \infty} + \|EU\|_{2 \to \infty}) (C_2(\nu) \log n + 2 \|E\|)}{\delta_r |\lambda_r|^{1/2}} \\
    &\qquad\qquad\qquad+
    |\lambda_r|^{-1}\|(I - UU^{\top}) \hat{U} |\hat{\Lambda}|^{1/2} \|_{2 \to \infty} \left(12 C_2(\nu) \log n + \frac{32 \|E\|^2 }{\delta_r}\right)
    .
    \end{split} 
\end{equation*}
Let $T_* = E U |\Lambda|^{-1/2} J W^{(n)}$. By combining the above expressions,
\begin{align*}
    (I - UU^{\top}) \hat{U} |\hat{\Lambda}|^{1/2}
    &=
    T_* + R_2 + R_3 + Y_0 + Y_1, \\
    \|Y_0\|_{2 \to \infty}
    &\leq
    \|T_*\|_{2 \to \infty} (\psi_0^2 + |\lambda_r|^{-1} \|E\| \psi_0 + |\lambda_r|^{-1} \psi_1 + 2^{1/2} |\lambda_r|^{-1/2} \psi_2), \\
    \|Y_1\|_{2 \to \infty}
    &\leq
    \frac{16 \sqrt{2} C_1(\nu)(r \rho_n \log n)^{1/2} \|E\|}{\delta_r |\lambda_r|^{1/2}}, \\
    &\qquad+
    \frac{8 \sqrt{2} (\|E\| \cdot \|U\|_{2 \to \infty} + \|EU\|_{2 \to \infty}) (C_{2}(\nu) \log n + 2 \|E\|)}{\delta_r |\lambda_r|^{1/2}}, \\
    \|Y_2\|_{2 \to \infty}
    &\leq
    \left(\frac{12 C_{2}(\nu) \log n}{|\lambda_r|} + \frac{32\|E\|^2}{\delta_r |\lambda_r|}\right) \|(I - UU^{\top}) \hat{U} |\hat{\Lambda}|^{1/2}\|_{2 \to \infty}
    .
\end{align*}
We therefore have
\begin{align*}
    &\|(I - UU^{\top})\hat{U} |\hat{\Lambda}|^{1/2}\|_{2 \to \infty} \\ 
    &\quad\leq
    \frac{1}{1 - \tfrac{12 C_{2}(\nu) \log n}{|\lambda_r|} - \frac{32 \|E\|^2}{\delta_r |\lambda_r|}} \bigl(\|T_*\|_{2 \to \infty} + \|R_2\|_{2 \to \infty} + \|Y_0\|_{2 \to \infty} + \|Y_1\|_{2 \to \infty} \bigr), \\
    &\|Y_2\|_{2 \to \infty} \\
    &\quad\leq
    \frac{24 C_{2}(\nu) \log n + 64 \delta_r^{-1} \|E\|^2}{|\lambda_r|} (\|T_*\|_{2 \to \infty} + \|R_2\|_{2 \to \infty} + \|Y_0\|_{2 \to \infty} + \|Y_1\|_{2 \to \infty} )
    ,
\end{align*}
provided that $|\lambda_r| \geq 24 C_{2}(\nu) \log n + 64 \delta_r^{-1} \|E\|^2$. 

\subsubsection*{Wrapping up}
Collecting the above bounds yields an expansion of the form
\begin{gather}
\label{eq:expansion_final1}
    \hat{U} |\hat{\Lambda}|^{1/2} - U |\Lambda|^{1/2} W^{(n)}
    =
    E U |\Lambda|^{-1/2} J W^{(n)}
    +
    R_0 + R_1 + R_2 + Y_0 + Y_1 + Y_2 
\end{gather}
where $T_* \coloneqq E U |\Lambda|^{-1/2} J W^{(n)}$ and the residual matrices $R_0$ through $Y_2$ satisfy
\begin{align*}
    \|R_0\|_{2 \to \infty}
    &\leq
    \epsilon_0 (\psi_0^2 + |\lambda_r|^{-1} \|E\| \psi_0 + |\lambda_r|^{-1} \psi_1), \\
    \|R_1\|_{2 \to \infty}
    &\leq
    \epsilon_0 |\lambda_r|^{-1/2} \psi_2, \\
    \|R_2\|_{2 \to \infty}
    &\leq
    \frac{2^{1/2} \epsilon_0(\psi_1 + \|E\| \psi_0)}{|\lambda_r|} + \frac{\sqrt{2} \epsilon_1 \psi_0}{|\lambda_r|^{1/2}}\left(\exp(-18 \sqrt{n}) \frac{4\|E\|}{\delta_r}\right) \\
    &\qquad+
    \sum_{k=1}^{n \rho_n^{1/2}} \psi_3^{(k)} |\hat{\lambda}_r|^{-(k+1/2)}, \\
    \|Y_0\|_{2 \to \infty}
    &\leq
    \|T_*\|_{2 \to \infty} (\psi_0^2 + |\lambda_r|^{-1} \|E\| \psi_0 + |\lambda_r|^{-1} \psi_1 + 2^{1/2} |\lambda_r|^{-1/2} \psi_2), \\
    \|Y_1\|_{2 \to \infty}
    &\leq
    \frac{16 \sqrt{2} C_{1}(\nu) (r \rho_n \log n)^{1/2} \|E\|}{\delta_r |\lambda_r|^{1/2}} \\
    &\qquad+
    \frac{8 \sqrt{2} (\|E\| \cdot \|U\|_{2 \to \infty} + \|EU\|_{2 \to \infty})(C_{2}(\nu) \log n + 2 \|E\|)}{\delta_r |\lambda_r|^{1/2}}, \\
    \|Y_2\|_{2 \to \infty}
    &\leq
    \frac{24 C_2(\nu) \log n + 64 \delta_r^{-1} \|E\|^2}{|\lambda_r|}\\
    &\qquad\times
    \bigl(\|T_*\|_{2 \to \infty} + \|R_2\|_{2 \to \infty} + \|Y_0\|_{2 \to \infty} + \|Y_1\|_{2 \to \infty} \bigr)
    .
\end{align*}
We now bound the quantities appearing in the above expressions.
First recall that
\begin{equation}
    \label{eq:epsilon_bounds}
    \epsilon_0
    =
    \|U |\Lambda|^{1/2}\|_{2 \to \infty} 
    \leq
    n^{1/2} \rho_n |\lambda_r|^{-1/2},
    \quad
    \epsilon_1
    =
    \|U_{\perp} \Lambda_{\perp}\|_{2 \to \infty} \leq n^{1/2} \rho_n
    . 
\end{equation}
Next, recall that, according to the conditions in \cref{eq:r_select_general2}, we have 
\begin{equation}
    \label{eq:select_r_general_proof}
    |\lambda_r|
    \geq
    \max\left\{24 C_{2}(\nu) \log n + \frac{64 \varsigma(\nu,n)^2}{\delta_r}, \sqrt{n \rho_n \vartheta(\nu+2,r,n)} \right\}
    ,
\end{equation}
where $\vartheta(c,r,n) = c \log n + r \log 9$ for any $c > 0$. Bounds for $\|E\|$, $\psi_0$, and $\psi_1$ are almost identical to those in the positive semidefinite case, namely with probability at least $1 - 3n^{-\nu}$ we have that
\begin{gather}
    \label{eq:e_spectral_norm}
    \|E\|
    \leq
    \varsigma(\nu,n), \\
    \label{eq:psi0_general}
    \psi_0
    \leq
    \frac{2\varsigma(\nu,n)}{\delta_r}, \\
    \label{eq:psi1_general}
    \psi_1
    \leq
    4 \sqrt{\rho_n \vartheta(\nu,r,n)} + \frac{8}{3} n \rho_n^{2} \lambda_r^{-2} \vartheta(\nu,r,n),
\end{gather}
hold simultaneously. Next, by \cref{lem:approximate_commute,eq:e_spectral_norm}, we have
\begin{equation}
    \label{eq:psi2_general}
    \begin{split}
    \psi_2
    &\leq
    8 |\lambda_r|^{-1/2} \sqrt{\rho_n \vartheta(\nu,r,n)} + \frac{16 n \rho_n^{2} \vartheta(\nu,r,n)}{3 |\lambda_r|^{5/2}} + \frac{4 \varsigma(\nu,n)^2}{|\lambda_r|^{1/2} \delta_r},
    \end{split}
\end{equation}
with probability at least $1 - 3n^{-\nu}$. Note that for bounding $\psi_1$ and $\psi_2$ we used the fact that $\|U\|_{2 \to \infty} \leq \|U |\Lambda|^{1/2}\|_{2 \to \infty} \times |\lambda_r|^{-1/2} \leq n^{1/2} \rho_n |\lambda_r|^{-1}$.
Now by \cref{eq:select_r_general_proof}, we have 
\begin{equation*}
    \sqrt{\rho_n \vartheta(\nu,r,n)}
    \geq
    n \rho_n^2 |\lambda_r|^{-2} \vartheta(\nu,r,n)
\end{equation*}
and the upper bounds for $\psi_1$ and $\psi_2$ simplify to 
\begin{gather}
    \label{eq:psi1_general2}
    \psi_1
    \leq
    \frac{20}{3} \sqrt{\rho_n \vartheta(\nu,r,n)}, \\
    \label{eq:psi2_general2}
    \psi_2
    \leq
    \frac{40}{3} |\lambda_r|^{-1/2} \sqrt{\rho_n \vartheta(\nu,r,n)} + \frac{4 \varsigma(\nu,n)^2}{\delta_r |\lambda_r|^{1/2}}
    .
\end{gather}
Next, by \cref{lem:technical4_new}, we have, {\em simultaneously} for all $1 \leq k \leq n$ that
\begin{equation}
    \label{eq:psi3_general}
    \begin{split}
    \psi_3^{(k)}
    &\leq
    \|U_{\perp} \Lambda_{\perp}^{k}\|_{2 \to \infty}\left(\sqrt{8 \rho_n \vartheta(\nu+2,r,n)} + \frac{8}{3}\|U\|_{2 \to \infty} \vartheta(\nu+2,r,n)\right), \\
    &\leq
    n^{1/2} \rho_n |\lambda_{r+1}|^{k-1} \left(\sqrt{8 \rho_n \vartheta(\nu+2,r,n)} + \frac{8}{3} n^{1/2} \rho_n \lambda_r^{-1} \vartheta(\nu+2,r,n)\right), \\
    &\leq
    n^{1/2} \rho_n |\lambda_{r+1}|^{k-1}  \times \frac{11}{2} \sqrt{\rho_n \vartheta(\nu+2,r,n)}
    \end{split}  
\end{equation}
with probability at least $1 - n^{-\nu}$, where the final inequality follows from \cref{eq:select_r_general_proof}. We therefore have
\begin{equation}
    \label{eq:sum_psi_3_k}
    \begin{split}
    \sum_{k=1}^{n} \frac{\psi_3^{(k)}}{ |\hat{\lambda}_r|^{k+1/2}}
    &\leq
    \frac{11 n^{1/2} \rho_n^{3/2} \sqrt{\vartheta(\nu+2,r,n)}}{2 |\hat{\lambda}_r|^{3/2}} \sum_{k=1}^{n \rho_n^{1/2}}\left|\frac{\hat{\lambda}_{r+1}}{\lambda_r}\right|^{k-1} \\
    &\leq
    \frac{11 \sqrt{2} n^{1/2} \rho_n^{3/2} \sqrt{\vartheta(\nu+2,r,n)}}{ |\lambda_r|^{3/2}} \sum_{k=1}^{n \rho_n^{1/2}} \left(1 - \frac{|\lambda_r| - |\lambda_{r+1}|}{|\lambda_r|}\right)^{k-1} \\
    &\leq
    \frac{11 \sqrt{2} n^{1/2} \rho_n^{3/2} \sqrt{\vartheta(\nu+2,r,n)}}{ |\lambda_r|^{1/2} \delta_r}
    .
    \end{split}
\end{equation}
We now bound $\|E U |\Lambda|^{1/2}\|_{2 \to \infty}$ and $\|E U \Lambda^{-1/2}\|_{2 \to \infty}$. By \cref{lem:ex_2inf}, we have
\begin{equation}
    \begin{split}
    \label{eq:EU_indefinite}
    \|EU\|_{2 \to \infty}
    &\leq
    \sqrt{8 \rho_n \vartheta(\nu + 1,r,n)} + \tfrac{8}{3}\|U\|_{2 \to \infty} \vartheta(\nu + 1,r,n) \\
    &\leq
    \frac{11}{2} \sqrt{\rho_n \vartheta(\nu + 1, r,n)}
    \end{split}
\end{equation}
with probability at least $1 - n^{-\nu}$, where the final inequality follows from the fact that $\|U\|_{2 \to \infty} \leq n^{1/2} \rho_n |\lambda_r|^{-1}$ and $|\lambda_{r}| \geq \sqrt{n \rho_n \vartheta(\nu+1,r,n)}$. \cref{eq:EU_indefinite} also implies
\begin{equation}
    \label{eq:EULambda_half_indefinite} 
    \|EU |\Lambda|^{-1/2}\|_{2 \to \infty} 
    \leq
    \frac{11}{2} |\lambda_r|^{-1/2} \sqrt{\rho_n \vartheta(\nu+1,r,n)}
    ,
\end{equation}
which yields the bound in \cref{eq:EUlambda_indefinite}. Finally, similar to the positive semidefinite case, we can show that the event $\mathcal{E}_1$ holds with probability at least $1 - 2n^{-\nu}$, where we take $C_1(\nu) = \sqrt{2(\nu+2)}$ and $C_2(\nu) = \tfrac{2}{3}(\nu + 2)$. 

Now, assume the above bounds hold. Substituting the corresponding quantities into the expressions for $R_0$ through $Y_1$ yields
\begin{gather*}
    \|R_0\|_{2 \to \infty}
    \lesssim
    \epsilon_0 \times \frac{n \rho_n}{\delta_r^2}, \\
    \|R_1\|_{2 \to \infty}
    \lesssim
    \frac{\epsilon_0 \rho_n^{1/2} (r^{1/2} + \log^{1/2}{n})}{|\lambda_r|} + \frac{\epsilon_0 n \rho_n}{|\lambda_r| \delta_r}, \\
    \|R_2\|_{2 \to \infty}
    \lesssim
    \epsilon_0 \left(\frac{\rho_n^{1/2} (r^{1/2} + \log^{1/2}{n})}{|\lambda_r|} + \frac{n \rho_n}{|\lambda_r| \delta_r}\right) + \frac{\epsilon_1 n \rho_n}{\delta_r^2 |\lambda_r|^{1/2}} + \frac{n^{1/2} \rho_n^{3/2} (r^{1/2} + \log^{1/2}{n})}{|\lambda_r|^{1/2} \delta_r}, \\
    \|Y_0\|_{2 \to \infty}
    \lesssim
    \frac{\rho_n^{1/2} (r^{1/2} + \log^{1/2}{n})}{|\lambda_r|^{1/2}}\left(\frac{n \rho_n}{\delta_r^2} + \frac{\rho_n^{1/2}(r^{1/2} + \log^{1/2}{n})}{|\lambda_r|}\right), \\
    \|Y_1\|_{2 \to \infty}
    \lesssim
    \frac{(r \rho_n \log n)^{1/2} \cdot (n \rho_n)^{1/2}}{\delta_r |\lambda_r|^{1/2}} + \frac{(\|U\|_{2 \to \infty} + \rho_n^{1/2}(r^{1/2} + \log^{1/2}{n})) ((n \rho_n)^{1/2} + \log n)}{\delta_r |\lambda_r|^{1/2}}.
\end{gather*}
By \cref{eq:epsilon_bounds}, the bound for $\|R_2\|_{2 \to \infty}$ can be further simplified to
\begin{equation*}
    \|R_2\|_{2 \to \infty}
    \lesssim
    \frac{n^{3/2} \rho_n^2}{\delta_r^2 |\lambda_r|^{1/2}} + \frac{n^{1/2} \rho_n^{3/2}(r^{1/2} + \log^{1/2}{n})}{\delta_r |\lambda_r|^{1/2}}
    .
\end{equation*}
Similarly, the bound for $\|Y_1\|_{2 \to \infty}$ can be further simplified to
\begin{equation*}
    \|Y_1\|_{2 \to \infty}
    \lesssim
    \frac{(r \rho_n \log n)^{1/2} \cdot (n \rho_n)^{1/2}}{\delta_r |\lambda_r|^{1/2}} + \frac{\rho_n^{1/2} \log^{3/2}{n}}{|\lambda_r|^{1/2} \delta_r}
    .
\end{equation*}
We then have
\begin{equation*}
    \|Y_0\|_{2 \to \infty} + \|Y_1\|_{2 \to \infty}
    \lesssim
    \frac{(r \rho_n \log n)^{1/2} \cdot (n \rho_n)^{1/2}}{\delta_r |\lambda_r|^{1/2}} + \frac{\rho_n^{1/2} \log^{3/2}{n}}{|\lambda_r|^{1/2} \delta_r}
    ,
\end{equation*}
where we had used the fact that $(n \rho_n)/\delta_r^2 \leq (n \rho_n)^{1/2}/\delta_r$, recalling that $\delta_r \geq 4 \varsigma(\nu,n) \geq (n \rho_n)^{1/2}$ (the dropped constant is immaterial here). 
Combining the above bounds we obtain 
\begin{gather*}
    \|R_0\|_{2 \to \infty} + \|R_1\|_{2 \to \infty}
    \lesssim
    \epsilon_0 \left(\frac{n \rho_n}{\delta_r^2} + \frac{\rho_n^{1/2} (r^{1/2} + \log^{1/2}{n})}{|\lambda_r|}\right), \\
    \|R_2\|_{2 \to \infty} + \|Y_0\|_{2 \to \infty} + \|Y_1\|_{2 \to \infty} 
    \lesssim
    \frac{n^{3/2} \rho_n^2}{\delta_r^2 |\lambda_r|^{1/2}} + \frac{(\rho_n \log n)^{1/2} ((r n \rho_n)^{1/2} + \log n)}{|\lambda_r|^{1/2} \delta_r}
    .
\end{gather*}
Recall that $T_* = E U |\Lambda|^{-1/2} J W^{(n)}$. Then, by \cref{eq:EULambda_half_indefinite}, we have
\begin{equation*}
    \begin{split}
    \|Y_2\|_{2 \to \infty}
    &\leq
    \frac{24 C_2(\nu) \log n + 64 \delta_r^{-1} \|E\|^2}{|\lambda_r|} \bigl(\|T_*\|_{2 \to \infty} + \|R_2\|_{2 \to \infty} + \|Y_0\|_{2 \to \infty} + \|Y_1\|_{2 \to \infty} \bigr) \\
    &\lesssim
    \frac{\rho_n^{1/2}(r^{1/2} + \log^{1/2}{n})}{|\lambda_r|^{1/2}}\Bigl(\frac{\log n}{|\lambda_r|} + \frac{n \rho_n}{\delta_r |\lambda_r|}\Bigr) + \|R_2\|_{2 \to \infty} + \|Y_0\|_{2 \to \infty} + \|Y_1\|_{2 \to \infty} \\
    &\lesssim
    \frac{n^{3/2} \rho_n^2}{\delta_r^2 |\lambda_r|^{1/2}} + \frac{(\rho_n \log n)^{1/2} ((r n \rho_n)^{1/2} + \log n)}{|\lambda_r|^{1/2} \delta_r}
    ,
    \end{split}
\end{equation*}
where the second inequality follows from the fact that $|\lambda_r| \geq 24 C_2(\nu) \log n + 64 \tfrac{\varsigma(\nu,n)^2}{\delta_r}$, and the third inequality follows from the fact that $(n \rho_n)/(|\lambda_r| \delta_r) \leq (n \rho_n)/\delta_r^2 \leq (n \rho_n)^{1/2}/\delta_r$. 
In summary, letting $\tilde{R} = R_0 + R_1 + R_2 + Y_0 + Y_1 + Y_2$, we have
\begin{equation*}
    \begin{split}
    \|\tilde{R}\|_{2 \to \infty}
    &\lesssim
    \epsilon_0 \left(\frac{n \rho_n}{\delta_r^2} + \frac{\rho_n^{1/2}(r^{1/2}
    +
    \log^{1/2}{n})}{|\lambda_r|}\right) + \frac{n^{3/2} \rho_n^2}{|\lambda_r|^{1/2} \delta_r^2}
    +
    \frac{(\rho_n \log n)^{1/2} ((r n \rho_n)^{1/2} + \log n)}{|\lambda_r|^{1/2} \delta_r}
    .
    \end{split}
\end{equation*}
Finally, as $\epsilon_0 \leq n^{1/2} \rho_n |\lambda_r|^{-1/2}$, the above bound simplifies to 
\begin{equation*}
    \begin{split}
    \|\tilde{R}\|_{2 \to \infty}
    &\lesssim
    \frac{n^{3/2} \rho_n^2}{|\lambda_r|^{1/2} \delta_r^2}
    +
    \frac{(\rho_n \log n)^{1/2}}{|\lambda_r|^{1/2}} \left(\frac{(r n \rho_n)^{1/2} + \log n}{\delta_r}\right)
    ,
    \end{split}
\end{equation*}
as desired. This concludes the proof of \cref{thm:general}.

\subsection{Proof of \cref{thm:generalU_ULambda}}
In this proof, for conciseness, we briefly sketch the main steps of the argument, since the technical details are similar to the proofs for \cref{thm:psd,thm:general}. Note that we will, on numerous occasions, silently apply \cref{eq:eigenvalues_relation} to replace $\hat{\lambda}_r$ with $\lambda_r$.

First, write $\hat{U} - U W^{(n)}$ as
\[
    \hat{U} - U W^{(n)}
    =
    \hat{U} - U U^{\top} \hat{U} + U U^{\top} \hat{U} - U W^{(n)}
    .
\]
For $R_0 \coloneqq U U^{\top} \hat{U} - U W^{(n)}$, we have
\[
    \|R_0\|_{2 \to \infty}
    \leq
    \|U\|_{2 \to \infty} \times \psi_0^2
    .
\]
Next, writing $\hat{U} = A \hat{U} \hat{\Lambda}^{1} = P \hat{U} \hat{\Lambda}^{-1} + P \hat{U} \hat{\Lambda}^{-1}$, we have
\[
    \hat{U} - U U^{\top} \hat{U}
    =
    (I - UU^{\top}) P \hat{U} \hat{\Lambda}^{-1}
    +
    E \hat{U} \hat{\Lambda}^{-1} - U U^{\top} E \hat{U} \hat{\Lambda}^{-1}
    .
\]
Consequently,
\[
    \|U U^{\top} E \hat{U} \hat{\Lambda}^{-1}\|_{2 \to \infty}
    \leq
    \frac{2 \|U\|_{2 \to \infty} (\psi_1 + \|E\| \psi_0)}{|\lambda_r|}
    .
\]
The term $(I - UU^{\top}) P \hat{U} \hat{\Lambda}^{-1}$ admits the expansion
\[
    \begin{split}
    \Pi_{U}^{\perp} P \hat{U} \hat{\Lambda}^{-1}
    &=
    \Pi_{U}^{\perp} P^{m} \hat{U} \hat{\Lambda}^{-m}
    +
    \sum_{k=1}^{m-1} \Pi_{U}^{\perp} P^{k} E \hat{U} \hat{\Lambda}^{-(k+1)} 
    \end{split}
\]
for any $m \geq 2$. Setting $m = n \rho_n^{1/2} + 1$, the same argument as for \cref{eq:bound_gap_indefinite3,eq:bound_gap_psd3} yields
\begin{equation}
\|\Pi_{U}^{\perp} P^{m} \hat{U} \hat{\Lambda}^{-m}\|_{2 \to \infty} \leq \begin{cases} 
2 \epsilon_1 \psi_0 |\lambda_r|^{-1} \exp(-9n^{1/2}), & \text{if $P$ is indefinite} \\
\sqrt{2} \rho_n^{1/2} \psi_0 \lambda_r^{-1/2} \exp(-9n^{1/2}) & \text{if $P$ is definite} 
\end{cases}. 
\end{equation}
Next, if $\kappa$ is positive semidefinite, then using the same derivations as for \cref{thm:psd} gives
\[
    \Bigl\|\sum_{k=1}^{n \rho_n^{1/2}} \Pi_{U}^{\perp} P^{k} E \hat{U} \hat{\Lambda}^{-(k+1)}\Bigr\|_{2 \to \infty}
    \leq
    \frac{2^{5/2} \rho_n^{1/2} \|E\| \psi_0}{\lambda_r^{1/2} \delta_r} + \sum_{k=1}^{n \rho_n^{1/2}} \psi_3^{(k)} \hat{\lambda}_{k}^{-(k+1)}
    .
\]
Otherwise, if $\kappa$ is indefinite, then using the same derivations as for \cref{thm:general} gives
\[
    \Bigl\|\sum_{k=1}^{n \rho_n^{1/2}} \Pi_{U}^{\perp} P^{k} E \hat{U} \hat{\Lambda}^{-(k+1)}\Bigr\|_{2 \to \infty}
    \leq \frac{8 \epsilon_1 \|E\| \psi_0}{|\lambda_r| \delta_r} + \sum_{k=1}^{n \rho_n^{1/2}} \psi_3^{(k)} \times |\hat{\lambda}_r|^{-(k+1)}
\]
Next, for $E \hat{U} \hat{\Lambda}^{-1}$, we have
\[
    E \hat{U} \hat{\Lambda}^{-1}
    =
    E U \Lambda^{-1}[W^{(n} + (U^{\top} \hat{U} - W^{(n)}) + (\Lambda U^{\top} \hat{U} \hat{\Lambda}^{-1} - U^{\top} \hat{U})] + E (I - UU^{\top}) \hat{U} \hat{\Lambda}^{-1}
    .
\]
Here,
\begin{gather*} 
    \|E U \Lambda^{-1} (U^{\top} \hat{U} - W^{(n)})\|_{2 \to \infty}
    \leq
    \|E U \Lambda^{-1}\|_{2 \to \infty} \times \psi_0^2, \\
    \|E U \Lambda^{-1} (\Lambda U^{\top} \hat{U} \hat{\Lambda}^{-1} - U^{\top} \hat{U})\|_{2 \to \infty}
    \leq
    2 |\lambda_r|^{-1} (\|E U \Lambda^{-1}\|_{2 \to \infty}( \psi_1 + \|E\| \psi_0)
    .
\end{gather*}
Finally, \cref{lem:loo} implies 
\begin{align*}
    &\|E (I - UU^{\top}) \hat{U} \hat{\Lambda}^{-1}\|_{2 \to \infty} \\
    &\qquad\leq
    \frac{32 C_1(\nu) (r \rho_n \log n)^{1/2} \|E\|}{\delta_r |\lambda_r|} \\
    &\qquad\qquad+
    \frac{16 (\|E\| \cdot \|U\|_{2 \to \infty} + \|E U\|_{2 \to \infty}) (C_{2}(\nu) \log n + 2 \|E\|)}{\delta_r |\lambda_r|} \\
    &\qquad\qquad+
    \|(I - UU^{\top}) \hat{U}\|_{2 \to \infty} \Bigl(\frac{12 C_2(\nu) \log n}{|\lambda_r|} + \frac{32 \|E\|^2}{\delta_r |\lambda_r|}\Bigr)
    .
\end{align*}
Now suppose $\kappa$ is positive semidefinite and $\lambda_r \geq 24 C_2(\nu) \log n + 64 \delta_r^{-1} \|E\|^2$. Collecting the above bounds we obtain an expression of the form
\[
    \hat{U} - U W^{(n)}
    =
    E U \Lambda^{-1} W^{(n)}
    +
    R_0 + R_1 + R_2 + Y_0 + Y_1
    ,
\]
where, with $T_* = E U \Lambda^{-1} W^{(n)}$, we have
\begin{gather*}
    \|R_0\|
    \leq
    \|U\|_{2 \to \infty} \times \psi_0^2,
    \quad
    \|R_1\|_{2 \to \infty} \leq 2 \lambda_r^{-1} 
    \|U\|_{2 \to \infty} (\psi_1 + \|E\| \psi_0), \\
    \|R_2\|_{2 \to \infty}
    \leq
    \frac{\sqrt{2} \rho_n^{1/2} \psi_0}{\lambda_r^{1/2}} \Bigl(\exp(-9 \sqrt{n}) + \frac{4 \|E\|}{\delta_r}\Bigr) + \sum_{k=1}^{n \rho_n^{1/2}} \psi_3^{(k)} \hat{\lambda}_r^{-(k+1)} \\
    \|Y_0\|_{2 \to \infty}
    \leq
    \|T_*\|_{2 \to \infty} (2\lambda_r^{-1}(\psi_1 + \|E\| \psi_0) + \psi_0^2) + \frac{32 C_{1}(\nu) (r \rho_n \log n)^{1/2} \|E\|}{\delta_r \lambda_r} \\
    \qquad\qquad+
    \frac{16 (\|E\| \cdot \|U\|_{2 \to \infty} + \|E U\|_{2 \to \infty}) (C_{\nu} \log n + 2 \|E\|)}{\delta_r \lambda_r}, \\
    \|Y_1\|_{2 \to \infty}
    \leq
    \frac{24 C_{2}(\nu) \log n + 64 \delta_r^{-1} \|E\|^2}{\lambda_r} \bigl(\|T_*\|_{2 \to \infty} + \|R_1\| _{2 \to \infty} + \|R_2\|_{2 \to \infty} + \|Y_0\|_{2 \to \infty} \bigr)
    .
\end{gather*}
Once again, substituting the bounds for $\psi_0, \psi_1, \psi_2, \psi_3^{(k)}$ and $\|E\|$ in the proof of \cref{thm:psd} together with some involved calculations, we obtain 
\begin{gather*}
    \hat{U} W^{(n)\top} - U
    =
    E U \Lambda^{-1} + \tilde{R}, \\
    \|\tilde{R}\|_{2 \to \infty}
    \lesssim
    \frac{n \rho_n^{3/2}}{\lambda_r^{1/2} \delta_r^2}
    +
    \frac{(\rho_n \log n)^{1/2}}{\lambda_r}\Bigl(\frac{\log n}{\delta_r} + \frac{(r n \rho_n)^{1/2}}{\delta_r}\Bigr)
\end{gather*} 
with probability at least $1 - O(n^{-\nu})$. The derivations when $\kappa$ is indefinite is similar and thus omitted.

For $\hat{U} \hat{\Lambda} - U \Lambda W^{(n)}$, we have
\[
    \hat{U} \hat{\Lambda} - U \Lambda W^{(n)}
    =
    (I - UU^{\top}) \hat{U} \hat{\Lambda} + U (U^{\top} \hat{U} \hat{\Lambda} - \Lambda U^{\top} \hat{U})
    +
    U \Lambda (U^{\top} \hat{U} - W^{(n)})
    .
\]
For $R_0 \coloneqq U \Lambda (U^{\top} \hat{U} - W^{(n)})$, we have
\[
    \|R_0\|_{2 \to \infty}
    \leq
    \|U \Lambda\|_{2 \to \infty} \times \|U^{\top} \hat{U} - W^{(n)}\| \leq n^{1/2} \rho_n \times \psi_0^2
    .
\]
Next, for $R_1 \coloneqq U (U^{\top} \hat{U} \hat{\Lambda} - \Lambda U^{\top} \hat{U})$ we have
\begin{equation*}
    \begin{split}
    \|R_1\|_{2 \to \infty}
    &\leq
    \|U\|_{2 \to \infty} \times \|U^{\top} \hat{U} \hat{\Lambda} - \Lambda U^{\top} \hat{U}\| \\
    &\leq
    \|U\|_{2 \to \infty} \times \|U^{\top} E \hat{U}\| \\
    &\leq
    \|U\|_{2 \to \infty} \times (\psi_1 + \|E\| \psi_0)
    .
    \end{split}
\end{equation*}
The term $(I - UU^{\top}) \hat{U} \hat{\Lambda}$ has the expansion
\begin{equation*}
    \begin{split}
    (I - UU^{\top}) \hat{U} \hat{\Lambda}
    &=
    (I - UU^{\top}) E \hat{U} + (I - UU)^{\top} P \hat{U} \\
    &=
    E \hat{U} + U U^{\top} E \hat{U} + (I - UU^{\top}) P^{m+1} \hat{U} \hat{\Lambda}^{-m} + \sum_{k=1}^{m} (I - UU^{\top}) P^{k} E \hat{U} \hat{\Lambda}^{-k}
    \end{split}
\end{equation*}
for any $m \geq 1$. 
Once again, we have
\[
    \|U U^{\top} E \hat{U}\|_{2 \to \infty}
    \leq
    \|U\|_{2 \to \infty} \leq \|U\|_{2 \to \infty} \times (\psi_1 + \E\|\psi_0)
    .
\]
Setting $m = n \rho_n^{1/2} + 1$ we have,
\begin{equation}
\|\Pi_{U}^{\perp} P^{m+1} \hat{U} \hat{\Lambda}^{-m}\|_{2 \to \infty} \leq \begin{cases} 
\epsilon_1 \psi_0 \exp(-9n^{1/2}), & \text{if $P$ is indefinite} \\
\rho_n^{1/2} \psi_0 \lambda_{r+1}^{1/2} \exp(-9n^{1/2}) & \text{if $P$ is definite} 
\end{cases}. 
\end{equation}
Next, if $\kappa$ is positive semidefinite, then using the same derivations as for \cref{thm:psd}, we have
\[
    \Bigl\|\sum_{k=1}^{n \rho_n^{1/2}} \Pi_{U}^{\perp} P^{k} E \hat{U} \hat{\Lambda}^{-k}\Bigr\|_{2 \to \infty}
    \leq \frac{2^{3/2} \rho_n^{1/2} \lambda_r^{1/2} \|E\| \psi_0}{\delta_r} + \sum_{k=1}^{n \rho_n^{1/2}} \psi_3^{(k)} \hat{\lambda}_{r}^{-k}   .
\]
Otherwise, if $\kappa$ is indefinite, then using the same derivations as for \cref{thm:general}, we have
\[
    \Bigl\|\sum_{k=1}^{\infty} \Pi_{U}^{\perp} P^{k} E \hat{U} \hat{\Lambda}^{-k}\Bigr\|_{2 \to \infty}
    \leq \frac{4 \epsilon_1 \|E\| \psi_0}{\delta_r} + \sum_{k=1}^{n \rho_n^{1/2}} \psi_3^{(k)} |\hat{\lambda}_r|^{-k}
    .
\]
Next, write
\[
    E \hat{U}
    =
    E U [W^{(n} + (U^{\top} \hat{U} - W^{(n)})] + E (I - UU^{\top}) \hat{U}
    .
\]
\cref{lem:loo} then implies
\[
    \begin{split}
        \|E (I - UU^{\top}) \hat{U}\|_{2 \to \infty} 
        &\leq
        \frac{16 C_1(\nu)(r \rho_n \log n)^{1/2} \|E\|}{\delta_r} \\
        &\qquad+
        \frac{8 (\|E\| \cdot \|U\|_{2 \to \infty} + \|E U\|_{2 \to \infty}) (C_{\nu} \log n + 2 \|E\|)}{\delta_r} \\
        &\qquad+
        \|(I - UU^{\top}) \hat{U} \hat{\Lambda}\|_{2 \to \infty} \Bigl(\frac{12 C_2(\nu) \log n}{|\lambda_r|} + \frac{32 \|E\|^2}{\delta_r |\lambda_r|}\Bigr)
        .
  \end{split}
\]
Now, suppose $\kappa$ is positive definite and $\lambda_r \geq 24 C_2(\nu) \log n + 64 \delta_r^{-1} \|E\|^2$. Collecting the above bounds we obtain an expression of the form
\[
    \hat{U} \hat{\Lambda} - U \Lambda W^{(n)}
    =
    E U W^{(n)} + R_0 + R_1 + R_2 + Y_0 + Y_1
    ,
\]
where, with $T_* = E U W^{(n)}$, we have
\begin{gather*}
    \|R_0\|
    \leq
    n^{1/2} \rho_n \times \psi_0^2,
    \quad \|R_1\|_{2 \to \infty} \leq  \|U\|_{2 \to \infty} (\psi_1 + \|E\| \psi_0), \\
    \|R_2\|_{2 \to \infty}
    \leq
    \rho_n^{1/2} \psi_0 \Bigl( \lambda_{r+1}^{1/2} \exp(-9n^{1/2}) + \frac{2^{3/2} \lambda_r^{1/2} \|E\|}{\delta_r} \Bigr) + \sum_{m=1}^{n \rho_n^{1/2}} \psi_3^{(k)} \hat{\lambda}_r^{-k} \\
    \|Y_0\|_{2 \to \infty}
    \leq
    \|T_*\|_{2 \to \infty} \psi_0^2 + \frac{16 C_{1}(\nu) (r \rho_n \log n)^{1/2} \|E\|}{\delta_r} \\
    \qquad\qquad\qquad\qquad\qquad\qquad\qquad+
    \frac{8 (\|E\| \cdot \|U\|_{2 \to \infty} + \|E U\|_{2 \to \infty}) (C_{\nu} \log n + 2 \|E\|)}{\delta_r}, \\
    \|Y_1\|_{2 \to \infty}
    \leq
    \frac{24 C_{2}(\nu) \log n + 64 \delta_r^{-1} \|E\|^2}{|\lambda_r|} \bigl(\|T_*\|_{2 \to \infty} + \|R_1\| _{2 \to \infty} + \|R_2\|_{2 \to \infty} + \|Y_0\|_{2 \to \infty} \bigr)
    .
\end{gather*}
Once again, substituting the bounds for $\psi_0, \psi_1, \psi_2, \psi_3^{(k)}$ and $\|E\|$ in the proof of \cref{thm:psd} together with some involved calculations, we obtain 
\begin{gather*}
    \hat{U} \hat{\Lambda} W^{(n)\top} - U \Lambda
    =
    E U + \tilde{R}, \\
    \|\tilde{R}\|_{2 \to \infty}
    \lesssim
    \frac{n^{3/2} \rho_n^{2}}{\delta_r^2} + (\rho_n \log n)^{1/2}\Bigl(\frac{\log n}{\delta_r} + \frac{(r n \rho_n)^{1/2}}{\delta_r}\Bigr)
\end{gather*} 
with probability at least $1 - O(n^{-\nu})$. The derivations when $\kappa$ is similar and we therefore omit the details. This completes the proof of \cref{thm:generalU_ULambda}.

\subsection{Deterministic perturbation bound}
\label{sec:deterministic} We now state a deterministic perturbation
bound for the $r$ leading eigenvectors $U$ of a symmetric matrix $M$
when it is additively perturbed by an arbitrary noise matrix $E$. The
result can be viewed as a generalization or reformulation of
\cref{eq:expansion_hatU} in \cref{thm:generalU_ULambda}. While there
are numerous terms in the expansion, they all depend on quantities
that are {\em linear} in $E$ and thus, for many inference problems,
can be bounded using standard matrix perturbation and matrix
concentration inequalities. In particular \cref{eq:expansion_hatU}
follows from \cref{thm:deterministic} by bounding $\|E\|$ using
\cref{lem:bandeira_vanhandel}, bounding $\psi_0$ using the Davis-Kahan
theorem (given an upper bound for $\|E\|$), bounding $\psi_1$ using
\cref{lem:technical2a}, bounding $\|U_{\perp} \Lambda_{\perp}
U_{\perp}^{\top} E U\|_{2 \to \infty}$ using
\cref{lem:technical4_new}, bounding $\|EU\|_{2 \to \infty}$ and
$\psi_*$ using \cref{lem:ex_2inf}, bounding $\eta$ using
\cref{lem:eta_bound}, and finally, verifying
\cref{eq:E1_general_noise} using a matrix Bernstein inequality (see
\cref{eq:loo4}). 
\begin{theorem}[Deterministic row-wise eigenvector perturbation bound]
\label{thm:deterministic} 
Let $M$ and $E$ be symmetric $n \times n$ matrices, and define $\hat{M} = M + E$. For a given $r \geq 1$, let the diagonal matrices $\hat{\Lambda}$ and $\Lambda$ contain the $r$ largest in magnitude eigenvalues of $\hat{M}$ and $M$, respectively, and let $\hat{U}$ and $U$ be the $n \times r$ matrices whose orthonormal columns are the corresponding eigenvectors of $\hat{M}$ and $M$. Also, for $h \in [n]$, denote by $\hat{M}^{[h]}$ the matrix obtained by replacing the entries in the $h$th row and $h$th column of $\hat{M}$ with those in the $h$th row and $h$th column of $M$, and let $\hat{U}^{[h]}$ be defined similar to $\hat{U}$, but with $\hat{M}^{[h]}$ in place of $\hat{M}$. Let $V^{[h]} = (I - UU^{\top}) \hat{U}^{[h]}$ and $\delta_{r} = |\lambda_{r}| - |\lambda_{r+1}|$. Suppose there exists quantities $\alpha > 0$ and $\beta > 0$ such that the following conditions
\begin{gather}
\label{eq:E0_general_noise}
\mathcal{E}_0 = \{ \delta_r \geq \max\{4 \|E\|, 8 \beta \} \}, \\
\label{eq:E1_general_noise}
\mathcal{E}_1 = \{ \forall h \in [n], \|e_h^{\top} E V^{[h]} \| \leq \alpha \|V^{[h]}\|_{F} + \beta \|V^{[h]}\|_{2 \to \infty}\}, \\
\label{eq:E2_general_noise}
\mathcal{E}_2 = \{ |\lambda_r| \geq 24 \beta + 64 \delta_r^{-1} \|E\| \}.
\end{gather}
hold simultaneously. 
We then have
\[ \hat{U} W^{(n)} - U = E U \Lambda^{-1} + R \]
where the residual matrix $R$ satisfies
\begin{gather*}
    \|R\|_{2 \to \infty} \leq r_{2,\infty}^{(0)} + r_{2,\infty}^{(1)} + r_{2, \infty}^{(2)} + y_{2, \infty}^{(0)} + y_{2,\infty}^{(1)}, \\
    r_{2, \infty}^{(0)} = \|U\|_{2 \to \infty} \psi_0^2, \quad r_{2, \infty}^{(1)} =  \frac{2 \|U\|_{2 \to \infty} (\psi_1 + \|E\| \psi_0)}{|\lambda_r|}, \\
    r_{2, \infty}^{(2)} = \frac{8 \|U_{\perp} \Lambda_{\perp}\|_{2 \to \infty} \times \|E\| \psi_0}{|\lambda_r| \delta_r}  + \frac{4 \|U_{\perp} \Lambda_{\perp} U_{\perp}^{\top} E U\|_{2 \to \infty}}{\lambda_r^2} + \frac{16 \|U_{\perp} \Lambda_{\perp}\|_{2 \to \infty} \times  \eta}{\lambda_r^2 \delta_r} \\
   y_{2,\infty}^{(0)} = \psi_* \Bigl(\frac{\psi_1 + \|E\| \psi_0}{2 |\lambda_r|} + \psi_0^2\Bigr) + \frac{32 \alpha r^{1/2} \|E\|}{|\lambda_r| \delta_r} + \frac{(\|E\| \cdot \|U\|_{2 \to \infty} + \|EU\|_{2 \to \infty})(16 \beta + 32 \|E\|)}{\delta_r |\lambda_r|}, \\
    y_{2,\infty}^{(1)} =  \frac{24 \beta + 64 \delta_r^{-1} \|E\|^2}{|\lambda_r|}\Bigl(\psi_{*} + r_{2,\infty}^{(1)} + r_{2,\infty}^{(2)} + y_{2,\infty}^{(0)}\Bigr)
\end{gather*} 
and $\psi_0, \psi_1, \eta, \psi_*$ are defined as
\begin{gather*} 
\psi_0 := \|(I - UU^{\top}) \hat{U} \|, \quad \psi_1 := \|U^{\top} E U \|, \quad \eta := \|\Lambda_{\perp} U_{\perp}^{\top} E U \|, \quad
\psi_* := \|E U \Lambda^{-1}\|_{2 \to \infty}.
\end{gather*}
If $M$ is positive semidefinite then the quantity $r_{2,\infty}^{(2)}$ can be simplified to 
\begin{gather*}
   r_{2, \infty}^{(2)} = \frac{2^{5/2} \|U_{\perp} \Lambda_{\perp}^{1/2}\|_{2 \to \infty} \times \|E\| \psi_0}{\lambda_r^{1/2} \delta_r}  +  \frac{8 \|U_{\perp} \Lambda_{\perp}^{1/2}\|_{2 \to \infty}  \tilde{\eta}}{\lambda_r \delta_r} 
\end{gather*}
where now $\tilde{\eta} := \|\Lambda_{\perp}^{1/2} U_{\perp}^{\top} E U \|$. 
\end{theorem}

\begin{proof}
The terms $r_{2, \infty}^{(0)}, r_{2, \infty}^{(1)}$ are derived the same way as the upper bounds for $\|R_0\|_{2 \to \infty}$ and $\|R_1\|_{2 \to \infty}$ 
in the proof of \cref{thm:generalU_ULambda}. 
Next, if $\mathcal{E}_1$ holds, then we can apply the same arguments as that in the proof of \cref{lem:loo} to obtain
\begin{equation}
\label{eq:loo_general_noise}
\begin{split}
\|E(I - UU^{\top}) \hat{U}\| &\leq
    \frac{16 \alpha r^{1/2} \|E\|}{\delta_r} \\
    &\qquad+
    \frac{ (\|E\| \cdot\|U\|_{2 \to \infty} + \|EU\|_{2 \to \infty}) (8 \beta  + 16 \|E\|)}{\delta_r} \\
    &\qquad+
    \|(I - UU^{\top}) \hat{U}\|_{2 \to \infty} \left(6 \beta +  \frac{16 \|E\|^2 }{\delta_r}\right).
 \end{split}
\end{equation}
Given \cref{eq:loo_general_noise}, the terms $y_{2,\infty}^{(0)}$ and $y_{2, \infty}^{(1)}$ are derived the same way as the upper bounds for $\|Y_0\|_{2 \to \infty}$ and $\|Y_1\|_{2 \to \infty}$ in the proof of \cref{thm:generalU_ULambda}. We now derive the expression for $r_{2, \infty}^{(2)}$, which is similar to the  upper bound for $\|R_2\|_{2 \to \infty}$ in \cref{thm:generalU_ULambda}.

First, we have
\[ (I- UU^{\top}) M \hat{U} \hat{\Lambda}^{-1} = (I - UU^{\top}) M^{m} \hat{U} \hat{\Lambda}^{-m} + \sum_{k=1}^{m-1} (I - UU^{\top}) M^{k} E \hat{U} \hat{\Lambda}^{-(k+1)}\]
which holds for any $m \geq 2$. 
We also have
\[ \Pi_{U}^{\perp} M^{k} E \hat{U} \hat{\Lambda}^{-(k+1)} = \Pi_{U}^{\perp} M^{k} E \Pi_{U}^{\perp} \hat{U} \hat{\Lambda}^{-(k+1)} + \Pi_{U}^{\perp} M^{k} E U U^{\top} \hat{U} \hat{\Lambda}^{-(k+1)} \]
where $\Pi_{U}^{\perp} = I - UU^{\top}$. 
Now suppose $M$ is positive semidefinite. Then
\[ \begin{split} \|\Pi_{U}^{\perp} M^{k} E \Pi_{U}^{\perp} \hat{U} \hat{\Lambda}^{-(k+1)}\|_{2 \to \infty} &\leq \|U_{\perp} \Lambda_{\perp}^{1/2}\|_{2 \to \infty} \times \|\Lambda_{\perp}\|^{k-1/2} \times \|E\| \times \|(I - UU^{\top}) \hat{U} \| \times \|\hat{\Lambda}^{-(k+1)}\| \\
& \leq \|U_{\perp} \Lambda_{\perp}^{1/2}\|_{2 \to \infty} \times \|E\| \times \psi_0 \times |\hat{\lambda}_{r}|^{-3/2} \times \Bigl(\frac{\lambda_{r+1}}{|\hat{\lambda}_r|}\Bigr)^{k - 1/2} \\
& \leq \|U_{\perp} \Lambda_{\perp}^{1/2}\|_{2 \to \infty} \times \|E\| \times \psi_0 \times 2^{3/2} \lambda_r^{-3/2} \times \Bigl(1 - \frac{\lambda_r - \lambda_{r+1}}{2 \lambda_r} \Bigr)^{k-1/2}
\end{split} \]
where the final inequality follows from the same argument as that in \cref{eq:bound_gap_psd,eq:eigenvalues_relation}.
Similarly, we also have
\[ \begin{split} \|\Pi_{U}^{\perp} M^{k} E U U^{\top} \hat{U} \hat{\Lambda}^{-(k+1)}\|_{2 \to \infty} &\leq \|U_{\perp} \Lambda_{\perp}^{1/2}\|_{2 \to \infty} \times \|\Lambda_{\perp}\|^{k-1} \times \|\Lambda_{\perp}^{1/2} U_{\perp}^{\top} E U\| \times \|U^{\top} \hat{U}\| \times \|\hat{\Lambda}^{-(k+1)}\| \\
& \leq \|U_{\perp} \Lambda_{\perp}^{1/2}\|_{2 \to \infty} \times \tilde{\eta} \times |\hat{\lambda}_{r}|^{-2} \times \Bigl(\frac{\lambda_{r+1}}{|\hat{\lambda}_r|}\Bigr)^{k - 1} \\
& \leq \|U_{\perp} \Lambda_{\perp}^{1/2}\|_{2 \to \infty} \times \tilde{\eta} \times 4 \lambda_r^{-2} \times \Bigl(1 - \frac{\lambda_r - \lambda_{r+1}}{2 \lambda_r} \Bigr)^{k-1/2}.
\end{split} \]
Finally, 
\[ \|(I - UU^{\top})M^{m} \hat{\Lambda}^{-m}\| \leq \Bigl(1 - \frac{\lambda_r - \lambda_{r+1}}{2 \lambda_r}\Bigr)^{m}.  \]
As $\delta_r = \lambda_{r} - \lambda_{r+1} > 0$, we have
\[ \lim_{m \rightarrow \infty} \|(I - UU^{\top})M^{m} \hat{\Lambda}^{-m}\| \leq \lim_{m \rightarrow \infty} \Bigl(1 - \frac{\lambda_r - \lambda_{r+1}}{2 \lambda_r}\Bigr)^{m}  = 0.\]
We also have
\[ \begin{split} \sum_{k=1}^{\infty} \|\Pi_{U}^{\perp} M^{k} E \Pi_{U}^{\perp} \hat{U} \hat{\Lambda}^{-(k+1)}\|_{2 \to \infty} &\leq 2^{3/2} \lambda_r^{-3/2} \|U_{\perp} \Lambda_{\perp}^{1/2}\|_{2 \to \infty} \times  \|E\| \psi_0 \sum_{k=1}^{\infty} \Bigl(1 - \frac{\lambda_r - \lambda_{r+1}}{2 \lambda_r} \Bigr)^{k-1/2} \\ & \leq 2^{5/2} \lambda_r^{-1/2} \delta_r^{-1} \|U_{\perp} \Lambda_{\perp}^{1/2}\|_{2 \to \infty} \times \|E\| \psi_0 \end{split} \]
\[ \begin{split} \sum_{k=1}^{\infty} \|\Pi_{U}^{\perp} M^{k} E U U^{\top} \hat{U} \hat{\Lambda}^{-(k+1)}\|_{2 \to \infty} &\leq 4 \lambda_r^{-2} \|U_{\perp} \Lambda_{\perp}^{1/2}\|_{2 \to \infty} \times  \tilde{\eta} \times \sum_{k=1}^{\infty} \Bigl(1 - \frac{\lambda_r - \lambda_{r+1}}{2 \lambda_r} \Bigr)^{k-1} \\ & \leq 8 \lambda_r^{-1} \delta_r^{-1} \|U_{\perp} \Lambda_{\perp}^{1/2}\|_{2 \to \infty} \tilde{\eta} \end{split} \]
Combining the above inequalities we obtain the expression
\[   r_{2, \infty}^{(2)} = \frac{2^{5/2} \|U_{\perp} \Lambda_{\perp}^{1/2}\|_{2 \to \infty} \times \|E\| \psi_0}{\lambda_r^{1/2} \delta_r}  +  \frac{8 \|U_{\perp} \Lambda_{\perp}^{1/2}\|_{2 \to \infty} \times  \tilde{\eta}}{\lambda_r \delta_r}  \]
when $M$ is positive semidefinite. 

Now, suppose $M$ is possibly indefinite. We then have 
\[ \begin{split} (I - UU^{\top}) M \hat{U} \hat{\Lambda}^{-1} &= (I - UU^{\top}) M^{m} \hat{U} \hat{\Lambda}^{-m} +  
(I - UU^{\top}) M E U U ^{\top} \hat{U} \hat{\Lambda}^{-2} \\ &+ \sum_{k=2}^{m-1}  (I - UU^{\top}) M^{k} U U^{\top} \hat{U} \hat{\Lambda}^{-(k+1)} + 
\sum_{k=1}^{m-1} (I - UU^{\top}) M^{k} (I - U U^{\top}) \hat{U} \hat{\Lambda}^{-(k+1)}
\end{split} \]
for any $m \geq 3$. For each $k \geq 1$ we have
\[ \begin{split} \|\Pi_{U}^{\perp} M^{k} E \Pi_{U}^{\perp} \hat{U} \hat{\Lambda}^{-(k+1)}\|_{2 \to \infty} &\leq \|U_{\perp} \Lambda_{\perp}\|_{2 \to \infty} \times \|\Lambda_{\perp}\|^{k-1} \times \|E\| \times \|(I - U U^{\top}) \hat{U}\| \times \|\hat{\Lambda}^{-(k+1)}\| \\
& \leq \|U_{\perp} \Lambda_{\perp}\|_{2 \to \infty}
\times |\hat{\lambda}_{r}|^{-2} \times \|E\| \times \psi_0 \times \Bigl(\frac{|\lambda_{r+1}|}{|\hat{\lambda}_r|}\Bigr)^{k - 1} \\
& \leq \|U_{\perp} \Lambda_{\perp}\|_{2 \to \infty} \times 4 |\lambda_r|^{-2} \times \|E\| \psi_0 \times \Bigl(1 - \frac{|\lambda_r| - |\lambda_{r+1}|}{2 |\lambda_r|} \Bigr)^{k-1}.
\end{split} \]
Similarly, for each $k \geq 2$ we have
\[ \begin{split} \|\Pi_{U}^{\perp} M^{k} E U U^{\top} \hat{U} \hat{\Lambda}^{-(k+1)}\|_{2 \to \infty} &\leq \|U_{\perp} \Lambda_{\perp}\|_{2 \to \infty} \times \|\Lambda_{\perp}\|^{k-2} \times \|\Lambda_{\perp} U_{\perp}^{\top} E U\| \times \|U^{\top} \hat{U}\| \times \|\hat{\Lambda}^{-(k+1)}\| \\
& \leq
\|U_{\perp} \Lambda_{\perp}\|_{2 \to \infty} \times \eta \times |\hat{\lambda}_{r}|^{-3} \times \Bigl(\frac{|\lambda_{r+1}|}{|\hat{\lambda}_r|}\Bigr)^{k - 2} \\
& \leq \|U_{\perp} \Lambda_{\perp}\|_{2 \to \infty} \times \eta \times 8 |\lambda_r|^{-3} \times \Bigl(1 - \frac{|\lambda_r| - |\lambda_{r+1}|}{2 |\lambda_r|} \Bigr)^{k-2}
\end{split} \]
Note that we avoided bounding $\|U_{\perp} |\Lambda_{\perp}|^{1/2}\|_{2 \to \infty}$ directly when $M$ is indefinite. 
Then by taking $m \rightarrow \infty$ obtain
\begin{gather*}
    \lim_{m \rightarrow \infty} \|(I - UU^{\top}) M^{m} \hat{U} \hat{\Lambda}^{-m}\| = 0, \\
    \sum_{k=1}^{\infty} \|\Pi_{U}^{\perp} M^{k} E \Pi_{U}^{\perp} \hat{U} \hat{\Lambda}^{-(k+1)}\|_{2 \to \infty} \leq 8 |\lambda_r|^{-1} \delta_r^{-1}\|U_{\perp} \Lambda_{\perp}\|_{2 \to \infty} \times \|E\| \psi_0, \\
    \sum_{k=2}^{\infty} \|\Pi_{U}^{\perp} M^{k} E U U^{\top} \hat{U} \hat{\Lambda}^{-(k+1)}\|_{2 \to \infty} \leq 16 \lambda_r^{-2} \delta_r^{-1} \|U_{\perp} \Lambda_{\perp}\|_{2 \to \infty} \times \eta
\end{gather*} 
Finally, we have
\[ \begin{split} \|(I - UU^{\top}) M E U U^{\top} \hat{U} \hat{\Lambda}^{-2}\|_{2 \to \infty} & \leq \|(I - UU^{\top}) M E U\|_{2 \to \infty} \times \|U^{\top} \hat{U}\| \times |\hat{\Lambda}^{-2}\| \\ & 
\leq \|(I - UU^{\top}) M E U\|_{2 \to \infty} \times 4 \lambda_r^{-2}
\end{split} \]
Combining the above bounds we obtain the expression 
\[ r_{2, \infty}^{(2)} = \frac{4 \|U_{\perp} \Lambda_{\perp} U_{\perp}^{\top} E U \|_{2 \to \infty}}{\lambda_r^2} + \frac{8 \|U_{\perp} \Lambda_{\perp} \|_{2 \to \infty} \times \|E\| \psi_0}{|\lambda_r| \delta_r}  +  \frac{16 \|U_{\perp} \Lambda_{\perp}\|_{2 \to \infty} \times  \eta}{\lambda_r^2 \delta_r} 
\]
when $M$ is indefinite. This concludes the proof of \cref{thm:deterministic}. 
\end{proof}

\subsection{Proof of \cref{cor:entrywise}}
Write $Z_{r} = U |\Lambda|^{1/2} \in \mathbb{R}^{n \times r}$ and $\hat{Z}_{r} = \hat{U} |\hat{\Lambda}|^{1/2} W^{(n)} \in \mathbb{R}^{n \times r}$. 
Define $\hat{P}_{r} = \hat{U} \hat{\Lambda} \hat{U}^{\top} \equiv \hat{Z}_{r} J \hat{Z}_{r}^{\top}$ and $P_{r} = U \Lambda U^{\top} \equiv Z_{r} J Z_{r}^{\top}$, where $J$ is the diagonal matrix with diagonal elements $+1$ and $-1$ such that $\Lambda = |\Lambda| J$. 
Further define $\xi_{r} = \|\hat{Z}_{r} - Z_{r}\|_{2 \to \infty}$. As
$\|M_{1} M_{2}^{\top}\|_{\max} \leq \|M_{1}\|_{2 \to \infty} \times \|M_{2}\|_{2 \to \infty}$ for conformable matrices $M_1$ and $M_2$, we have
\begin{equation}
    \label{eq:uniform_expansion1}
    \begin{split}
    \|\hat{P}_{r} - P_{r}\|_{\max}
    &=
    \|Z_{r} J (Z_{r} - \hat{Z}_{r})^{\top} + (Z_{r} - \hat{Z}_{r}) J Z_{r}^{\top} - (Z_{r} - \hat{Z}_{r}) J(Z_{r} - \hat{Z}_{r})^{\top}\|_{\max} \\
    &\leq
    2 \|Z_{r} J (Z_{r} - \hat{Z}_{r})^{\top}\|_{\max} + \|(Z_{r} - \hat{Z}_{r})J (Z_{r} - \hat{Z}_{r})^{\top}\|_{\max} \\
    &=
    2 \|Z_{r} J (E U \Lambda^{-1/2} + Q)^{\top}\|_{\max} + \xi_{r}^{2} \\
    &\leq
    2\|U JU^{\top}E \|_{\max} + 2 \|Z_{r}\|_{2 \to \infty} \cdot \|Q\|_{2 \to \infty} + \xi_{r}^{2} \\
    &\leq
    2 \|U\|_{2 \to \infty} \cdot \|E U \|_{2 \to \infty} + 2 \|U |\Lambda|^{1/2}\|_{2 \to \infty} \cdot \|Q\|_{2 \to \infty} + \xi_{r}^{2}
    .
    \end{split}
\end{equation}
Now, supposing $\kappa$ is a positive semidefinite kernel, we have
\begin{equation}
    \label{eq:uniform_expansion2psd}
    \begin{split}
    \rho_{n}^{-1} \|\hat{P}_{r} - P_{r}\|_{\max}
    &\leq
    2 \rho_{n}^{-1} \|EU\|_{2 \to \infty} \cdot \|U\|_{2 \to \infty}
    +
    \rho_{n}^{-1/2} \|Q\|_{2 \to \infty} + (\rho_{n}^{-1/2} \xi_{r})^{2}
    ,
    \end{split}
\end{equation}
since $\|U \Lambda^{1/2}\|_{2 \to \infty} \leq \rho_{n}^{1/2}$. In this setting, \cref{lem:ex_2inf} gives
\begin{equation}
    \label{eq:eu_2inf}
    \|EU\|_{2 \to \infty}
    \lesssim
    \rho_{n}^{1/2}(r^{1/2} + \log^{1/2}{n})
\end{equation} 
with high probability. Further, supposing $r$ is chosen so that the conditions in \cref{eq:condition_rem4} are satisfied (which can be achieved simply by using the eigenvalues $\{\hat{\lambda}_{j}\}_{j \ge 1}$ of $A$, similar to \cref{cor:main1}), then
\begin{equation}
    \label{eq:q_xi_bound}
    \|Q\|_{2 \to \infty}
    =
    o(\rho_{n}^{1/2} \lambda_{r}^{-1/2} (r^{1/2} + \log^{1/2}{n})) 
    \quad
    \text{and}
    \quad
    \xi_{r}
    \lesssim
    \rho_{n}^{1/2} \lambda_{r}^{-1/2} (r^{1/2} + \log^{1/2}{n})
\end{equation}
with high probability. As $\|U\|_{2 \to \infty} \leq \rho_{n}^{1/2} \lambda_{r}^{-1/2}$, combining \cref{eq:eu_2inf,eq:q_xi_bound} yields
\begin{equation}
    \begin{split}
    \label{eq:phat-p_max_part1b}
    \rho_{n}^{-1}\|\hat{P}_{r} - P_{r}\|_{\max}
    \lesssim
    \frac{(r^{1/2} + \log^{1/2}{n})}{\lambda_{r}^{1/2}}
    \end{split}
\end{equation}
with high probability. Hence, under the aforementioned assumptions, we arrive at the entrywise bound
\begin{equation}
    \label{eq:phat-p_max_part2_proof}
    \begin{split}
    \rho_{n}^{-1}\|\hat{P}_{r} - P\|_{\max}
    &\leq
    \rho_{n}^{-1}\|P_{r} - P\|_{\max}
    +
    O\left(\lambda_{r}^{-1/2} (r^{1/2} + \log^{1/2} n)\right)
    ,
\end{split}
\end{equation}
which holds with high probability.

Next, suppose instead that $\kappa$ is possibly indefinite. Then, in place of \cref{eq:uniform_expansion2psd}, we have
\begin{equation}
    \label{eq:uniform_expansion2indefinite}
    \begin{split}
    \rho_{n}^{-1} \|\hat{P}_{r} - P_{r}\|_{\max}
    &\leq
    2 \rho_{n}^{-1} \|EU\|_{2 \to \infty} \cdot \|U\|_{2 \to \infty} 
    +
    \frac{n^{1/2}}{|\lambda_{r}|^{1/2}} \|Q\|_{2 \to \infty}
    +
    (\rho_n^{-1/2} \xi_{r})^{2}
    ,
    \end{split}
\end{equation}
since $\|U |\Lambda|^{1/2}\|_{2 \to \infty} \leq n^{1/2} \rho_{n} |\lambda_{r}|^{-1/2}$ per \cref{rem:comparison_indefinite}. \cref{eq:eu_2inf} still holds under the indefinite setting while \cref{eq:q_xi_bound} holds with $|\lambda_{r}|^{-1/2}$ in place of $\lambda_{r}^{-1/2}$. Finally $\|U\|_{2 \to \infty} \leq n^{1/2} \rho_{n} |\lambda_{r}|^{-1}$, and hence, by combining \cref{eq:eu_2inf,eq:q_xi_bound}, we obtain the more widely-applicable entrywise bounds
\begin{gather}
    \label{eq:phat-p_max_part1b_indefinite}
    \rho_{n}^{-1}\|\hat{P}_{r} - P_{r}\|_{\max}
    \lesssim
    \frac{ (n \rho_{n})^{1/2} (r^{1/2} + \log^{1/2}{n})}{|\lambda_{r}|}, \\
    \label{eq:phat-p_max_part1b_indefinite2_proof}
    \rho_{n}^{-1}\|\hat{P}_{r} - P\|_{\max} \leq \rho_n^{-1} \|P_r - P\|_{\max}
    +
    O\left(|\lambda_r|^{-1} (n \rho_n)^{1/2} (r^{1/2} + \log^{1/2}{n}\right),
\end{gather}
which holds with high probability. This concludes the proof of \cref{cor:entrywise}.

\subsection{Proof of \cref{prop:mercer_entrywise}}
\label{sec:proof_mercer_entrywise}
Let $\mathcal{H}$ denote the reproducing kernel Hilbert space (RKHS) associated with $\kappa$. Given $\{X_{i}\}_{i=1}^{n}$, let $\mathcal{K}_{\mathcal{H},n}$ denote the linear operator
\begin{equation*}
    \mathcal{K}_{\mathcal{H},n} \eta
    =
    \sum_{i=1}^{n} \langle \eta, \Phi(X_{i}) \rangle_{\mathcal{H}} \Phi(X_{i}),
    \qquad
    \text{for all }
    \eta \in \mathcal{H}
    ,
\end{equation*}
where $\Phi(x) = \kappa(\cdot, x) \in \mathcal{H}$ denotes the reproducing element at $x \in \Omega$. The (non-zero) eigenvalues of $\mathcal{K}_{\mathcal{H},n}$ are the same as those of $\rho_{n}^{-1} P$, while the eigenfunctions of $\mathcal{K}_{\mathcal{H},n}$ are {\em extensions} of the corresponding eigenvectors of $\rho_{n}^{-1} P$. See \cite{rosasco_1} for more details.

Let $\hat{\Pi}_{r}$ denote the projection operator onto the $r$-dimensional subspaces spanned by the $r$ leading eigenvectors of $\mathcal{K}_{\mathcal{H},n}$, and let $p^{(r)}_{ij}$ denote the $ij$-th entry of $P_{r}$. We then have 
\begin{equation*}
    p^{(r)}_{ij}
    = \rho_n \langle \hat{\Pi}_{r} \Phi(X_{i}), \hat{\Pi}_{r} \Phi(X_{j}) \rangle_{\mathcal{H}}.
\end{equation*}
See the proof of \cite[Lemma~3.4]{tang2012universally} for a derivation of the above identity. Consequently,
\begin{equation}
    \label{eq:uniform_part0}
    \begin{split}
    \rho_n^{-1}|p^{(r)}_{ij} - p_{ij}|
    &=
    \left|\langle \hat{\Pi}_r \Phi(X_i), \hat{\Pi}_r \Phi(X_j) \rangle_{\mathcal{H}} - \langle \Phi(X_i), \Phi(X_j) \rangle_{\mathcal{H}}\right| \\
    &\leq
    \|(\hat{\Pi}_r - \Pi_r) \Phi(X_i)\|_{\mathcal{H}} \times \|\Pi_r \Phi(X_j)\|_{\mathcal{H}} \\
    &\qquad+
    \|(\hat{\Pi}_r - \Pi_r) \Phi(X_j)\|_{\mathcal{H}} \times \|\Pi_r \Phi(X_i)\|_{\mathcal{H}} \\
    &\qquad+
    \|(\hat{\Pi}_r - \Pi_r) \Phi(X_i)\|_{\mathcal{H}} \times \|(\hat{\Pi}_r - \Pi_r) \Phi(X_j)\|_{\mathcal{H}} \\
    &\qquad+
    \left|\langle \Pi_r \Phi(X_i), \Pi_r \Phi(X_j) \rangle_{\mathcal{H}} - \langle \Phi(X_i), \Phi(X_j) \rangle_{\mathcal{H}} \right|
    ,
\end{split}
\end{equation}
where $\Pi_r$ is the projection onto the subspaces spanned by the $r$ leading {\em eigenfunctions} of $\mathcal{K}$. We now have
\begin{equation}
    \label{eq:uniform_part1}
    \begin{split}
    \|(\hat{\Pi}_r - \Pi_r)\Phi(X_i)\|_{\mathcal{H}}
    &\leq
    \|\hat{\Pi}_r - \Pi_r\|_{\mathcal{H}}
    \times
    \|\Phi(X_i)\|_{\mathcal{H}}
    \leq
    \kappa(X_i, X_i)
    \times
    \|\hat{\Pi}_r - \Pi_r\|_{\mathcal{H}}
    .
\end{split}
\end{equation}
Perturbation bounds for projection operators (e.g., Proposition~6 and Theorem~7 in \cite{rosasco_1}) yield
\begin{equation}
    \label{eq:uniform_part2}
    \|\hat{\Pi}_r - \Pi_r\|_{\mathcal{H}}
    \lesssim 
    \frac{\log^{1/2}{n}}{n^{1/2}(\mu_r - \mu_{r+1})}
\end{equation}
with high probability. Next, by Mercer's theorem \cite[Corollary~4.50]{steinwart08:_suppor_vector_machin}, $\{\sqrt{\mu_r} \phi_r\}_{r \ge 1}$ forms an orthonormal basis for $\mathcal{H}$, and we can write
\begin{equation*}
    \Pi_r \Phi(X_i)
    =
    \sum_{k=1}^{r} \langle \sqrt{\mu_k} \phi_k(\cdot), \Phi(X_i) \rangle_{\mathcal{H}} \sqrt{\mu_k} \phi_k(\cdot)
    =
    \sum_{k=1}^{r} \mu_k \phi_k(X_i) \phi_k(\cdot)
    ,
\end{equation*}
and $\Phi(X_i) = \sum_{k=1}^{\infty} \mu_k \phi_k(X_i) \phi_k(\cdot)$. Here, the final equality in the expression for $\Pi_r \Phi(X_i)$ follows from the reproducing kernel property of $\Phi$. Finally, the inner product $\langle \cdot, \cdot \rangle_{\mathcal{H}}$ can be manipulated using \cite[Theorem~4.51]{steinwart08:_suppor_vector_machin} to yield
\begin{gather}
    \label{eq:uniform_part2b}
    \langle \Pi_r \Phi(X_i), \Pi_r \Phi(X_j) \rangle_{\mathcal{H}} = \sum_{k=1}^{r} \mu_k \phi_k(X_i) \phi_k(X_j), \\
    \label{eq:uniform_part2c}
    \langle \Phi(X_i), \Phi(X_j) \rangle_{\mathcal{H}} = \sum_{k=1}^{\infty} \mu_k \phi_k(X_i) \phi_k(X_j).
\end{gather}

Plugging \cref{eq:uniform_part2c,eq:uniform_part2b,eq:uniform_part2,eq:uniform_part1} into \cref{eq:uniform_part0}, we obtain
\begin{equation}
    \label{eq:Pr-P_definite1_proof}
    \begin{split}
    \max_{ij} \rho_n^{-1} |p_{ij}^{(r)} - p_{ij}|
    &\lesssim 
    \frac{\log^{1/2}{n}}{n^{1/2}(\mu_r - \mu_{r+1})}
    +
    \max_{ij} \left|\sum_{k > r} \mu_k \phi_k(X_i) \phi_k(X_j)\right|
\end{split}
\end{equation}
with high probability. This concludes the proof of Proposition~\ref{prop:mercer_entrywise}. 
\subsection{Proof of \cref{thm:quadratic_test}}
\label{sec:quad_test_proof}
Let $\hat{\xi}_i$ and $\xi_i$ denote the $i$-th row of $\hat{U} \hat{\Lambda}$ and $U \Lambda^{1/2}$, respectively. Suppose that the null hypothesis $\mathbb{H}_0 \colon X_i = X_j$ is true. We first note that a direct application of \cref{thm:generalU_ULambda} is sub-optimal as then, to make $\|Q\|_{2 \to \infty} = o_{p}(\rho_n^{1/2})$, we need $\delta_r^2 = \omega((n \rho_n)^{3/2}$ which leads to a more stringent condition than $\delta_r = \omega(\sqrt{n \rho_n \log n})$ given in \cref{thm:quadratic_test}. Rather, we need to exploit the fact that we can drop the term $n^{3/2} \rho_n^2 \delta_r^{-2}$ from the bound for $Q$ under $\mathbb{H}_0$. Indeed, from the proof of \cref{thm:generalU_ULambda}, we have the decomposition
\begin{equation}
    \label{eq:xii-xij_thm_quad_test}
    \begin{split}
    \hat{U} \hat{\Lambda} - U \Lambda W^{(n)}
    &=
    U \Lambda (U^{\top} \hat{U} - W^{(n)}) +  U (U^{\top} \hat{U} \hat{\Lambda} - \Lambda U^{\top} \hat{U}) + U U^{\top} E \hat{U} + (I - UU)^{\top} P \hat{U} \\
    &\qquad+
    E U [W^{(n)} + (U^{\top} \hat{U} - W^{(n)}] + E(I - UU^{\top}) \hat{U}
    .
    \end{split} 
\end{equation}
Then, under $\mathbb{H}_0$, the $i$-th and $j$-th row of all matrix products in \cref{eq:xii-xij_thm_quad_test} whose first term is either $U$ or $(I - UU^{\top})$ are the same. Hence 
\[
    \hat{\xi}_i - \hat{\xi}_j
    =
    (E_i - E_j) U [W^{(n)} + (U^{\top} \hat{U} - W^{(n)}] + (E_i - E_j) (I - UU^{\top}) \hat{U}
    ,
\]
where $E_i$ and $E_j$ are the $i$-th and $j$-th row of $E$, respectively. Finally, by following the remaining steps in the proof of \cref{thm:generalU_ULambda}, we obtain
\begin{equation}
    \label{eq:xii-xij_thm_quad_test2}
    \hat{\xi}_i - \hat{\xi}_j
    =
    (E_i - E_j) U W^{(n)} + \gamma_{ij},
    \quad
    \|\gamma_{ij}\|
    =
    O\Bigl[(\rho_n \log n)^{1/2}\Bigl(\frac{\log n}{\delta_r} + \frac{(r n \rho_n)^{1/2}}{\delta_r}\Bigr)\Bigr]
\end{equation}
with high probability.
Now, choose $r$ such that the conditions
\begin{gather}
    \label{eq:select_r_quad1}
    \delta_r
    =
    \omega\left( \max\left\{\log^{3/2}{n}, \sqrt{r n \rho_n \log n} \right\}\right)
\end{gather}
is satisfied. Then, \cref{eq:xii-xij_thm_quad_test2} implies
\[
    \hat{\xi}_i - \hat{\xi}_j
    =
    (E_i - E_j) U + \gamma_{ij},
    \quad
    \|\gamma_{ij}\|
    =
    o(\rho_n^{1/2})
\]
with high probability.  
We therefore have
\[
    \Bigl|\|\hat{\xi}_i - \hat{\xi}_j\|^2 - \|(E_i - E_j)U\|^2\Bigr|
    =
    o(\rho_n) + o(\rho_n^{1/2}) \times \|(E_i - E_j) U\|
\]
with high probability.

The remainder of this proof is devoted to analyzing $\tilde{T} \coloneqq \|(E_i - E_j)U\|^2$. In particular, we will show that $\tilde{T}$, properly centered and scaled by a term of order $\rho_n r^{1/2}$, converges to a weighted sum of independent $\chi^2_1$ random variables. As a by-product, we have $\|(E_i - E_j)U\| = O_{p}(\rho_n^{1/2} r^{1/2})$, so that
\[
    \Bigl|\|\hat{\xi}_i - \hat{\xi}_j\|^2 - \|(E_i - E_j)U\|^2\Bigr|
    =
    o_{p}(\rho_n r^{1/2})
    .
\]
Slutsky's theorem implies the limiting distribution of $T \coloneqq \|\hat{\xi}_i - \hat{\xi}_j\|^2$ is the same as that of $\tilde{T}$, after centering and scaling by a term of order $\rho_n r^{1/2}$.

We now derive the limiting distribution for $\tilde{T}$. 
Recall \cref{eq:quadratic_test} and write
\[
    \tilde{T}
    =
    \|(E_i - E_j)U\|^2 = \zeta^{\top} \mathcal{D} U U^{\top} \mathcal{D} \zeta
    ,
\]
where $\zeta = (\zeta_1, \zeta_2, \dots, \zeta_n)$ is a vector whose components are independent sub-Gaussian random variables with mean zero and variance one. 

Next, we apply the following comparison result for quadratic forms.
\begin{theorem}[Quadratic forms comparison \cite{rotar_1}]
\label{thm:clt_quadratic}
Let $X_1, X_2, \dots,$ be a sequence of independent random variables with 
\begin{equation}
    \label{eq:mean_0_quadratic_form}
    \mathbb{E}[X_k] = 0,
    \quad
    \text{and}
    \quad
    \mathrm{Var}[X_k] = 1
    \quad
    \text{for all $k$.}
\end{equation}
Let $W(n) = \sum_{1 \leq k < \ell \leq n} X_k X_{\ell}
m_{k \ell}$ be a quadratic form, where the coefficients $(m_{k \ell})$ can vary with $n$. Suppose $M = [m_{k\ell}]$ is symmetric, with zeroes on the diagonal, and 
\[
    \sum_{k=1}^{n} \sum_{\ell=1}^{n} m_{k \ell}^2
    =
    1
    .
\]
Let $F = (F_1, F_2, \dots)$ and $G = (G_1, G_2, \dots, )$ be two arbitrary collections of cumulative distribution functions for random variables $(X_1, X_2, \dots)$ satisfying \cref{eq:mean_0_quadratic_form}. Denote by $Q(W(n), F)$ and $Q(W(n), G)$ the distribution of $W(n)$ when the  $(X_1, X_2, \dots)$ has CDFs $F$ and $G$, respectively. Let $\Psi_{k} = F_k - G_k$. Define
\[
    s_{k}^2(n)
    =
    \sum_{\ell=1}^{n} m_{k \ell}^2
    .
\]
Now, suppose that, as $n \rightarrow \infty$, the following conditions are satisfied
\begin{enumerate}
    \item For any $A \rightarrow \infty$,
    \begin{equation}
        \label{eq:technical_quadratic_1}
            \sup_{1 \leq k \leq n}  \int_{|x| \geq A} x^2 dG_k(x)
            \rightarrow
            0
            .
    \end{equation}
    \item For any fixed but arbitrary $\epsilon > 0$,
    \begin{gather}
        \label{eq:technical_quadratic_2}
        \sum_{k=1}^{n} \sum_{\ell=k+1}^{n} m_{k \ell}^2 \iint_{|xy| \geq |\epsilon/m_{k \ell}|} |xy \Psi_k(x) \Psi_{\ell}(y)| \, \mathrm{d} x \, \mathrm{d} y
        \rightarrow
        0
        , \\
        \label{eq:technical_quadratic_3}
        \sum_{k=1}^{n} s_k^2(n) \int_{|x| \geq \epsilon/s_k(n)} |x \Psi_k(x)| \, \mathrm{d} x 
        \rightarrow
        0
        .
        \end{gather}
    \end{enumerate}
Then, $Q(W(n), F) \rightarrow Q(W(n), G)$ in {\em distribution} as $n \rightarrow \infty$.
\end{theorem}
Recall that we denote $d_k^2 = 2p_{ik}(1 - p_{ik})$. We now verify the conditions in
\cref{thm:clt_quadratic} for the case when the $F_k$'s are the CDFs for discrete random variables with probability mass function
\begin{gather*}
    \mathbb{P}(X_k = 0)
    =
    1 - d_{k}^2 = p_{ik}^2 + (1 - p_{ik})^2, \\
    \mathbb{P}(X_k = -1/d_{k})
    =
    \mathbb{P}(X_k = 1/d_{k}) = p_{ik}(1 - p_{ik}) = \frac{1}{2} d_{k}^2
\end{gather*}
and $G_k \equiv \Phi$ for all $k$ where $\Phi$ is the CDF of the standard normal distribution.

\cref{eq:technical_quadratic_1} is straightforward to verify. Indeed, $x^2 e^{-x^2/2} \precsim x^{-2}$ as $x$ increases and hence 
\[
    \lim_{A \rightarrow \infty} \sup_{1 \leq k \leq n} \int_{|x| \geq A} x^2 dG_k(x)
    \precsim
    \lim_{A \rightarrow \infty} \sup_{1 \leq k \leq n} \int_{|x| \geq A} x^{-2} \, \mathrm{d} x
    \precsim
    \lim_{A \rightarrow \infty} \sup_{1 \leq k \leq n} A^{-1}
    =
    0
    .
\]
We next verify the conditions in \cref{eq:technical_quadratic_2,eq:technical_quadratic_3}. Define
\[
    M^*
    =
    \mathcal{D} U U^{\top} \mathcal{D},
    \quad
    M^{(0)}
    =
    M^* - \mathrm{diag}(M^*),
    \quad M
    =
    \frac{1}{\|M^{(0)}\|_{F}} M^{(0)}
    .
\]
Here, $M^{(0)}$ is obtained by setting the diagonal entries of $M^*$ to zero while $\|M\|_{F} = 1$.  
Let $m^*_{k \ell}$ denote the $k\ell$-th entry of $M^*$. From the above criteria for $r$, we have
\begin{gather}
    \label{eq:proof_thm_test1}
    \|M^*\|_{F}^2
    =
    \mathrm{tr} \bigl(U^{\top} \mathcal{D}^2 U\bigr)^2 \asymp \rho_n^2 r, \\
    \label{eq:proof_thm_test2}
    \max_{k} |m_{kk}^*|
    \leq
    \bigl(\|\mathcal{D}\|^2 \times \|U\|_{2 \to \infty}^2\bigr) \leq \frac{4 \rho_n^2}{|\lambda_r|}, \\
    \label{eq:proof_thm_test3}
    \|M^{(0)}\|_{F}^2
    =
    \|M^*\|_{F}^2 - \sum_{k} (m_{kk}^*)^2
    \leq
    \|M^*\|_{F}^2 - \frac{16 n \rho_n^4}{\lambda_r^2} 
    =
    (1 - o(1)) \|M^*\|_{F}^2
    .
\end{gather}
The quadratic form $\zeta^{\top} \mathcal{D} UU^{\top} \mathcal{D} \zeta$ in \cref{eq:quadratic_test} can be written as
\begin{equation}
    \label{eq:quadratic_term_part2}
    \zeta^{\top} M^* \zeta
    =
    \zeta^{\top} M^{(0)} \zeta + \sum_{k} \zeta_k^2 m_{kk}^* = \|M^{(0)}\|_{F} \times \zeta^{\top} M \zeta + \sum_{k} \zeta_k^2 m_{kk}^*
    .
\end{equation}
For the condition in \cref{eq:technical_quadratic_3}, let $m_{kk}^{(2)}$ and $m_{kk}^{(2*)}$ denote the $k$-th diagonal element of $M^2$ and $(M^*)^2$.
Then,
\[
    \max_{k} m_{kk}^{(2*)}
    \leq
    \|(M^*)^2\|_{\max} = \|\mathcal{D} U U^{\top} \mathcal{D}^2 U U^{\top} \mathcal{D}\|_{\max}
    \leq
    \|\mathcal{D} U\|_{2 \to \infty} \times \|U^{\top} \mathcal{D}^2 U\| \times \|\mathcal{D} U\|_{2 \to \infty}
    ,
\]
where the first inequality follows from $\|MN^{\top}\|_{\max} \leq \|M\|_{2 \to \infty} \times \|N\|_{2 \to \infty}$ together with \cref{eq:2toinf_submultiplicative}. As $\mathcal{D}$ is a diagonal matrix, we have $\|\mathcal{D} U\|_{2 \to \infty} \leq \|\mathcal{D}\| \times \|U\|_{2 \to \infty}$. Hence,
\begin{equation*}
    \begin{split}
    s_{k}^2(n)
    =
    \sum_{\ell=1}^{n} m_{k \ell}^2
    \leq
    m_{kk}^{(2)}
    &\leq
    \frac{1}{\|M_0\|_F^2} \times m_{kk}^{(2*)} \\ 
    &\leq
    \frac{1}{\|M_0\|_{F}^2} \times \|\mathcal{D}\|^2 \times \|U\|_{2 \to \infty}^2 \times \|U^{\top} \mathcal{D}^2 U\| \\
    &\lesssim
    (\rho_n^2 r)^{-1} \times \rho_n \times n \rho_n^2 \times \lambda_r^{-2} \times \rho_n \lesssim r^{-1} \lambda_r^{-2} n \rho_n^2
    .
\end{split}
\end{equation*}
Now, choose an arbitrary but fixed $\epsilon > 0$. Then,
\[
    \epsilon/s_{k}(n)
    \gtrsim
    r^{1/2} |\lambda_r| (n^{-1/2} \rho_n^{-1}) \epsilon
    =
    \omega(\rho_n^{-1/2}) = \omega(1/d_{kk})
    ,
\]
as our criterion in \cref{eq:select_r_quad1} implies $\lambda_r = \omega((n \rho_n \log n)^{1/2})$ as $n \rightarrow \infty$. We thus have $F_k(-\epsilon/s_{k}(n)) = 0$ and $F_k(\epsilon/s_k(n)) = 1$ for sufficiently large $n$, hence
\begin{equation*}
    \begin{split}
    \int_{|x| \geq \epsilon/s_{k}(n)} |x \Psi_k(x)| \, \mathrm{d} x
    &=
    \int_{x \leq - \epsilon/s_{k}(n)} |x \Phi(x)| \, \mathrm{d} x + \int_{x \geq \epsilon/s_{k}(n)} |x (1 - \Phi(x))| \, \mathrm{d} x \\
    &=
    2 \int_{x \geq \epsilon/s_{k}(n)} x(1 - \Phi(x)) \, \mathrm{d}x
    .
  \end{split}
\end{equation*}
Next, recall the Mill's ratio upper bound
\[
    \frac{1 - \Phi(x)}{\phi(x)}
    \leq
    \frac{1}{x}
    ,
\]
where $\phi(x)$ is the probability density function of a standard normal random variable. We then have
\begin{equation}
    \label{eq:derive_technical1}
    \int_{x \geq \epsilon/s_{k}(n)} x(1 - \Phi(x)) \, \mathrm{d}x
    \leq
    \int_{x \geq \epsilon/s_{k}(n)} \phi(x) \, \mathrm{d}x
    =
    1 - \Phi\Bigl(\frac{\epsilon}{s_{k}(n)}\Bigr) 
    \leq \frac{\phi(\epsilon/s_{k}(n))}{\epsilon/s_k(n)}
    ,
\end{equation}
and hence
\begin{equation}
    \begin{split}
    \sum_{k=1}^{n} s_{k}^2(n) \int_{|x| \geq \epsilon/s_{k}(n)} |x \Psi_k(x)| \, \mathrm{d} x &\leq 2 \sum_{k=1}^{n} \epsilon^{-1} s_{k}^3(n) \phi(\epsilon/s_{k}(n)) \\
    &\leq
    2 n \epsilon^{-1} \bigl(\max_{k} s_{k}^3(n)  \phi(\epsilon/s_{k}(n))\bigr) \\
    &\precsim
    \epsilon^{-1} \bigl(\max_{k} s_{k}^3(n)\bigr) \times \bigl(\max_{k} e^{\log(n) -\epsilon^2/(2s_k^2(n))}\bigr)
    .
    \end{split}
\end{equation}
As $\epsilon^2/s_{k}^2(n) \gtrsim r \lambda_r^{2} (n \rho_n^{-2}) \epsilon^2 $ and $\epsilon > 0$ is fixed but arbitrary, we have
\[
    \log(n) - \epsilon^2/(2s_k^{2}(n))
    \leq
    0
\]
as $n$ increases (see \cref{eq:select_r_quad1}. We thus have
\[
    \sum_{k=1}^{n} s_{k}^2(n) \int_{|x| \geq \epsilon/s_{k}(n)} |x \Psi(x)| \, \mathrm{d} x 
    \precsim
    \epsilon^{-1} \times \max_{k} s_{k}^3(n)
    \precsim
    \epsilon^{-1} (r^{-1} \lambda_r^{-1} n^{1/2} \rho_n)^{3/2}
    \rightarrow
    0
\]
as $n \rightarrow \infty$. \cref{eq:technical_quadratic_3} is therefore satisfied. 

For \cref{eq:technical_quadratic_2}, we have
\[
    \begin{split} \max_{k \ell} m_{k \ell}
    &\leq
    \frac{1}{\|M^{(0)}\|_{F}} \times \|M^*\|_{\max} \\ 
    &\leq
    \frac{1}{\|M^{(0)}\|_{F}} \times \|\mathcal{D}\|^2 \times \|U\|_{2 \to \infty}^2 \lesssim \frac{1}{\rho_n r^{1/2}} \times \rho_n \times \frac{n \rho_n^2}{\lambda_r^2} \lesssim \frac{n \rho_n^2}{r^{1/2}\lambda_r^2}
    .
    \end{split}
\]
Now let $c_{k \ell} = (|\epsilon/m_{k \ell}|)^{1/2}$, and write $h(x,y) = |xy\Psi_{k}(x) \Psi_{\ell}(y)|$. Define the regions
\begin{gather*}
    S_1
    =
    \{(x,y) \colon |x| \geq c_{k \ell}, |y| \leq c_{k \ell}\},
    \quad
    S_2
    =
    \{(x,y) \colon |x| \leq c_{k \ell}, |y| \geq c_{k \ell}\},
    \quad
    S_3
    =
    \{|x| \geq c_{k \ell}, |y| \geq c_{k \ell}\}
    ,
\end{gather*}
and note that $S \coloneqq \{(x,y) \colon |xy| \geq |\epsilon/m_{k \ell}|\} \subset S_1 \cup S_2 \cup S_3$. We then have
\begin{equation*}
\begin{split}
    \iint_{S} h(x,y) \mathrm{d} x \, \mathrm{d} y
    &\leq
    \iint_{S_1} h(x,y) \, \mathrm{d} y \, \mathrm{d} x + \int_{S_2} h(x,y) \mathrm{d} y \, \mathrm{d} x + \iint_{S_3} h(x,y) \mathrm{d} y \, \mathrm{d} x \\
    &\leq
    \int_{|x| \leq c_{k \ell}} |x \Psi_k(x)| \, \mathrm{d} x \int_{|y| \geq c_{k \ell}} |y \Psi_{\ell}(y)| \, \mathrm{d} y + \int_{|x| \geq c_{k \ell}} |x \Psi_k(x)| \, \mathrm{d} x \int_{|y| \leq c_{k \ell}} |y \Psi_{\ell}(y)| \, \mathrm{d} y \\
    &+
    \int_{|x| \geq c_{k \ell}} |x \Psi_k(x)| \, \mathrm{d}x \int_{|y| \geq c_{k \ell}} |y \Psi_{\ell}(y)| \, \mathrm{d}y
    .
\end{split}
\end{equation*}
Now, $c_{k \ell} \succsim \epsilon^{1/2} r^{1/4} |\lambda_r|/(n^{1/2} \rho_n) = \omega(\rho_n^{-1/2})$ {\em uniformly} for all $\{k, \ell\}$ provided that \cref{eq:select_r_quad1} is satisfied. Thus, by a similar argument to that for deriving \cref{eq:derive_technical1}, we have
\[
    \int_{|x| \geq c_{k \ell}} |x \Psi_k(x)| \mathrm{d} x \leq \frac{2}{c_{k \ell}} \phi(c_{k \ell})
\]
for all $\{k, \ell\}$, and similarly for $\int_{|y| \geq c_{k \ell}} |y \Psi_{\ell}(y)| \mathrm{d} y$. 
We therefore have
\begin{equation*}
    \begin{split}
    \iint_{S} h(x,y) \, \mathrm{d} x \, \mathrm{d} y 
    &\leq
    \frac{2}{c_{k \ell}} \phi(c_{k \ell}) \Bigl(\int_{|x| \leq c_{k \ell}} |x \Psi_k(x)| \, \mathrm{d} x + \int_{|y| \leq c_{k \ell}} |y \Psi_{\ell}(y)| \, \mathrm{d} y\Bigr) + \frac{4}{c_{k \ell}^2}\phi^2(c_{k \ell}) \\
    &\leq
    2 c_{k \ell} \phi(c_{k \ell}) + \frac{4}{c_{k \ell}^2}\phi^2(c_{k \ell})
    .
\end{split}
\end{equation*}
Finally, as $m_{k \ell}^2 \lesssim n^2 \rho_n^4/(r \lambda_r^4)$ for all $\{k, \ell\}$, we have
\begin{equation*}
    \begin{split}
    \sum_{k=1}^{n} \sum_{j = k+1}^{n} m_{k \ell}^2 \int_{|xy| \geq |\epsilon/m_{k \ell}} h(x,y) \, \mathrm{d} x \, \mathrm{d} y
    &\precsim
    \frac{n^4 \rho_n^4}{r \lambda_r^4} \max_{k \ell} \Bigl(c_{k \ell} \phi(c_{k \ell}) + \frac{1}{c_{k \ell}^2}\phi^2(c_{k \ell})\Bigr) 
    \rightarrow
    0
\end{split}
\end{equation*}
as $n \rightarrow \infty$. The condition in \cref{eq:technical_quadratic_2} is therefore satisfied. 

Now, $\zeta_k^2 = d_{k}^{-2} Y_k$, where the $\{Y_k\}$ are independent Bernoulli random variables with success probabilities $\{d_{k}^2\}$. We thus have
\[
    \mathbb{E}[\zeta_k^2]
    =
    1,
    \quad
    \mathrm{Var}[\zeta_{kk}^2]
    =
    \frac{d_{k}^2(1 - d_{k}^2)}{d_{k}^4} = \frac{1 - d_{k}^2}{d_{k}^2}
    .
\]
Therefore, by Bernstein's inequality, we have
\[
    \Bigl|\sum_{k=1}^{n} (\zeta_{k}^2 - 1) m_{kk}^* \Bigr|
    \precsim
    \Bigl( \sum_{k=1}^{n} \frac{1- d_{k}^2}{d_{k}^2} (m_{kk}^*)^2\Bigr)^{1/2} \times \log^{1/2}{n}
    \precsim
    \frac{n^{1/2} \rho_n^{3/2} \log^{1/2}{n}}{|\lambda_r|}
\]
with high probability, and hence 
\[
    \frac{1}{\|M_0\|_{F}} \Bigl|\sum_{k=1}^{n} (\zeta_{k}^2 - 1) m_{kk}^* \Bigr|
    \precsim
    \frac{(n \rho_n \log n)^{1/2}}{r^{1/2} |\lambda_r|}
    \rightarrow
    0
\]
almost surely. We therefore have
\[
    \frac{\zeta^{\top} M^* \zeta - \sum_{k} m_{kk}^*}{\|M_0\|_{F}}
    =
    \zeta^{\top} M \zeta + \frac{1}{\|M_0\|_{F}} \sum_{k=1}^{n} (\zeta_k^2 - 1) m_{kk}^*
    =
    \zeta^{\top} M \zeta + o_P(1)
    ,
\]
and hence, by applying \cref{thm:clt_quadratic}, we have
\[
    \frac{\zeta^{\top} M^* \zeta - \sum_{k} m_{kk}^*}{\|M_0\|_{F}}
    =
    Z^{\top} M Z + o_{P}(1)
    =
    \sum_{s \geq 1} Z_s^2 \lambda_s(M) + o_{P}(1)
    ,
\]
where $Z = (Z_1, Z_2, \dots)$ is a vector of independent $N(0,1)$ random variables and $\lambda_s(M)$ are the eigenvalues of the matrix $M$ ordered in decreasing magnitudes. As the diagonal entries of $M$ are all zeroes and $\|M\|_{F}^2 = 1$, we have
\begin{gather*}
    \mathbb{E}\left[\sum_{s \geq 1} Z_s^2 \lambda_s(M)\right]
    =
    \sum_{s \geq 1} \mathbb{E}[Z_s^2] \lambda_s(M) = \sum_{s \geq 1} \lambda_s(M) = 0, \\ \mathrm{Var}\left[\sum_{s \geq 1} Z_s^2 \lambda_s(M)\right]
    =
    \sum_{s \geq 1} \mathrm{Var}[Z_s^2] \times \lambda_s^2(M) = 2 \|M\|_{F}^2 = 2
    .
\end{gather*}
Next, by Weyl's inequality, we have
\begin{equation*}
    \begin{split}
    \max_{s} \Bigl|\lambda_s(M) - \frac{1}{\|M_0\|_{F}} \lambda_s(M^*)\Bigr|
    &\leq
    \frac{1}{\|M_0\|_{F}}\max_{k} m_{kk}^* \\
    &\lesssim
    \frac{1}{\rho_n r^{1/2}} \times \frac{\rho_n^2}{|\lambda_r|} \lesssim \frac{\rho_n}{r^{1/2} |\lambda_r|}
    .
  \end{split}
\end{equation*}
Note that $M^* = \mathcal{D} U U^{\top} \mathcal{D}$ is of rank $r$ and the non-zero eigenvalues of $M^*$ are the same as those of $U^{\top} \mathcal{D}^2 U$ and hence of magnitude $\rho_n$. We thus have
\[
    |\lambda_s(M)|
    \asymp
    r^{-1/2}
    \quad
    \text{for $1 \leq s \leq r$},
    \quad
    |\lambda_s(M)|
    \precsim
    \frac{\rho_n}{r^{1/2} |\lambda_r|},
    \qquad
    \text{for $s \geq r+1$}
    .
\]
Next, define $Y^{(r)} = \tfrac{1}{\|M_0\|_F}\sum_{s=1}^r (Z_s^2 - 1) \lambda_s(M^*)$. Then, $\mathbb{E}[Y^{(r)}] = 0 = \mathbb{E}[\sum_{s \geq 1} Z_s^2 \lambda_s(M)]$. Furthermore, we also have
\begin{equation}
    \begin{split}
    \mathbb{E}\Bigl[Y^{(r)} - \sum_{s\geq 1} Z_s^2 \lambda_s(M)\Bigr]^2
    &=
    \mathbb{E}\Bigl[\sum_{s=1}^{r} Z_s^2 \Bigl(\lambda_s(M) - \frac{\lambda_s(M^*)}{\|M_0\|_{F}}\Bigr) + \sum_{s \geq r+1} Z_s^2 \lambda_s(M)\Bigr]^2 \\
    &=
    2\Bigl(\sum_{s=1}^{r} \Bigl(\lambda_s(M) - \frac{\lambda_s(M^*)}{\|M_0\|_{F}}\Bigr)^2 + \sum_{s \geq r+1}  \lambda_s^2(M)\Bigr) \\
    &\lesssim
    \frac{n \rho_n^2}{r \lambda_r^2}
    .
\end{split}
\end{equation}
Therefore, as $n \rightarrow \infty$, we have
\[
    Y^{(r)} - \sum_{s \geq 1} Z_s^2 \lambda_s(M) 
    \rightarrow
    0
\]
in probability. 
In summary, we have under the null hypothesis of $X_i = X_j$ that
\begin{equation}
    \label{eq:two_sample_test_distribution}
    \frac{\|(A \hat{U})_{i} - (A \hat{U})_{j}\|^2 - \mathrm{tr} \, M^{*}}{\|M_0\|_{F}} 
    \rightsquigarrow
    \frac{1}{\|M_0\|_F} \sum_{s=1}^{r} (Z_s^2 - 1) \lambda_s(M^{*})
    .
\end{equation}
Recalling \cref{eq:proof_thm_test3}, we have by Slutsky's theorem that
\begin{equation}
    \label{eq:two_sample_test_distribution_Slutsky}
    \frac{\|(A \hat{U})_{i} - (A \hat{U})_{j}\|^2 - \mathrm{tr} \, M^{*}}{\|M^*\|_{F}} 
    \rightsquigarrow
    \frac{1}{\|M^*\|_F} \sum_{s=1}^{r} (Z_s^2 - 1) \lambda_s(M^{*})
    .
\end{equation}
Finally, if $\kappa$ has infinite rank, then we can choose $r \rightarrow \infty$ as $n \rightarrow \infty$ such that, by the Lindeberg--Feller central limit theorem (see Theorem~27.2 and Problem 27.6 in \cite{billingsley}), we have
\[
    \frac{1}{\|M^*\|_F} \sum_{s=1}^{r} (Z_s^2 - 1) \lambda_s(M^*)
    \rightsquigarrow
    N(0,2)
    .
\]
This concludes the proof of \cref{thm:quadratic_test}

\subsection{Proof of \cref{lem:estimate}}
First, consider the expansion
\[
    \hat{\theta} - \mathrm{tr} M^*
    =
    \mathrm{tr} \, (\hat{U}^{\top} \hat{\mathcal{D}}^2 \hat{U} - U^{\top} \hat{\mathcal{D}}^2 U)
    +
    \mathrm{tr} \, (U^{\top} (\hat{\mathcal{D}}^2 - \mathcal{D}^2) U)
    .
\]
Recall that $\hat{\mathcal{D}}$ is a diagonal matrix with $\hat{d}_{k} \in \{0,1\}$ for all $k$. Now, let $S = \{k \colon \hat{d}_{k} = 1\}$ and $\zeta_{2,\infty} = \|\hat{U} W - U\|_{2 \to \infty}$ where $W = W^{(n)}$ is the orthogonal matrix appearing in \cref{thm:generalU_ULambda}. If we select $r$ according to the criterion in \cref{eq:select_r_lemma_test}
then
\[ \zeta_{2, \infty} \lesssim \frac{\rho_n^{1/2} (r^{1/2} + \log^{1/2}{n})}{|\lambda_r|}\]
with high probability (see \cref{eq:EUlambda_indefiniteU})
 Also, $\|U\|_{2 \to \infty} \leq |\lambda_r|^{-1} n^{1/2} \rho_n$ (see \cref{eq:epsilon_bounds}). 
 
 We therefore have
\begin{equation}
    \label{eq:bound_estimate1}
    \begin{split}
    \Bigl|\mathrm{tr} \, (\hat{U}^{\top} \hat{\mathcal{D}}^2 \hat{U} - U^{\top} \hat{\mathcal{D}}^2 U)\Bigr|
    &=
    \Bigl|\sum_{k \in S} (\|\hat{U}_k\|^2 - \|U_k\|^2)\Bigr| \\
    &\leq
    \bigl(2 \zeta_{2,\infty}\|U\|_{2 \to \infty} + \zeta_{2,\infty}^2\bigr) \times |S| \\
    &\lesssim
    \frac{n^{1/2} \rho_n^{3/2} (r^{1/2} + \log^{1/2}{n})}{\lambda_r^2} \times |S| \\
    &\precsim
    \frac{n^{3/2} \rho_n^{5/2} (r^{1/2} + \log^{1/2}{n})}{\lambda_r^2}
    \end{split}
\end{equation}
with high probability, and final inequality follows from Chernoff bound and the fact that $\hat{d}_{k}$ are independent Bernoulli random variables with success probabilities $d_k^2 = 2p_{ik}(1 - p_{ik}) \lesssim \rho_n$. For the term $U^{\top}(\hat{\mathcal{D}}^2 - \mathcal{D}^2)U$ we have
\[
    \mathrm{tr} \, (U^{\top}(\hat{\mathcal{D}}^2 - \mathcal{D}^2) U)
    =
    \sum_{k} \|U_k\|^2 \times \bigl((a_{ik} - a_{jk})^2 - d_{k}^2\bigr)
    .
\]
Now, $\mathbb{E}[(a_{ik} - a_{jk})^2] = d_{k}^2$ and $\mathrm{Var}[(a_{ik} - a_{jk})^2] = d_{k}^2 (1 - d_{k}^2) \precsim \rho_n$. Therefore, by Bernstein's inequality, we have
\begin{equation}
    \label{eq:bound_estimate2}
    \sum_{k} \|U_k\|^2 \bigl((a_{ik} - a_{jk})^2 - d_{k}^2\bigr) \precsim (n \rho_n)^{1/2} \|U\|_{2 \to \infty}^2 \log^{1/2}{n}
    \lesssim
    \frac{n^{3/2}\rho_n^{5/2} \log^{1/2}{n}}{\lambda_r^2}
\end{equation}
with high probability. Combining \cref{eq:proof_thm_test1,eq:bound_estimate1,eq:bound_estimate2} we obtain
\[
    \frac{\hat{\theta} - \mathrm{tr} \, M^*}{\|M^*\|_{F}}
    \lesssim
    \frac{(n \rho_n)^{3/2} (r^{1/2} + \log^{1/2}{n})}{r^{1/2} \lambda_r^2}
    \rightarrow
    0
\]
as claimed. For the term $\hat{\sigma}$, first note that
\begin{equation*}
    \begin{split}
    \bigl|\hat{\sigma} - \|M^*\|_{F}\bigr|
    &=
    \bigl| \|\hat{U}^{\top} \hat{\mathcal{D}}^2 \hat{U}\|_{F} - \|U^{\top} \mathcal{D}^2 U\|_{F}\bigr| \\
    &=
    \bigl|\|W^{\top} \hat{U}^{\top} \hat{\mathcal{D}}^2 \hat{U} W\|_{F} - \|U^{\top} \mathcal{D}^2 U\|_{F}\bigr| \\
    &\leq
    \|W^{\top} \hat{U}^{\top} \hat{\mathcal{D}}^2 \hat{U} W - U^{\top} \mathcal{D}^2 U\|_{F} \\
    &\leq 
    \|W^{\top} \hat{U}^{\top} \hat{\mathcal{D}}^2 \hat{U} W - U^{\top} \hat{\mathcal{D}}^2 U\|_{F} + \|U^{\top} (\hat{\mathcal{D}}^2 - \mathcal{D}^2) U\|_{F} \\
    &\leq
    \| \hat{\mathcal{D}} (\hat{U} W - U)\|_{F}^2 + 2 \|\hat{\mathcal{D}} (\hat{U} W - U)\|_{F} \cdot \|\hat{\mathcal{D}} U\|_{F} + \|U^{\top} (\hat{\mathcal{D}}^2 - \mathcal{D}^2) U\|_{F}
    .
    \end{split}
\end{equation*}
Following a similar argument to that for \cref{eq:bound_estimate1}, we have
\begin{equation}
   \label{eq:bound_estimate_3a}
   \begin{split}
   \|\hat{\mathcal{D}} (\hat{U} W - U)\|_{F}
   &\leq
   |S|^{1/2} \times \|\hat{U} W - U\|_{2 \to \infty} \\
   &\lesssim
   \frac{\rho_n^{1/2} (r^{1/2} + \log^{1/2}{n})}{|\lambda_r|} \times |S|^{1/2} \\
   &\lesssim
   \frac{n^{1/2} \rho_n (r^{1/2} + \log^{1/2}{n})}{|\lambda_r|}
   \end{split}
\end{equation}
with high probability. Furthermore, we also have
\[
    \|\hat{\mathcal{D}} U\|_{F}
    \leq
    |S|^{1/2} \times \|U\|_{2 \to \infty}
    \leq
    |S|^{1/2} \times n^{1/2} \rho_n |\lambda_r|^{-1} \lesssim n \rho_n^{3/2} |\lambda_r|^{-1}
\]
   with high probability. 
Combining the above bounds gives
\begin{equation}
    \label{eq:bound_estimate2b0}
    \begin{split}
    \|W^{\top} \hat{U}^{\top} \hat{\mathcal{D}}^2 \hat{U} W - U^{\top} \hat{\mathcal{D}}^2 U\|_{F}
    &\leq
    \| \hat{\mathcal{D}} (\hat{U} W - U)\|_{F}^2 + 2 \|\hat{\mathcal{D}} (\hat{U} W - U)\|_{F} \cdot \|\hat{\mathcal{D}} U\|_{F} \\
    &\lesssim
    \frac{n \rho_n^2 (r + \log n) + n^{3/2} \rho_n^{5/2} (r^{1/2} + \log^{1/2}{n})}{\lambda_r^2}
    \end{split}
\end{equation}
with high probability.

Next, we note that $\|U^{\top} (\hat{\mathcal{D}}^2 - \mathcal{D}^2) U\|_{F} \leq r \|U^{\top} (\hat{\mathcal{D}}^2 - \mathcal{D}^2) U\|_{\max}$. Furthermore, the $rs$-th entry of $U^{\top} (\hat{\mathcal{D}}^2 - \mathcal{D}^2) U$ can be written as
\[
    \sum_{k} u_{rk} u_{sk} [(a_{ik} - a_{jk})^2 - d_k^2]
    ,
\]
where $u_{rk}$ is the $rk$-th entry of $U$. The above is, conditioned on $P$, also a sum of independent mean zero random variables. Hence, by another application of Bernstein's inequality, we have
\begin{equation}
    \label{eq:bound_estimate2b1}
    \begin{split}
    \sum_{k} u_{rk} u_{sk} [(a_{ik} - a_{jk})^2 - d_k^2] &\lesssim (n \rho_n)^{1/2} \times \|U\|_{2 \to \infty}^2 \times \log^{1/2}{n} \\
    &\lesssim
    \frac{n^{3/2} \rho_n^{5/2} \log^{1/2}{n}}{\lambda_r^2} 
    \end{split}
\end{equation}
with high probability, and thus
\begin{equation}
    \label{eq:bound_estimate2b2}
    \begin{split}
    \|\hat{U}^{\top} \hat{\mathcal{D}}^2 \hat{U} - U^{\top} \mathcal{D}^2 U\|_{F}
    \lesssim
    \frac{n^{3/2} \rho_n^{5/2} r \log^{1/2}{n}}{\lambda_r^2}
    \end{split}
\end{equation}
with high probability. Combining \cref{eq:bound_estimate2b0,eq:bound_estimate2b2} we obtain
\begin{equation}
    \label{eq:bound_estimate2b3}
    \|\hat{U}^{\top} \hat{\mathcal{D}}^2 \hat{U} - U^{\top} \mathcal{D}^2 U\|_{F}
    \lesssim
    \frac{n^{3/2} \rho_n^{5/2} r \log^{1/2}{n}}{\lambda_r^2}
    .
\end{equation}
\cref{eq:bound_estimate2b3} then implies
\begin{equation}
     \label{eq:convergence_sigma_1}
     \begin{split}
     \frac{ \|W^{\top} \hat{U}^{\top} \hat{\mathcal{D}}^2 \hat{U} W - U^{\top} \mathcal{D}^2 U\|_{F}}{\|M_*\|_{F}}
     &\lesssim
     \frac{n^{3/2} \rho_n^{5/2} r \log^{1/2}{n}}{\lambda_r^2} \times \frac{1}{\rho_n \sqrt{r}} \\
     &\lesssim
     \frac{(n \rho_n)^{3/2} (r \log n)^{1/2}}{\lambda_r^2}
     \overset{\mathrm{p}}{\longrightarrow}
     0
     \end{split}
\end{equation}
as $n \rightarrow \infty$, provided that $r$ is chosen according to the criteria in \cref{thm:quadratic_test}. \cref{eq:convergence_sigma_1} also guarantees
$\hat{\sigma}/\|M_*\|_{F} \rightarrow 1$ in
probability as $n \rightarrow \infty$.

\section{Technical lemmas}
\label{sec:technical_lemmas}

We first recall a version of Bernstein's inequality for bounded random variables.
\begin{lemma}[Theorem~2.8.4 in \cite{vershynin2018high}]
    \label{lem:Bernstein}
    Let $X_1, X_2, \dots, X_m$ be independent mean zero random variables. Suppose there exists a constant $M > 0$ such that $|X_i| \leq M$ almost surely for all $1 \leq i \leq m$. Let $\sigma^2 = \sum_{i=1}^{m} \mathbb{E}[X_i^2]$. Then, for every $t > 0$, it holds that
    \begin{equation}
        \label{eq:bernstein}
        \mathbb{P}\Bigl(\bigl|\sum_{i=1}^{m} X_i \bigr| \geq t\Bigr)
        \leq
        2\exp\left(\frac{-t^2/2}{\sigma^2 + Mt/3}\right)
        .
    \end{equation}
    In particular, for any $\nu > 0$,
    \begin{equation}
        \label{eq:bernstein_restatement}
        \bigl|\sum_{i=1}^{m} X_i \bigr|
        \leq
        \sqrt{2 \nu} \sigma (\log n)^{1/2} + \frac{2}{3} \nu M \log n
    \end{equation}
    with probability at least $1 - 2n^{-\nu}$. 
\end{lemma}

The next lemma provides a high-probability bound for the spectral norm of $E$ and follows by adapting the results in \cite{bandeira2016sharp} to the setting of the current paper. 
\begin{lemma}
    \label{lem:bandeira_vanhandel}
    Assume the setting in \cref{thm:psd} or \cref{thm:general}. Then, for any $\alpha \geq 3$ and $t \geq 0$,
    \begin{equation}
        \label{eq:bandeira_explicit}
        \mathbb{P}\bigl(\|E\| \geq e^{2/\alpha}(2 \sqrt{2} \sigma + 14 \alpha \sqrt{\log n}) + t\bigr)
        \leq
        \exp(-t^2/2)
        ,
    \end{equation}
    where $\sigma^2 = \max_{i} \sum_{j} p_{ij}(1 - p_{ij})$. 
\end{lemma}
\begin{proof}[Proof of \cref{lem:bandeira_vanhandel}]
    First, note that the spectral norm of a matrix is a separately convex function of its entries. As the entries of $E$ are independent random variables bounded in magnitude $1$, \cite[Theorem~6.10]{boucheron13:_concen_inequal} yields
    \[
        \mathbb{P}( \|E\| \geq \mathbb{E}[\|E\|] + t) 
        \leq
        \exp(-t^2/2)
        .
    \]
    Next, let $A'$ be an independent copy of $A$, so that $E' = A' - P$ is an independent copy of $E$. Then, $\mathbb{E}[\|E\|] \leq \mathbb{E}[\|E - E'\|]$ by Jensen's inequality, where we had used the fact that $\mathbb{E}[E'] = 0$. Now, the upper triangular entries of $E - E'$ are independent, {\em symmetric} random variables (i.e., $E_{ij} - E_{ij}'$ has the same distribution as $E_{ij}' - E_{ij}$) bounded in absolute value by $1$. Therefore, by \cite[Corollary~3.6]{bandeira2016sharp}, for any $\alpha \geq 3$,
    \[
        \mathbb{E}[\|E - E'\|]
        \leq
        e^{2/\alpha}(2 \tilde{\sigma} + 14 \alpha \sqrt{\log n})
        ,
    \]
    where $\tilde{\sigma}^2 = \max_{i} \sum_{j} \mathrm{Var}[(E_{ij} - E_{ij}')]$. Straightforward calculations yield
    \[
        \tilde{\sigma}^2
        =
        \max_{i} \sum_{j} 2p_{ij}(1 - p_{ij})(1 - 2p_{ij}(1 - p_{ij}))
        \leq
        2 \max_{i} \sum_{j} p_{ij}(1 - p_{ij})
        .
    \]
    \cref{eq:bandeira_explicit} follows directly from combining the above bounds. This completes the proof of \cref{lem:bandeira_vanhandel}.
\end{proof}

Taking $\alpha = 4$ in \cref{lem:bandeira_vanhandel} and applying $\sigma^2 \leq n \rho_n$, it holds for any $\nu \geq 0$ that
\begin{equation}
    \label{eq:norm_E}
    \mathbb{P}(\|E\| \geq 2 \sqrt{2e n \rho_n} + (56 \sqrt{e} + \sqrt{2 \nu}) \sqrt{\log n})
    \leq
    n^{-\nu}
    .
\end{equation}
Next, we present a collection of lemmas for bounding the quantities $\psi_1, \psi_2$ and $\psi_3^{(k)}$ in the proofs of \cref{thm:psd} through \cref{thm:generalU_ULambda}.

\begin{lemma}
    \label{lem:technical2a}
    Assume the setting in \cref{thm:psd} or \cref{thm:general}. Let $U$ be a $n \times r$ matrix not depending on $E$ with $U^{\top} U = I$. Then,
    \begin{gather}
        \|U^{\top} E U\|
        \leq
        4 \sqrt{\rho_n \vartheta(\nu,r,n)} + \frac{8}{3} \|U\|_{2 \to \infty}^2 \vartheta(\nu,r,n), \\
        \|U^{\top} E \hat{U} \|
        \leq
        4 \sqrt{\rho_n \vartheta(\nu,r,n)} + \frac{8}{3} \|U\|_{2 \to \infty}^2 \vartheta(\nu,r,n) + \frac{2\|E\|^2}{\delta_r}
        ,
    \end{gather}
    with probability at least $1 - 2n^{-\nu}$, where $\vartheta(\nu,r,n) = \nu \log n + r \log 9$.  
\end{lemma}

\begin{proof}[Proof of \cref{lem:technical2a}]
    First, we shall bound $\|U^{\top} E U\|$ using an $\epsilon$-net argument. By the definition of the operator norm,
    \[
        \|U^{\top} E U\|
        =
        \max_{x \in \mathbb{R}^{p} \colon \|x\| = 1 } |x^{\top} U^{\top} E U x|
        .
    \]
    Now, fix $x \in \mathbb{R}^{r}$, $\|x\|=1$, and let $\xi = U x$, so
    \[
        x^{\top} U^{\top} E U x
        =
        2 \sum_{i < j} e_{ij} \xi_i \xi_j + \sum_{i} e_{ii} \xi_i^2
        .
    \]
    Letting $u_i$ denote the $i$-th row of $U$, we have $\max_{i} |\xi_i| \leq \max_{i} \|u_i\| \times \|x\| \leq \|U\|_{2 \to \infty}$. In particular, $\sum_{i < j} e_{ij} \xi_i \xi_j$ is a sum of independent mean zero random variables satisfying
    \[
        2 \max_{i,j} |e_{ij} \xi_i \xi_j|
        \leq
        2
        \|U\|_{2 \to \infty}^2
        .
    \]
    Define $\sigma^2 = \sum_{i < j} \mathrm{Var}[2 e_{ij}  \xi_i \xi_j] + \sum_{i} \mathrm{Var}[e_{ii} \xi_i^2]$. Since $\|\xi\| = \|U x\| = \|x\| = 1$, we have
    \begin{equation*}
        \begin{split}
        \sigma^2
        &=
        4
        \sum_{i < j} p_{ij}(1 - p_{ij}) \xi_i^2 \xi_j^{2} + \sum_{i} p_{ii} (1 - p_{ii}) \xi_i^4
        \leq
        2 \sum_{i} \sum_{j} p_{ij} (1 - p_{ij}) \xi_i^2 \xi_j^2 \leq 2 \rho_n
        .
    \end{split}
    \end{equation*}
    Applying Bernstein's inequality therefore yields
    \[
        \mathbb{P}\Bigl(\bigl|x^{\top} U^{\top} E U x\bigr| \geq t\Bigr)
        \leq
        2 \exp\Bigl(\frac{-t^2}{2 \rho_n + 2 \|U\|_{2 \to \infty}^2 t/3}\Bigr)
        .
    \] 
    Now, let $\mathcal{S}_{1/4}$ be a minimal $1/4$-net for the unit ball in $\mathbb{R}^{r}$ with $\epsilon \in (0,1)$. Thus $|\mathcal{S}_{\epsilon}| \leq 9^{r}$, for example, by \cite[Corollary~4.2.13]{vershynin2018high}. Next, let $x_*$ satisfy $|x_*^{\top} U^{\top} E U x_*| = \|U^{\top} E U\|$ and $\|x_*\| = 1$, so there exists a vector $z \in \mathcal{S}_{1/4}$ such that $\|z - x_*\| \leq 1/4$. Furthermore,
    \begin{equation*}
        \begin{split}
        |x_*^{\top} U^{\top} E U x_*|
        &=
        |(x_* - z)^{\top} U^{\top} E U x_* + z^{\top} U^{\top} E U (x_* - z) + z^{\top} U^{\top} E U z| \\
        &\leq
        |z^{\top} U^{\top} E U z| + \|z - x_*\| \times (\|U^{\top} E U x_*\| + \|U^{\top} E U z\|) \\
        &\leq
        |z^{\top} U^{\top} E U z| + \frac{1}{2}\|U^{\top} E U\|
        .
        \end{split}
    \end{equation*}
    Hence, for all $t > 0$,
    \begin{equation*}
        \begin{split}
        \mathbb{P}[\|U^{\top} E U\| \geq 2t]
        &\leq
        \mathbb{P}\Bigl[\sup_{z \in \mathcal{S}_{1/2}} |z^{\top} U^{\top} E U z| \geq t \Bigr] \\
        &\leq
        2 |\mathcal{S}_{1/4}| \exp\Bigl(\frac{-t^2}{4 \rho_n + 4 \|U\|_{2 \to \infty}^2 t/3}\Bigr) \\
        &\leq
        2 \exp\Bigl(r \log (9) - \frac{t^2}{4 \rho_n + 4 \|U\|_{2 \to \infty}^2 t/3}\Bigr)
        .
        \end{split}
    \end{equation*}
    

    Define $\vartheta(\nu,r,n) = \nu \log n + r \log 9$, and let
    \begin{equation}
        \label{eq:t_*}
        t_*
        =
        2 \sqrt{\rho_n \vartheta(\nu,r,n)} + \frac{4}{3} \|U\|_{2 \to \infty}^2 \vartheta(\nu,r,n)
        .
    \end{equation}
    We then have
    \begin{equation*}
        \begin{split}
        \mathbb{P}\bigl(\|U^{\top} E U \| \geq 2 t_*\bigr)
        \leq
        2n^{-\nu}
        ,
        \end{split}
    \end{equation*}
    which yields the stated bound for $\|U^{\top} E U \|$ as desired.
    
    The stated bound for $U^{\top} E \hat{U}$ follows directly from the triangle inequality and the Davis--Kahan theorem, namely
    \begin{equation*}
        \begin{split}
        \|U^{\top} E \hat{U} \|
        &\leq
        \|U^{\top} E U U^{\top} \hat{U} \|
        +
        \|U^{\top} E (I - U U^{\top}) \hat{U} \| \\
        &\leq
        \|U^{\top} E U \| + \|E\| \times \|(I - UU^{\top}) \hat{U}\| \\
        &\leq
        \|U^{\top} E U \| + \frac{2 \|E\|^2}{\delta_r}
        .
        \end{split}
    \end{equation*}
    This completes the proof of \cref{lem:technical2a}.
\end{proof}


\begin{lemma}
    \label{lem:approximate_commute}
    Assume the setting in \cref{thm:general}. Let $T = |\Lambda|^{1/2} J U^{\top} \hat{U} - U^{\top} \hat{U} |\hat{\Lambda}|^{1/2} J$.
    Then,
    \begin{gather}
        \label{eq:approximate_commute}
        \|T\|
        \leq
        |\lambda_r|^{-1/2}\Bigl(8 \sqrt{\rho_n \vartheta(\nu,r,n)} + \frac{16}{3} \|U\|_{2 \to \infty}^2 \vartheta(\nu,r,n) + \frac{ 4 \|E\|^2}{\delta_r}\Bigr)
    \end{gather}
    with probability at least $1 - 2n^{-\nu}$. Furthermore, under the setting in \cref{thm:psd} where $\lambda_1 \geq \lambda_2 \geq \dots \geq \lambda_r > 0$ then, for $T = \Lambda^{1/2} U^{\top} \hat{U} - U^{\top} \hat{U} \hat{\Lambda}^{1/2}$, \cref{eq:approximate_commute} can be improved to
    \begin{gather}
        \label{eq:approximate_commute_psd}
        \|T\|
        \leq
        \lambda_r^{-1/2}\Bigl(4 \sqrt{\rho_n \vartheta(\nu,r,n)} + \frac{8}{3} \|U\|_{2 \to \infty}^2 \vartheta(\nu,r,n) + \frac{ 2 \|E\|^2}{\delta_r}\Bigr)
    \end{gather}
    with probability at least $1 - 2n^{-\nu}$.
\end{lemma}

\begin{proof}[Proof of \cref{lem:approximate_commute}]
    First, we collect several observations. Suppose that $t_{ij}$, the $ij$-th entry of $T$, is associated with $\lambda_i \geq 0$ and $\hat{\lambda}_j \geq 0$. Then,
    \begin{equation*}
        t_{ij}
        =
        u_{i}^{\top} \hat{u}_j (\lambda_i^{1/2} - \hat{\lambda}_j^{1/2})
        =
        \frac{u_i^{\top} \hat{u}_j (\lambda_i - \hat{\lambda}_j)}{\hat{\lambda}_i^{1/2} + \hat{\lambda}_j^{1/2}} = \frac{u_i^{\top}(A - P) \hat{u}_j}{\hat{\lambda}_i^{1/2} + \hat{\lambda}_j^{1/2}}
        =
        \frac{u_i^{\top}(A - P) \hat{u}_j}{|\lambda_i|^{1/2} + |\hat{\lambda}_j|^{1/2}}
        .
    \end{equation*}
    Similarly, if $t_{ij}$ is associated with $\lambda_i \leq 0$ and $\hat{\lambda}_j \leq 0$. then
    \begin{equation*}
        t_{ij}
        =
        u_{i}^{\top} \hat{u}_j (|\hat{\lambda}_j|^{1/2} - |\lambda_i|^{1/2})
        =
        \frac{u_i^{\top} \hat{u}_j (|\hat{\lambda}_j| - |\lambda_i|)}{|\lambda_i|^{1/2} + |\hat{\lambda}_j|^{1/2}}
        =
        \frac{u_i^{\top}(A - P) \hat{u}_j}{|\lambda_i|^{1/2} + |\hat{\lambda}_j|^{1/2}}
        .
    \end{equation*}
    Next, if $t_{ij}$ is associated with $\lambda_i \geq 0$ and $\hat{\lambda}_j \leq 0$, then
    \begin{equation*}
        \begin{split}
        t_{ij}
        &=
        u_{i}^{\top} \hat{u}_j (\lambda_i^{1/2} + |\hat{\lambda}_j|^{1/2}) \\
        &=
        u_i^{\top} \hat{u}_j \Bigl(\frac{|\lambda_i| + |\hat{\lambda}_j| + 2 |\lambda_i|^{1/2} \cdot |\hat{\lambda}_j|^{1/2}}{ |\lambda_i|^{1/2} + |\hat{\lambda}_j|^{1/2}} \Bigr) \\
        &=
        \frac{-u_i^{\top}(A - P) \hat{u}_j}{|\lambda_i|^{1/2} + |\hat{\lambda}_j|^{1/2}} \Bigl(1 + \frac{2 |\lambda_i|^{1/2} \cdot |\hat{\lambda}_j|^{1/2}}{|\lambda_i| + |\hat{\lambda}_j|}\Bigr)
        .
        \end{split}
    \end{equation*}
    Finally, if $t_{ij}$ is associated with $\lambda_i \leq 0$ and $\hat{\lambda}_j \geq 0$, then
    \begin{equation*}
        \begin{split}
        t_{ij}
        &=
        u_{i}^{\top} \hat{u}_j (-|\lambda_i|^{1/2} - \hat{\lambda}_j^{1/2}) \\
        &=
        -u_i^{\top} \hat{u}_j \Bigl(\frac{|\lambda_i| + |\hat{\lambda}_j| + 2 |\lambda_i|^{1/2} \cdot |\hat{\lambda}_j|^{1/2}}{ |\lambda_i|^{1/2} + |\hat{\lambda}_j|^{1/2}} \Bigr) \\
        &=
        \frac{u_i^{\top}(A - P) \hat{u}_j}{|\lambda_i|^{1/2} + |\hat{\lambda}_j|^{1/2}} \Bigl(1 + \frac{2 |\lambda_i|^{1/2} \cdot |\hat{\lambda}_j|^{1/2}}{|\lambda_i| + |\hat{\lambda}_j|}\Bigr)
        .
        \end{split}
    \end{equation*}
    Now, let $S$ denote the $r \times r$ matrix defined element-wise as $s_{ij} = 1/(|\lambda_i|^{1/2} + |\hat{\lambda}_j|^{1/2})$, and let $H$ denote the $r \times r$ matrix with entries $h_{ij} = 2 (|\lambda_i|^{1/2} \cdot |\hat{\lambda}_j|^{1/2})/(|\lambda_i| + |\hat{\lambda}_j|)$. 
    As before, we write $E = (A - P)$.
    Combining the above identities for $T$ yields
    \begin{gather*}
        T
        =
        (J U^{\top}E \hat{U} J) \circ S + \mathcal{P}^{(1)} \circ S \circ H,
        \qquad
        \text{where}
        \quad
        \mathcal{P}^{(1)}
        =
        \begin{bmatrix}
            0 & -U_{+}^{\top} E \hat{U}_{-} \\ -U_{-}^{\top}E \hat{U}_{+} & 0
        \end{bmatrix}
        ,
    \end{gather*}
    and $\circ$ denote the Hadamard (elementwise) matrix product. Next, define
    \[
        \mathcal{P}^{(2)}
        =
        \begin{bmatrix}
            0 & \,\, & \mathcal{P}^{(1)} \\
            0 & \,\, & 0
        \end{bmatrix}
        ,
    \]
    and let $\gamma = (|\lambda_1^{\downarrow}|, |\lambda_2^{\downarrow}|, \dots, |\lambda_r^{\downarrow}|, |\hat{\lambda}_1^{\downarrow}|, |\hat{\lambda}_2^{\downarrow}|, \dots,|\hat{\lambda}_r^{\downarrow}|)$. Here, $\lambda_i^{\downarrow}$ is a rearrangement of $\lambda_1, \lambda_2, \dots \lambda_r$ in decreasing value (recall that the $\lambda_i$ are ordered in decreasing modulus), and similarly for $\hat{\lambda}_i^{\downarrow}$.
    Letting $\tilde{S}$ be the $2r \times 2r$ matrix with entries $\tilde{s}_{ij} = 1/(\gamma_i + \gamma_j)$ and $\tilde{H}$ be the $2r \times 2r$ matrix with entries $\tilde{h}_{ij} = 2 \gamma_i^{1/2} \gamma_j^{1/2}/(\gamma_i + \gamma_j)$, we have
    \[
        \|\mathcal{P}^{(1)} \circ S \circ H\|
        =
        \|\mathcal{P}^{(2)} \circ \tilde{S} \circ \tilde{H}\|
        .
    \]
    Here, $\tilde{H}$ is the Hadamard product of a positive semidefinite matrix with entries $2 \gamma_i^{1/2} \gamma_j^{1/2}$ and a {\em Cauchy matrix} with entries $1/(\gamma_i + \gamma_j)$. Hence, by the Schur product formula \cite[Theorem~5.2.1]{horn85:_toppics_matrix_analy}), $\tilde{H}$ is positive semidefinite. It follows from \cite[Theorem~5.5.18]{horn85:_toppics_matrix_analy} that
    \[
        \|\mathcal{P}^{(2)} \circ \tilde{S} \circ \tilde{H}\|
        \leq
        \|\mathcal{P}^{(2)} \circ \tilde{S}\| \times \max_{1 \leq i \leq 2r} \tilde{h}_{ii}
        =
        \|\mathcal{P}^{(2)} \circ \tilde{S}\|
        =
        \|\mathcal{P}^{(1)} \circ S \|
        .
    \]
    Here, $S$ is also a Cauchy matrix and hence, by \cite[Eq.~(17)]{horn_norm_bounds}, we have 
    \begin{gather*}
        \|\mathcal{P}^{(1)} \circ S\|
        \leq
        \frac{\|\mathcal{P}^{(1)}\|}{\min_{i} |\lambda_i|^{1/2} + \min_{j} |\hat{\lambda}_j|^{1/2}} \leq |\lambda_r|^{-1/2} \|\mathcal{P}^{(1)}\| \leq |\lambda_r|^{-1/2} \|U^{\top} E \hat{U}\|, \\
        \|(J U^{\top}E \hat{U} J) \circ S\|
        \leq
        \frac{\|U^{\top}E \hat{U}\|}{\min_{i} |\lambda_i|^{1/2} + \min_{j} |\hat{\lambda}_j|^{1/2}}
        \leq
        |\lambda_r|^{-1/2} \|U^{\top}E \hat{U}\|
        .
    \end{gather*}
    In summary, we have $\|T\| \leq 2|\lambda_r|^{-1/2} \|U^{\top} E \hat{U}\|$. From \cref{lem:technical2a},
    \[
        \|U^{\top} E U\| 
        \leq
        4 \sqrt{\rho_n \vartheta(\nu,r,n)} + \frac{8}{3} \|U\|_{2 \to \infty}^2 \vartheta(\nu,r,n)
    \]
    with probability at least $1 - 2n^{-\nu}$, where $\vartheta(\nu,r,n) = \nu \log n + r \log 9$. 
    Meanwhile, by the Davis--Kahan theorem,
    \[
        \|U^{\top} E (I - U U^{\top}) \hat{U}\|
        \leq
        \|E\| \times \|(I - U U^{\top} \hat{U}\|
        \leq
        \frac{2 \|E\|^{2}}{\delta_r}
        .
    \]
    Combining the above bounds yields
    \begin{equation}
        \begin{split}
        \label{eq:T_spectral}
        \|T\| 
        \leq
        |\lambda_r|^{-1/2}\Bigl(8 \sqrt{\rho_n \vartheta(\nu,r,n)} + \frac{16}{3} \|U\|_{2 \to \infty}^2 \vartheta(\nu,r,n)
        +
        \frac{ 4 \|E\|^2}{\delta_r}\Bigr)
    \end{split}
    \end{equation}
    with probability at least $1 - 2n^{-\nu}$. Finally, if $\lambda_1 \geq \lambda_2 \geq \dots \geq \lambda_{r} \geq 0$ then $r_{-} = 0, U = U_{+}, \mathcal{P}^{(1)} = 0$, and \cref{eq:T_spectral} can be improved to
    \begin{equation}
        \label{eq:term2_variant}
        \|T\|
        \leq
        \lambda_r^{-1/2}\Bigl(4 \sqrt{\rho_n \vartheta(\nu,r,n)} + \frac{8}{3} \|U\|_{2 \to \infty}^2 \vartheta(\nu,r,n)
        +
        \frac{ 2 \|E\|^2}{\delta_r}\Bigr)
    \end{equation}
    with probability at least $1 - 2n^{-\nu}$. This completes the proof of \cref{lem:approximate_commute}.
\end{proof}

\begin{lemma}
    \label{lem:technical4_new}
    Assume the setting in \cref{thm:psd} or \cref{thm:general}. Then for any $k \in [n]$,
    \[
        \begin{split}
        \|(I - UU^{\top}) P^{k} E U\|_{2 \to \infty}
        &=
        \|U_{\perp} \Lambda_{\perp}^{k} U_{\perp}^{\top} E U \|_{2 \to \infty} \\
        &\leq
        \|U_{\perp} \Lambda_{\perp}^{k}\|_{2 \to \infty}\Bigl(\sqrt{8 \rho_n \vartheta(\nu,r,n)} + \frac{8}{3} \|U\|_{2 \to \infty} \vartheta(\nu,r,n)\Bigr)
        \end{split} 
    \]
    with probability at least $1 - n^{-(\nu-1)}$, where $\vartheta(\nu,r,n) = \nu \log n + r \log 5$.  
\end{lemma}

\begin{proof}[Proof of \cref{lem:technical4_new}]
    Let $e_i$ denote the $i$-th elementary basis vector in $\mathbb{R}^{n}$, and let $\xi_i$ denote the $i$-th row of $U_{\perp} \Lambda_{\perp}^{k} U_{\perp}^{\top}$. We have
    \[
        \|U_{\perp} \Lambda_{\perp}^{k} U_{\perp}^{\top} E U\|_{2 \to \infty}
        =
        \max_{i} \|e_i^{\top} U_{\perp} \Lambda_{\perp}^{k} U_{\perp}^{\top}  E U\|
        =
        \max_{i} \|\xi_i^{\top} E U\|
        .
    \]
    We shall bound $\|\xi_i^{\top} E U\|$ using a standard $\epsilon$-net argument. To that end, for any $i$,
    \[
        \|\xi_i^{\top} E U\|
        =
        \sup_{v \in \mathcal{S}} \xi_i^{\top} E U v
        ,
    \]
    where $\mathcal{S} = \{v \in \mathbb{R}^{r} \colon \|v\| = 1\}$. Let $\mathcal{S}_{1/2}$ be a $1/2$-cover of $\mathcal{S}$ with respect to the $\ell_2$ norm. Then, for any $v \in \mathcal{S}$ and $w \in \mathcal{S}_{1/2}$ with $\|v - w\| \leq 1/2$, we have
    \[
        |\xi_i^{\top} E U (v - w)|
        \leq
        \|\xi_i^{\top} E U\| \cdot \|v - w\|
        \leq
        \frac{1}{2} \|\xi_i^{\top} E U \|,
    \]
    and hence
    \[
        \|\xi_i^{\top} E U\|
        =
        \sup_{v \in \mathcal{S}} \xi_i^{\top} E U v
        \leq
        \sup_{w \in \mathcal{S}_{1/2}} \xi_i^{\top} E U w + \frac{1}{2}  \|\xi_i^{\top} E U\|
        .
    \]
    Rearranging the above inequality gives
    \[
        \|\xi_i^{\top} E U \|
        \leq
        2 \sup_{v \in \mathcal{S}_{1/2}} \xi_i^{\top} E U v
        .
    \]
    Here, for any $w \in \mathcal{S}_{1/2}$, Bernstein's inequality guarantees
    \begin{equation}
        \label{eq:bernstein_lemm1}
        \mathbb{P}(\xi_i^{\top} E U w \geq t)
        \leq
        \exp\Bigl(\frac{-t^2}{2\sigma^2 + 2Mt/3}\Bigr)
        ,
    \end{equation}
    for $M = \|Uw\|_{\infty} \cdot \|\xi_i\|_{\infty}$ and $\sigma^2 \leq 2 \rho_n \|\xi_i\|^2 \cdot \|Uw\|^2 = 2 \rho_n \|\xi_i\|^2$. 

    Observe that $\|Uw\|_{\infty} \leq \|U\|_{2 \to \infty}$ since $\|w\| = 1$. Let $\vartheta(\nu,r,n) = \nu \log n + r \log 5$ and
    \[
        t_*
        =
        \frac{4M}{3} \vartheta(\nu,r,n) + \sqrt{2} \rho_n^{1/2} \|\xi_i\| \sqrt{\vartheta(\nu,r,n)}
        .
    \]
    Consequently,
    \[
        \frac{t_*^2}{2 \sigma^2 + 2Mt_*/3}
        \geq
        \vartheta(\nu,r,n)
        ,
    \]
    and hence
    \[
        P(\xi_i^{\top} E U w \geq t_*)
        \leq
        \exp(-\vartheta(\nu,r,n)) = n^{-\nu} \times 5^{-r}
        .
    \]
    In the present setting, $|\mathcal{S}_{1/2}| \leq 5^{r}$ holds by a volumetric argument (e.g., see \cite[Corollary~4.2.13]{vershynin2018high}). By taking a union bound over all $w \in \mathcal{S}_{1/2}$, we therefore obtain
    \[
        P(\sup_{w \in \mathcal{S}_{1/2}} \xi_i^{\top} E U w \geq t_*)
        \leq
        n^{-\nu}
        .
    \]
    In summary, we have $\|\xi_i^{\top} E U\| \leq 2t_*$ with probability at least $1 - n^{-\nu}$. By taking another union bound over all $i \in [n]$, we obtain
    \[
        \begin{split}
        \|U_{\perp} \Lambda_{\perp}^{k} U_{\perp}^{\top} E U\|_{2 \to \infty}
        &\leq
        \|U_{\perp} \Lambda_{\perp}^{k} U_{\perp}^{\top} \|_{2 \to \infty}\Bigl(\frac{8}{3} \|U\|_{2 \to \infty} \vartheta(\nu,r,n) + \sqrt{8 \rho_n\vartheta(\nu,r,n)}\Bigr) \\
        &\leq
        \|U_{\perp} \Lambda_{\perp}^{k}\|_{2 \to \infty}\Bigl(\sqrt{8 \rho_n \vartheta(\nu,r,n)} + \frac{8}{3} \|U\|_{2 \to \infty} \vartheta(\nu,r,n)\Bigr)
        \end{split}
    \]
    with probability at least $1 - n^{-(\nu-1)}$, where the second inequality follows from \cref{eq:2toinf_submultiplicative}. This completes the proof of \cref{lem:technical4_new}.
\end{proof}

Lastly, we present a technical lemma for bounding $\|EU\|_{2 \to \infty}$. The result here is a slight improvement compared to a direct application of matrix Bernstein's inequality. See also
\cite[Lemma~3.3]{lei2019unified}. 
\begin{lemma}
    \label{lem:ex_2inf}
    Consider the setting in \cref{thm:psd} through \cref{thm:generalU_ULambda}. Let $V$ be a $n \times r$ matrix independent of $E$. Denote by $e_i$ the $i$-th elementary basis vector in $\mathbb{R}^{n}$. Then, for any $i \in [n]$,
    \begin{equation}
        \label{eq:ex_2infi}
        \|e_i^{\top} E V\| \leq 2 \sqrt{2 \rho_n \vartheta(\nu,r,n)} \|V\| + 
        \tfrac{8}{3} \|V\|_{2 \to \infty} \vartheta(\nu,r,n)
    \end{equation}
    with probability at least $1 - n^{-\nu}$, where $\vartheta(\nu,r,n) = \nu \log n + r \log 5$. \cref{eq:ex_2infi} then implies
    \begin{equation}
        \label{eq:ex_2inf}
        \|E V\|_{2 \to \infty} \leq 2\sqrt{2 \rho_n \vartheta(\nu,r,n)} \|V\| + 
        \tfrac{8}{3} \|V\|_{2 \to \infty} \vartheta(\nu,r,n)
    \end{equation} 
    with probability at least $1 - n^{-(\nu - 1)}$. 
\end{lemma}

\begin{proof}[Proof of \cref{lem:ex_2inf}]
    We follow the same $\epsilon$-net argument as in the proof of \cref{lem:technical4_new}. More specifically, let $\mathcal{S} = \{v \in \mathbb{R}^{r} \colon \|v\| = 1\}$ and $\mathcal{S}_{1/2}$ be a {\em minimal} $1/2$ cover of $\mathcal{S}$. For any $i \in [n]$, we have
    \[
        \|e_i^{\top} E V\|
        \leq
        2 \max_{i} \sup_{w \in \mathcal{S}_{1/2}}e_i^{\top} E U w
        .
    \]
    Next, by Bernstein's inequality, for any fixed but arbitrary $w \in \mathcal{S}_{1/2}$, we have
    \[
        \mathbb{P}(e_i^{\top} E U w \geq t)
        \leq
        \exp\Bigl(\frac{-t^2}{2 \sigma^2 + 2Mt/3}\Bigr)
        ,
    \]
    where $M = \|Vw\|_{\infty} \leq \|V\|_{2 \to \infty}$ and $\sigma^2 \leq 2 \rho_n \|Vw\|^2 \leq 2 \rho_n \|V\|^2$. Now, define
    \[
        t_*
        =
        \tfrac{4}{3}\|V\|_{2 \to \infty} \vartheta(\nu,r,n) + \sqrt{2 \rho_n \vartheta(\nu,r,n)} \|V\|
        .
    \]
    Hence,
    \[
        \mathbb{P}(e_i^{\top} E V w \geq t_*) \leq n^{-\nu} 5^{-r}
        .
    \]
    Now, choose some $i \in [n]$. As $|\mathcal{S}_{1/2}| \leq 5^{r}$, taking a union bound over $\mathcal{S}_{1/2}$ we obtain
    \[
        \mathbb{P}(\|e_i^{\top} E V\| \geq 2 t_*)
        \leq
        n^{-\nu}
    \]
    as claimed in \cref{eq:ex_2infi}.
    Finally, taking a union over all $i \in [n]$, we obtain
    \[
        P(\|EV\|_{2 \to \infty} \geq 2t_*) = P(\max_{i \in [n]} \|e_i^{\top} E V\|)
        \leq
        n^{-(\nu-1)}
        ,
    \]
    as claimed in \cref{eq:ex_2inf}. This completes the proof of \cref{lem:ex_2inf}.
\end{proof}

\begin{lemma}
\label{lem:eta_bound}
Consider the setting in \cref{thm:psd} through \cref{thm:generalU_ULambda}. Define
\[ \sigma_*^2 = \max\Bigl\{ n \rho_n \|U_{\perp} \Lambda_{\perp}\|_{2 \to \infty}^2, n \rho_n \|U\|_{2 \to \infty}^2 \lambda_{r+1}^2 \Bigr\}\]
We then have
\begin{equation}
\label{eq:eta_bound}
\|\Lambda_{\perp} U_{\perp}^{\top} E U \| \leq  2 \sigma_* \sqrt{2 \log n} + \frac{4}{3} \|U\|_{2 \to \infty} \cdot \|U_{\perp} \Lambda_{\perp}\|_{2 \to \infty} \cdot \log n
\end{equation}
with probability at least $1 - n^{-(\nu - 1)}$.

Similarly, if $P$ is positive semidefinite then
\begin{equation}
\label{eq:eta_bound_psd}
\|\Lambda_{\perp}^{1/2} U_{\perp}^{\top} E U \| \leq 2 \tilde{\sigma}_* \sqrt{2 \log n} + \frac{4}{3} \|U\|_{2 \to \infty} \cdot \|U_{\perp} \Lambda_{\perp}^{1/2}\|_{2 \to \infty} \cdot \log n
\end{equation}
with probability at least $1 - n^{-(\nu - 1)}$, where now
\[ \tilde{\sigma}_*^2 = \max \Bigl\{ n \rho_n \|U_{\perp} \Lambda_{\perp}^{1/2}\|_{2 \to \infty}^2, n \rho_n \|U\|_{2 \to \infty}^2 \lambda_{r+1}\Bigr\} \]
\end{lemma}

\begin{proof}
    Let $v_{i}$ denote the $i$-th column of $\Lambda_{\perp} U_{\perp}^{\top}$ and $u_i$ denote the $i$-th column of $U^{\top}$. Then
    \[ \Lambda_{\perp} U_{\perp}^{\top} E U = \sum_{i=1}^{n} \sum_{j=1}^{n} e_{ij} v_{i} u_j^{\top} = \sum_{i=1}^{n} e_{ii} v_i u_i^{\top} + \sum_{i \not = j}
    e_{ij} (v_i u_j^{\top} + v_{j} u_i^{\top})
    \]
    which is a sum of independent mean $0$ random matrices. 

    Let $X_{ii} = e_{ij} u_i v_i^{\top}$ and $X_{ij} = e_{ij} (v_{i} u_{j}^{\top} + v_{j} u_{i}^{\top})$. 
    Then \[\|X_{ij}\| \leq L:= 2 \max_{ij} \|u_i\| \cdot \|v_j\| \leq 2 \|U_{\perp} \Lambda_{\perp}\|_{2 \to \infty} \times \|U\|_{2 \to \infty}\]
    almost surely. 
    Furthermore, for any matrices $M$ and $N$ we have $(M + N)^{\top}(M + N) \preceq 2 M^{\top} M + 2 N^{\top} N$ where $\preceq$ denote the Lowner ordering for symmetric matrices. Hence
    \[ \begin{split} \sum_{i} \mathbb{E}[X_{ii}^{\top} X_{ii}] + \sum_{j \not = i} \mathbb{E}[X_{ij}^{\top} X_{ij}] &\preceq \sum_{i} \sum_{j} 2 \mathrm{Var}[e_{ij}] (u_{j} v_i^{\top} v_i u_{j}^{\top} + u_{i} v_{j}^{\top} v_{j} u_{i}^{\top}) \\ & \preceq 2 \rho_n \sum_{i} \sum_{j} u_{i} u_{i}^{\top} \|v_j\|^2 + u_{j} u_{j}^{\top} \|v_i\|^2 \\ 
    & \preceq 4 n \rho_n \|U_{\perp} \Lambda_{\perp}\|_{2 \to \infty}^2 \sum_{i} u_{i} u_i^{\top} \preceq 4 n \rho_n \|U_{\perp} \Lambda_{\perp}\|_{2 \to \infty}^2 \times I \end{split} \]
    Similarly, we also have
    \[ \begin{split} \sum_{i} \mathbb{E}[X_{ii} X_{ii}^{\top}] + \sum_{j \not = i} \mathbb{E}[X_{ij} X_{ij}^{\top}] &\preceq \sum_{i} \sum_{j} 2 \mathrm{Var}[e_{ij}] (v_{i} u_{j}^{\top} u_{j} v_i^{\top} + v_j u_{i}^{\top} u_{i} v_{j}^{\top}) \\ & \preceq 2 \rho_n \sum_{i} \sum_{j} v_{i} v_{i}^{\top} \|u_j\|^2 + v_{j} v_{j}^{\top} \|u_i\|^2 \\ 
    & \preceq 4 n \rho_n \|U\|_{2 \to \infty}^2 \sum_{i} v_j v_j^{\top} \preceq 4 n \rho_n \|U\|_{2 \to \infty}^2 \times \Lambda_{\perp}^2 \end{split} \]
    Define
    \[ \sigma_*^2 = \max\{ n \rho_n \|U_{\perp} \Lambda_{\perp}\|_{2 \to \infty}^2, n \rho_n \|U\|_{2 \to \infty}^2 \lambda_{r+1}^2\}. \]
    Then by a matrix Bernstein's inequality \cite[Theorem~1.6]{tropp} we have
    \[ \mathbb{P}( \|\Lambda_{\perp} U_{\perp}^{\top} E U\| \geq t) \leq n \exp\Bigl(-\frac{t^2}{8 \sigma_*^2 + 2 Lt/3}\Bigr). \]
    Taking $t = 2 \sqrt{2} \sigma_{*} \sqrt{\log n} + \frac{2}{3} L \log n$, we obtain the bound in \cref{eq:eta_bound}. Derivation of \cref{eq:eta_bound_psd} follows the same argument and is thus omitted. 
\end{proof}

\end{document}